\documentclass[12pt]{amsart}
\usepackage{graphicx}
\usepackage{amsmath}
\usepackage{amssymb}
\usepackage{amsfonts}
\usepackage{amsthm}
\usepackage{cancel}
\usepackage[all]{xy}
\usepackage[neveradjust]{paralist}
\usepackage{enumitem}
\usepackage{float}
\usepackage{multicol}
\usepackage{mathrsfs}
\usepackage{mathdots}
\usepackage{mathtools}
\usepackage{todonotes}
\usepackage[normalem]{ulem} 
\usepackage{caption} 
\usepackage{bbm}
\usepackage[pagebackref,colorlinks=true,linkcolor=blue,citecolor=blue]{hyperref}
\usepackage{tikz-cd}
\usepackage{tikz}
\tikzcdset{scale cd/.style={every label/.append style={scale=#1},
    cells={nodes={scale=#1}}}}
\usetikzlibrary{matrix, positioning}

\newlength{\wdth}

\usepackage{tabularray}

\newcommand{\Z}{\mathbb{Z}}

\newcommand{\R}{\mathbb{R}}

\newcommand{\Eseg}{{\mathrm{Eseg}}}
\newcommand{\Esegz}{{\mathrm{Eseg}^{\mathbb{Z}}}}
\newcommand{\Vseg}{{\mathrm{Vseg}}}
\newcommand{\Vsegz}{{\mathrm{Vseg}^{\mathbb{Z}}}}
\newcommand{\Rep}{{\mathrm{Rep}}}

\newcommand{\VRep}{{\mathrm{VRep}}}

\newcommand{\B}{{\mathcal B}}

\newcommand{\inv}{^{-1}}
\newcommand{\sm}{\setminus}

\newcommand{\GL}{\mathrm{GL}}
\newcommand{\SO}{\mathrm{SO}}
\newcommand{\OO}{\mathrm{O}}
\newcommand{\SL}{\mathrm{SL}}
\newcommand{\Sp}{\mathrm{Sp}}

\newcommand{\BC}{\mathbb{C}}

\newcommand{\BB}{\mathcal{B}}

\newcommand{\EE}{\mathcal{E}}
\newcommand{\FF}{\mathcal{F}}

\newcommand{\Block}{\mathrm{Block}}
\newcommand{\ABlock}{\mathrm{ABlock}}

\newcommand{\rc}{\operatorname{rc}}
\newcommand{\tc}{\operatorname{tc}}

\renewcommand{\implies}{\Rightarrow}

\newcommand{\comment}[1]{}

\newtheorem{thm}{Theorem}[section]
\newtheorem{cor}[thm]{Corollary}
\newtheorem{lemma}[thm]{Lemma}
\newtheorem{prop}[thm]{Proposition}

\newtheorem{ques/conj}[thm]{Question/Conjecture}

\newtheorem{defn}[thm]{Definition}
\newtheorem{prob}[thm]{Problem}
\newtheorem{rmk}[thm]{Remark}

\newtheorem{exmp}[thm]{Example}

\DeclareMathOperator{\supp}{supp}

\numberwithin{equation}{section}

\begin{document}

\title{On Arthur packets containing a fixed tempered representation}
\author{Alexander Hazeltine, Aarya Kumar, Andrew Tung}
\date{\today}



\thanks{The research of the first named author was supported by the AMS-Simons Travel Grant program. The research of the second and third named authors was supported by Department of Mathematics at the University of Michigan and NSF Grant DMS-2301507.}

\begin{abstract}
    We determine the number of local Arthur packets containing a certain fixed tempered representation for classical $p$-adic groups. More specifically, given a tempered extended multi-segment supported in the integers, we determine a count for all extended multi-segments which arise from it through applications of the operators arising from the theory of intersections of local Arthur packets.
\end{abstract}

\maketitle


\section{Introduction}

Let $F$ be a $p$-adic field. For simplicity, we let $G_n$ denote the split group $\Sp_{2n}(F)$ throughout this introduction. We remark that our results hold for pure inner forms of quasi-split orthogonal or symplectic groups, but extra assumptions may be required to pass from the representation theory to the combinatorial arguments in certain cases (see Remark \ref{rmk Arthur packet obstruction}). 

A fundamental problem in the Langlands program is the establishment of the local Langlands correspondence, which is a parameterization of equivalence classes of complex irreducible admissible representations of a connected reductive $p$-adic groups.
Arthur proved the existence of the local Langlands correspondence for $G_n$ by establishing the existence of local Arthur packets (\cite[Theorem 1.5.1]{Art13}). The theory of local Arthur packets has proven very useful in studying representations of $G_n$; however, local Arthur packets tend to be ill-behaved. For example, local Arthur packets often have nontrivial intersections. An understanding of these intersections has been achieved in \cite{Ato23, HLL22, HLL25} largely through the use of various combinatorial algorithms. For theoretical applications, it is desirable to have more precise results and algorithms. For example, the following problem does not immediately have a clear, uniform answer from the combinatorial algorithms.

\begin{prob}\label{problem tempered}
    Given an irreducible tempered representation $\pi$ of $G_n$, determine the number of local Arthur packets which contain $\pi$.
\end{prob}

Answering this problem has important consequences. For example, a known case is when $\pi$ is supercuspidal. In this situation, work of M{\oe}glin implies that $\pi$ lies in as many local Arthur packets as possible (\cite{Moe06b,Moe09a}; see also \cite[Theorem 7.5]{GGP20}). Gan, Gross, and Prasad then use this result to construct examples of non-relevant Arthur parameters which contribute nonzero multiplicities to the non-tempered Gan-Gross-Prasad conjectures (\cite[\S7]{GGP20}). A solution to the above problem would be a generalization of M{\oe}glin's result and would become a source for further such examples.

Our main result is a recursive formula for determining the number of local Arthur packets which contain $\pi$, where $\pi$ is a certain kind of irreducible tempered representation.

\begin{thm}[{Theorem \ref{thm-count-temp}}]\label{thm tempered count intro}
For certain irreducible tempered representations $\pi$, there is a recursive formula for determining the number of local Arthur packets which contain $\pi$.
\end{thm}

We mention the ideas behind the theorem here. First, the parameterization of local Arthur packets using extended multi-segments (\cite{Ato20b, HLL25}; see \S\ref{sec extended multi-segments}) and the theory of intersections of local Arthur packets (\cite{HLL22, HLL25}; see \S\ref{sec intersections}) reduces the problem to understanding certain extended multi-segments. These naturally break into two cases: the case where the extended multi-segment is supported in the integers (see Definition \ref{def multi-segment}) and the case where the extended multi-segment is supported in the half-integers. In this article, we answer the integer case completely in Theorem \ref{thm-count-temp}. Hence the above theorem applies to any tempered representations whose associated extended multi-segments is supported entirely in the integers. In general, a tempered representation $\pi$ can have an extended multi-segment which is supported in both the integers and half-integers, in which case our theorem provides a lower bound for the number of local Arthur packets containing $\pi$.

The key idea behind Theorem \ref{thm tempered count intro} is that an extended multi-segment (associated to a tempered representation $\pi$) which is supported in the integers can be decomposed into blocks (see Definition \ref{def-block}). This decomposition is uniquely determined by $\pi$ (Lemma \ref{lem-unique-block-decomp}). Theorem \ref{thm-count-temp} then relates the number of local Arthur packets containing $\pi$ with the number of local Arthur packets each block ``belongs'' to (this is made precise in \S\ref{sec-results}).  Theorem \ref{thm-count-block-temp} further gives a recursive formula for determining this count in terms of smaller blocks. 

\begin{rmk}
    To answer Problem \ref{problem tempered}, it only remains to complete the half-integer case. We expect that the results from the integer case will prove useful for this as the operators involved in the combinatorics are largely the same; however, there will certainly be some differences. For example, there is an extra operation that one needs to consider (the partial dual; see \cite[Definition 6.5]{HLL22}) which does not occur in the integer case. 
\end{rmk}

We mention that, despite not completely answering Problem \ref{problem tempered}, Theorem \ref{thm tempered count intro} is still useful for theoretical applications. Indeed, in a forthcoming article (\cite{HKT}), we will use Theorem \ref{thm tempered count intro} to study the number of Arthur packets containing the theta lift to the first occurrence of going-up tower of an irreducible tempered (or anti-tempered) representation of $G_n$.

Here is the organization of this article. In \S\ref{sec background}, we recall various results on local Arthur packets. In \S\ref{sec-results}, we state the main results precisely. In preparation for their proofs, we collect some preliminary definitions and statements in \S\ref{sec basic notions and lemmas}. In \S\ref{sec-individual-blocks}, we study the case of individual blocks. From this theory, we prove the first main result (the number of local Arthur packets containing a given block, Theorem \ref{thm-count-block-temp})  in \S\ref{sec count block}. Finally, we study how blocks may interact with each other in \S\ref{sec interation of blocks} and prove the other main result (Theorem \ref{thm-count-temp}) which gives Theorem \ref{thm tempered count intro}.

\subsection*{Acknowledgments}
The first named author thanks Yiannis Sakellaridis for inspiring conversations which lead to the genesis of this research, especially in regards to the forthcoming application to the local theta correspondence. The authors also thank Chi-Heng Lo for enlightening discussions and remarks. Finally, the authors thank Wee Teck Gan for helpful comments and support.

\section{Background}\label{sec background}

Recall that $F$ is a non-Archimedean local field of characteristic 0. For the moment, let $\mathrm{G}$ be a reductive group defined over $F$ and $G=\mathrm{G}(F).$ We are primarily concerned with the set of equivalence classes of complex irreducible admissible representations of $G$, which we denote by $\Pi(G).$ 

We let $\mathrm{G}_n=\Sp_{2n},$ $ \SO_{2n+1},$ or $\OO_{2n}$ and assume that $G_n^+=\mathrm{G}_n(F)$ is quasi-split. When $G_n^+$ is a quasi-split non-split even orthogonal group, we remark that $G_n^+$ has a quasi-split non-split pure inner form which we denote by $G_n^-.$ For brevity, we let $G_n\in\{G_n^\pm\}.$ We set $\epsilon_{G_n}=-1$ if $G_n=G_n^-$ and $\epsilon_{G_n^+}=1$, otherwise. We expect our results to apply more generally than groups considered above. Specifically, we expect analogous results for pure inner forms of quasi-split symplectic, odd special orthogonal, or even orthogonal groups; however, the necessary preliminary results on local Arthur packets have not been established in the necessary generality (see Remark \ref{rmk Arthur packet obstruction}).

\subsection{Local Arthur packets}\label{sec local Arthur packets}

In this subsection, we recall some results concerning local Arthur packets. Recall that $\OO_{2n}$ is disconnected. Consequently, we consider the L-group for the identity component ${}^LG_n:={}^L{G_n^\circ}.$ We let $\widehat{G}_n(\BC)$ denote the usual dual group when $G_n$ is connected and set $\widehat{\OO_{2n}}:=\OO_{2n}(\BC)$ otherwise.
\begin{defn}
A \emph{local Arthur parameter} of $G_n$ is a homomorphism
$$\psi: W_F \times \SL_2(\mathbb{C}) \times \SL_2(\mathbb{C}) \rightarrow {}^LG_n$$
which can be expressed as a direct sum of irreducible representations
\begin{equation}\label{eq decomp psi +}
  \psi = \bigoplus_{i=1}^r \phi_i|\cdot|^{x_i} \otimes S_{a_i} \otimes S_{b_i},  
\end{equation}
satisfying the following conditions:
\begin{enumerate}
    \item [(1)]$\phi_i(W_F)$ is bounded and consists of semi-simple elements, and $\dim(\phi_i)=d_i$;
    \item [(2)] $x_i \in \R$ and $|x_i|<\frac{1}{2}$;
    \item [(3)]the restrictions of $\psi$ to both copies of $\SL_2(\mathbb{C})$ are algebraic, $S_k$ denotes the unique $k$-dimensional irreducible representation of $\SL_2(\mathbb{C})$, and we have 
    $\sum_{i=1}^r d_ia_ib_i = N,
$
where $N=2n$ if $G_n$ is not symplectic and $N=2n+1$ otherwise.
\end{enumerate}

Two local Arthur parameters of $G_n$ are said to be \emph{equivalent} if they are conjugate under $\widehat{G}_n(\BC)$.  
We let $\Psi^+(G_n)$ denote the set of equivalence classes of  local Arthur parameters. We will not distinguish a local Arthur parameter $\psi$ and its equivalence class. We let $\Psi(G_n)$ denote the subset of equivalence classes of  bounded  local Arthur parameters, i.e., those $\psi$ for which $x_i=0$ for any $i=1,\dots,r$ in the decomposition \eqref{eq decomp psi +}.
\end{defn}

By the Local Langlands Correspondence for $\GL_{d}(F)$, any bounded representation $\phi$ of $W_F$ corresponds to an irreducible unitary supercuspidal representation $\rho$ of $\GL_{d}(F)$ (\cite{HT01, Hen00, Sch13}). Consequently, we identify \eqref{eq decomp psi +} as
\begin{equation}\label{A-param decomp}
  \psi = \bigoplus_{\rho}\left(\bigoplus_{i\in I_\rho} \rho|\cdot|^{x_i} \otimes S_{a_i} \otimes S_{b_i}\right),  
\end{equation}
where the first sum runs over a finite set of
irreducible unitary supercuspidal representations $\rho$ of $\GL_d(F)$ where $d \in \mathbb{Z}_{\geq 1}$.

For a local Arthur parameter $\psi \in \Psi(G_n)$, Arthur showed the existence of a finite multi-set $\Pi_\psi$, called a local Arthur packet, consisting of irreducible unitary representations of $G_n$ that satisfy certain twisted endoscopic character identities  (\cite[Theorem 2.2.1]{Art13}; see also \cite{AGIKMS24}). We do not recall the precise definition of $\Pi_\psi$. Instead, it suffices for our purposes to recall a parameterization of $\Pi_\psi$ using extended multi-segments (see Theorem \ref{thm Arthur packets extended multi-segment parameterization}).

M{\oe}glin showed that the computation of $\Pi_\psi$ can be reduced to the ``good parity'' case (see Theorem \ref{thm red to gp} below). We proceed by recalling this reduction.
\begin{defn}\label{def Arthur parameter good parity}
    Let $\psi$ be a local Arthur parameter as in \eqref{A-param decomp}. We say that $\psi$ is of \emph{good parity} if $\psi \in \Psi(G_n)$, i.e., $x_i=0$ for all $i$, and every summand $\rho \otimes S_{a_i} \otimes S_{b_i}$ is self-dual and of the same type as $\psi$.
We let $\Psi_{gp}(G_n)$ denote the subset of $\Psi^+(G_n)$ consisting of local Arthur parameters of good parity.
\end{defn}
We explicate this condition further. 
Consider an irreducible summand $\rho|\cdot|^x\otimes S_a \otimes S_b$ of $\psi$ as in \eqref{A-param decomp}. This summand is self-dual and orthogonal if and only if $x=0,$ $\rho$ is orthogonal (resp. symplectic), and $a_i+b_i$ is even (resp. odd). It is self-dual and symplectic if and only if $x=0,$ $\rho$ is symplectic (resp. orthogonal), and $a_i+b_i$ is even (resp. odd).

Let $\psi\in\Psi^+(G_n).$ Since $\psi$ is self-dual, we have a decomposition 
$
\psi=\psi_{ngp} + \psi_{gp}+\psi_{ngp}^\vee,
$
where $\psi_{ngp}^\vee$ denotes the dual of $\psi_{ngp}$, $\psi_{gp}\in\Psi_{gp}(G_n),$ and $\psi_{gp}$ is maximal for this decomposition, i.e., if we decompose $\psi_{ngp}+\psi_{ngp}^\vee$ as in \eqref{A-param decomp}, then any irreducible summand $\rho|\cdot|^x\otimes S_a\otimes S_b$ is not of good parity. Note that $\psi_{gp}$ is uniquely determined by this decomposition, but $\psi_{ngp}$ is not necessarily unique. M{\oe}glin showed that the local Arthur packet $\Pi_\psi$ can be constructed from $\Pi_{\psi_{gp}}.$

\begin{thm}[{\cite[Proposition 5.1]{Moe11b}}]\label{thm red to gp}
Let $\psi\in\Psi^+(G_n)$ with decomposition $\psi=\psi_{ngp}+\psi_{gp}+\psi_{ngp}^\vee$ as above. Then, there exists $\tau\in\Pi(\GL_d(F))$ (determined by $\psi_{ngp}$) such that for any $\pi_{gp}\in\Pi_{\psi_{gp}},$ the normalized parabolic induction $\tau\rtimes\pi_{gp}$ is irreducible and $\Pi_\psi=\{\tau\rtimes\pi_{gp} \ | \ \pi_{gp}\in\Pi_{\psi_{gp}}\}.$
\end{thm}

\subsection{Extended multi-segments}\label{sec extended multi-segments}
Theorem \ref{thm red to gp} reduces Problem \ref{problem tempered} to the good parity setting. In this section, we recall a parameterization of $\Pi_\psi$ for $\psi\in\Psi_{gp}(G_n).$ using extended multi-segments (Theorem \ref{thm Arthur packets extended multi-segment parameterization}). Furthermore, we also recall some results on the theory of intersections of local Arthur packets from \cite{HLL22, HLL25}. 

We fix the following notation throughout this section. Let $\psi\in\Psi_{gp}(G_n)$ with decomposition 
\[ \psi= \bigoplus_{\rho} \bigoplus_{i \in I_{\rho}} \rho \otimes S_{a_i} \otimes S_{b_i}. \]
We set $A_i=\frac{a_i+b_i}{2}-1$ and $B_i=\frac{a_i-b_i}{2}$ for $i\in I_\rho.$ 
 
We say that a total order $>_\psi$ on $I_\rho$ is \emph{admissible} if satisfies:
\[
\tag{$P$}
\text{
For $i,j \in I_\rho$, 
if $A_i > A_j$ and $B_i > B_j$, 
then $i >_\psi j$.
}
\]
Sometimes we consider an order $>_\psi$ on $I_\rho$ satisfying:
\[
\tag{$P'$}
\text{
For $i,j \in I_\rho$, 
if $B_i > B_j$, 
then $i >_\psi j$.
}
\]
Note that ($P'$) implies ($P$). For brevity, we often write $>$ instead of $>_\psi$ when it is clear that we are working with a fixed admissible order. 

Suppose now that we have fixed an admissible order for $\psi.$ We define the \emph{support} of $\psi$ to be the collection of ordered multi-sets 
$$\supp(\psi) := \cup_{\rho}\{ [A_i,B_i]_{\rho} \}_{i \in (I_\rho,>)}.
$$
Note that $\supp(\psi)$ depends implicitly on the fixed admissible order.

We recall the definition of extended multi-segments.

\begin{defn}\label{def multi-segment} \
\begin{enumerate}
\item
An \emph{extended segment} is a triple $([A,B]_\rho, l, \eta)$,
where
\begin{itemize}
\item
$[A,B]_\rho = \{\rho|\cdot|^A, \rho|\cdot|^{A-1}, \dots, \rho|\cdot|^B \}$ is a segment 
for an irreducible unitary supercuspidal representation $\rho$ of some $\GL_d(F)$; 
\item
$l \in \Z$ with $0 \leq l \leq \frac{b}{2}$, where $b = \#[A,B]_\rho = A-B+1$; 
\item
$\eta \in \{\pm1\}$. 
\end{itemize}

\item
An \emph{extended multi-segment} for $G_n$ is 
an equivalence class (via the equivalence defined below) of multi-sets of extended segments 
\[
\EE = \cup_{\rho}\{ ([A_i,B_i]_{\rho}, l_i, \eta_i) \}_{i \in (I_\rho,>)}
\]
such that 
\begin{itemize}
\item
$I_\rho$ is a totally ordered finite set with a fixed admissible total order $>$;

\item
$A_i + B_i \geq 0$ for all $\rho$ and $i \in I_\rho$; 

\item
we have that
\[
\psi_{\EE} = \bigoplus_\rho \bigoplus_{i \in I_\rho} \rho \otimes S_{a_i} \otimes S_{b_i}, 
\]
where $(a_i, b_i) = (A_i+B_i+1, A_i-B_i+1)$,
is a local Arthur parameter for $G_n$ of good parity;
\item and the following sign condition holds
\begin{align}\label{eq sign condition}
\prod_{\rho} \prod_{i \in I_\rho} (-1)^{[\frac{b_i}{2}]+l_i} \eta_i^{b_i} = \epsilon_{G_n}.
\end{align}
\end{itemize}

\item
Extended segments $([A,B]_\rho, l, \eta)$ and $([A',B']_{\rho'}, l', \eta')$ are \emph{weakly equivalent} 
if 
\begin{itemize}
\item
$[A,B]_\rho = [A',B']_{\rho'}$; 
\item
$l = l'$; and 
\item
$\eta = \eta'$ whenever $l = l' < \frac{b}{2}$. 
\end{itemize}
Two extended multi-segments 
$\EE = \cup_{\rho}\{ ([A_i,B_i]_{\rho}, l_i, \eta_i) \}_{i \in (I_\rho,>)}$ 
and 
$\EE' = \cup_{\rho}\{ ([A'_i,B'_i]_{\rho}, l'_i, \eta'_i) \}_{i \in (I_\rho,>)}$ 
are \emph{weakly equivalent}
if for any $\rho$ and $i \in I_\rho$, the extended segments $([A_i,B_i]_\rho, l_i, \eta_i)$ and $([A'_i,B'_i]_{\rho}, l'_i, \eta'_i)$ are weakly equivalent.

Note that if $l=\frac{b}{2}$, then $\eta$ is arbitrary. In this case, we always take $\eta=1$ (unless explicitly stated otherwise, e.g., in Definition \ref{def dual}).

\item
We define the \emph{support} of $\EE$ to be the collection of ordered multi-sets 
\[
\supp(\EE) = \cup_{\rho}\{ [A_i,B_i]_{\rho} \}_{i \in (I_\rho,>)}.
\]
We implicitly include the admissible order $>$ in $\supp(\EE).$
\item We say that $\EE$ is \emph{integral} (or supported in the integers) if $A_i\in\mathbb{Z}$ for any $i\in I_\rho$ and any $\rho$.
\item We let $\Eseg(G_n)$ denote the set of all extended multi-segments of $G_n$ up to weak equivalence. We further let $\Esegz(G_n)$ denote the subset of $\Eseg(G_n)$ consisting of integral extended multi-segments up to weak equivalence.
\end{enumerate}
\end{defn}
If the admissible order $>$ is clear in the context, for $k \in I_{\rho}$, we often let $k+1 \in I_{\rho}$ be the unique element adjacent with $k$ and $k+1>k$.

The data in an extended multi-segment can be cumbersome to list out in detail. Instead, 
we attach a symbol to each extended multi-segment by the same way in \cite[Section 3]{Ato20b}. We give an example to explain this.
\begin{exmp}\label{exmp symbol}
 Let $\rho$ be an orthogonal representation of $\GL_d(F)$. The symbol
\[\EE=\bordermatrix{
& -1 & 0 &1 & 2 &3 & 4\cr
& \lhd & \lhd & \oplus & \ominus & \rhd & \rhd \cr
&  &  &  & \lhd &\rhd  &  \cr
&  &  &  &  &  & \ominus \cr
}_{\rho}\]
corresponds to $\EE= \{ ([A_i,B_i]_{\rho},l_i,\eta_i)\}_{i \in (1<2<3)}$ of $\Sp_{44d}(F)$ where the data is given as follows.
\begin{itemize}
    \item  $([A_1,B_1]_{\rho},[A_2,B_2]_{\rho},[A_3,B_3]_{\rho})=([4,-1]_{\rho},[3,2]_{\rho},[4,4]_{\rho})$ specify the ``support" of each row.
    \item  $(l_1,l_2,l_3)=(2,1,0)$ counts the number of pairs of triangles in each row. 
    \item  $(\eta_1,\eta_2,\eta_3)=(1, 1, -1)$ records the sign of the first circle in each row. Note that setting $\eta_2=\pm1$ results in weakly equivalent extended multi-segments.
\end{itemize}
The associated local Arthur parameter is
    \[ \psi_{\EE}= \rho \otimes S_{4}\otimes S_{6} + \rho \otimes S_{6}\otimes S_{2} + \rho \otimes S_9 \otimes S_1.  \]
\end{exmp}

With the above symbol in mind, we often say that an extended segment $r=([A,B]_\rho,l,\eta)$ is a \emph{row} of $\EE$ if $r\in \EE.$ Furthermore, we define the support of $r$ to be $\supp(r)=[A,B]_\rho$ and let $A(r)=A$, $B(r)=B,$ $a(r)=A(r)+B(r)+1,$ $b(r)=A(r)-B(r)+1,$ $l(r)=l$, and $\eta(r)=\eta.$

Let $\EE\in\Eseg(G_n).$ We attach a representation $\pi(\EE)$ of $G_n$ via \cite[Algorithm 5.18]{HLL25} which is an adaptation of \cite[Algorithm 6.3]{HJLLZ24}. We have that either $\pi(\EE)$ vanishes or $\pi(\EE)\in\Pi(G_n).$ We do not recall the full construction here as, for our purposes, it suffices to recall various consequences of the definition.

\begin{rmk}
We remark that \cite[Algorithm 5.18]{HLL25} agrees with that of Atobe given in \cite[\S3.2]{Ato20b} for split symplectic and odd special orthogonal groups. The advantage of \cite[Algorithm 5.18]{HLL25} over \cite[\S3.2]{Ato20b} is that it only assumes the local Langlands correspondence and hence applies more broadly than \cite[\S3.2]{Ato20b} which assumes a theory of derivatives as in \cite{AM23}. In particular, \cite[Algorithm 5.18]{HLL25} is well-defined for quasi-split even orthogonal groups.
\end{rmk}

Extended multi-segments provide a parameterization of local Arthur packets.
\begin{thm}\label{thm Arthur packets extended multi-segment parameterization}
    Suppose $\psi= \bigoplus_{\rho} \bigoplus_{i \in I_{\rho}} \rho \otimes S_{a_i} \otimes S_{b_i}$ is a local Arthur parameter of good parity of $G_n$ which is assumed to be quasi-split. Choose an admissible order $>$ on $I_{\rho}$ for each $\rho$ that satisfies ($P'$) if $\frac{a_i-b_i}{2}<0$ for some $i \in I_{\rho}$. Then
\[ \bigoplus_{\pi \in \Pi_{\psi}} \pi= \bigoplus_{\EE} \pi(\EE),\]
where $\EE$ runs over all extended multi-segments with $\supp(\EE)= \supp(\psi)$ and $\pi(\EE) \neq 0$. 
\end{thm}
For split symplectic and odd special orthogonal groups, the above theorem was first proven by Atobe (\cite[Theorem 3.3]{Ato20b}) using the construction of $\pi(\EE)$ from \cite[\S3.2]{Ato20b}. Then \cite[Algorithm 6.3]{HJLLZ24} provided another way to compute the representation $\pi(\EE)$ given in \cite[\S3.2]{Ato20b} which only uses the local Langlands correspondence. Finally, in \cite[Theorem 6.10]{HLL25}, the above theorem was proven for quasi-split even orthogonal groups (assuming a local Langlands correspondence as in \cite{AG17b}).

M{\oe}glin showed that local Arthur packets are multiplicity-free (\cite{Moe11a}). Consequently, we have the following corollary.
\begin{cor}\label{cor same support}
    Let $\EE_1, \EE_2\in\Eseg(G_n)$ and suppose that $\supp(\EE_1)=\supp(\EE_2).$ If $\frac{a_i-b_i}{2}<0$ for some $i \in I_{\rho}$, then we also require that the order on $I_\rho$ is $(P').$ If $\pi(\EE_1)=\pi(\EE_2)\neq 0,$ then $\EE_1=\EE_2.$
\end{cor}

Recall that $\supp(\EE)$ implicitly records the admissible order on $\EE\in\Eseg(G_n)$. Thus, the hypothesis $\supp(\EE_1)=\supp(\EE_2)$ in the above corollary also asserts that the admissible orders on $\EE_1$ and $\EE_2$ agree.

\subsection{Intersections}\label{sec intersections}

In this subsection, we recall the theory of intersections of local Arthur packets as developed in \cite{HLL22, HLL25}. We begin by recalling various operators which are used in the classification of these intersections (see Theorem \ref{thm intersections of local Arthur packets}), along with some other useful operators and their properties.

We note that the effect of an operator often only depends on a fixed $\rho.$ To simplify the definitions of the operators, we introduce the following notation.

\begin{defn}
Let $\EE= \cup_{\rho} \{([A_i,B_i]_{\rho}, l_i, \eta_i)\}_{i \in (I_{\rho},>)}\in\Eseg(G_n).$
We set
\[ \EE_{\rho}=\{([A_i,B_i]_{\rho}, l_i, \eta_i)\}_{i \in (I_{\rho},>)},\  \EE^{\rho}=\cup_{\rho' \not\cong \rho}\{([A_i,B_i]_{\rho'}, l_i, \eta_i)\}_{i \in (I_{\rho'},>)}. \]
We let $\Vseg(G_n)$ denote the set of elements of the form $\FF=\cup_{\rho'} \EE_{\rho'}$
where the union is over a finite set of irreducible self-dual supercuspidal representations $\rho'$ of $\GL_d(F),$ $d\geq 1,$ and $\EE\in\Eseg(G_n).$ Essentially, $\mathcal{F}$ is an extended multi-segment (for some group) except that we not enforce the sign condition \eqref{eq sign condition}. We say that $\mathcal{F}$ is a \emph{virtual extended multi-segment}.

We also set $\Vseg_\rho(G_n)$ to be the set of elements of the form $\EE_\rho$ for some $\EE\in\Eseg(G_n)$ and $\Vseg^\mathbb{Z}(G_n)$ to be the set of $\FF=\cup_{\rho'} \EE_{\rho'}$ for some $\EE\in\Esegz(G_n)$. We set $\Vseg_\rho^\mathbb{Z}(G_n)=\Vseg_\rho(G_n)\cap\Vseg^\mathbb{Z}(G_n)$. The good parity requirement may cause $\Vseg_\rho^\mathbb{Z}(G_n)$ to be empty and when we write it, we implicitly assume that $\rho$ is taken so that it is nonempty.
\end{defn}

We also recall the definition of a symbol.
\begin{defn}\label{def symbol}
  A \emph{symbol} is a multi-set of extended segments 
\[
\EE = \cup_{\rho}\{ ([A_i,B_i]_{\rho}, l_i, \eta_i) \}_{i \in (I_\rho>)},
\] which satisfies the same conditions in Definition \ref{def multi-segment}(2) except we drop the condition $0 \leq l_i \leq \frac{b_i}{2}$, for each $i \in I_{\rho}$.
\end{defn}

Any change of admissible orders can be derived from a composition of the row exchange operators $R_k$ which we recall from \cite[Definition 3.15]{HLL22}.

\begin{defn}[Row exchange]\label{def row exchange} 
Suppose $\EE$ is a symbol where
$$\EE_{\rho}=\{([A_i,B_i]_{\rho},l_i,\eta_i)\}_{i \in (I_{\rho},>)}.$$
For $k<k+1 \in I_{\rho}$, let $\gg$ be the total order on $I_\rho$ defined by $k\gg k+1$ and if $(i,j)\neq (k,k+1)$, then $ i \gg j$ if and only if $
i >j .$ 

Suppose $\gg$ is not an admissible order on $I_{\rho}$, then we define $R_k(\EE)=\EE$. Otherwise, we define 
\[R_{k}(\EE_{\rho})=\{([A_i,B_i]_{\rho},l_i',\eta_i')\}_{i \in (I_{\rho},\gg)},\]
where $( l_i',\eta_i')=(l_i,\eta_i)$ for $i \neq k,k+1$, and $(l_k',\eta_k')$ and $(l_{k+1}', \eta_{k+1}')$ are given as follows: Denote $\epsilon=(-1)^{A_k-B_k}\eta_k\eta_{k+1}$.
\begin{enumerate}
    \item [Case 1.] $ [A_k,B_k]_{\rho} \supset [A_{k+1},B_{k+1}]_{\rho}$:
    
    In this case, we set $(l_{k+1}',\eta_{k+1}')=(l_{k+1}, (-1)^{A_k-B_k}\eta_{k+1})$, and
    \begin{enumerate}
    \item [(a)] If $\epsilon=1$ and $b_k- 2l_k < 2(b_{k+1}-2l_{k+1})$, then
    \[ (l_k', \eta_{k}')= (b_k-(l_k+ (b_{k+1}-2l_{k+1})), (-1)^{A_{k+1}-B_{k+1}} \eta_k).  \]
    \item [(b)] If $\epsilon=1$ and $b_k- 2l_k \geq  2(b_{k+1}-2l_{k+1})$, then
    \[ (l_{k}', \eta_{k}')= (l_k+ (b_{k+1}-2l_{k+1}), (-1)^{A_{k+1}-B_{k+1}+1} \eta_k).  \]
    \item [(c)] If $\epsilon=-1$, then
    \[ (l_{k}', \eta_{k}')= (l_k- (b_{k+1}-2l_{k+1}), (-1)^{A_{k+1}-B_{k+1}+1} \eta_k).  \]
\end{enumerate}
    \item [Case 2.] $ [A_k,B_k]_{\rho} \subset [A_{k+1},B_{k+1}]_{\rho}$:
    
    In this case, we set $(l_{k}',\eta_{k}')=(l_{k}, (-1)^{A_{k+1}-B_{k+1}}\eta_{k})$, and
    \begin{enumerate}
   \item [(a)] If $\epsilon=1$ and $b_{k+1}- 2l_{k+1} < 2(b_{k}-2l_{k})$, then
    \[ (l_{k+1}', \eta_{k+1}')= (b_{k+1}-(l_{k+1}+ (b_{k}-2l_{k})), (-1)^{A_{k}-B_{k}} \eta_{k+1}).  \]
    \item [(b)] If $\epsilon=1$ and $b_{k+1}- 2l_{k+1} \geq  2(b_{k}-2l_{k})$,
    then
    \[ (l_{k+1}', \eta_{k+1}')= (l_{k+1}+ (b_{k}-2l_{k}), (-1)^{A_{k}-B_{k}+1} \eta_{k+1}).  \]
    \item [(c)] If $\epsilon=-1$, then
    \[ (l_{k+1}', \eta_{k+1}')= (l_{k+1}- (b_{k}-2l_{k}), (-1)^{A_{k}-B_{k}+1} \eta_{k+1}).  \]
\end{enumerate}
\end{enumerate}
Finally, we define $R_{k}(\EE)= \EE^{\rho} \cup R_{k}(\EE_{\rho})$.
\end{defn}

We remark that there is another definition of row exchange given in \cite[Section 4.2]{Ato20b}; however, these definitions agree when $\pi(\EE)\neq 0.$ 
The next operator we recall is known as union-intersection.

\begin{defn}[union-intersection]\label{ui def}
 Let $\EE\in\Eseg(G)$. For $k< k+1 \in I_{\rho}$, we define an operator $ui_k$, called union-intersection, on $\EE$ as follows. Write 
 \[ \EE_{\rho}= \{([A_i,B_i]_\rho, l_i,\eta_i)\}_{i \in (I_{\rho},>)}.\]
  Denote $\epsilon=(-1)^{A_k-B_k}\eta_k \eta_{k+1}.$ If $A_{k+1}>A_k$, $B_{k+1}>B_k$ and any of the following cases holds:
\begin{enumerate}
    \item [{Case 1}.] $ \epsilon=1$ and $A_{k+1}-l_{k+1}=A_k-l_k,$
    \item [{Case 2}.] $ \epsilon=1$ and $B_{k+1}+l_{k+1}=B_k+l_k,$
    \item [{Case 3}.] $ \epsilon=-1$ and $B_{k+1}+l_{k+1}=A_k-l_k+1.$
\end{enumerate}
We define
\begin{align*}
     ui_{k}(\EE_{\rho})=\{ ([A_i',B_i']_{\rho},l_i',\eta_i')\}_{i \in (I_{\rho}, >)},
\end{align*} 
where $ ([A_i',B_i']_{\rho},l_i',\eta_i')=([A_i,B_i]_{\rho},l_i,\eta_i)$ for $i \neq k,k+1$, $[A_k',B_k']_{\rho}=[A_{k+1},B_k]_{\rho}$, $[A_{k+1}',B_{k+1}']_{\rho}=[A_k,B_{k+1}]_{\rho}$, and $( l_k', \eta_k', l_{k+1}',\eta_{k+1}' )$ are given case by case as follows:
\begin{enumerate}
    \item[$(1)$] in Case 1, $( l_k', \eta_k', l_{k+1}',\eta_{k+1}' )= (l_k,\eta_k, l_{k+1}-(A_{k+1}-A_k), (-1)^{A_{k+1}-A_k}\eta_{k+1})$;
    \item [$(2)$] in Case 2, if $b_k-2l_k \geq A_{k+1}-A_k$, then
    \[( l_k', \eta_k', l_{k+1}',\eta_{k+1}' )= (l_k+(A_{k+1}-A_k),\eta_k, l_{k+1}, (-1)^{A_{k+1}-A_k}\eta_{k+1}),\]
    if $b_k-2l_k < A_{k+1}-A_k$, then
    \[( l_k', \eta_k', l_{k+1}',\eta_{k+1}' )= (b_k-l_k,-\eta_k, l_{k+1}, (-1)^{A_{k+1}-A_k}\eta_{k+1});\]
    \item [$(3)$] in Case 3, if $l_{k+1} \leq  l_k$, then
    \[( l_k', \eta_k', l_{k+1}',\eta_{k+1}' )= (l_k,\eta_k, l_{k+1}, (-1)^{A_{k+1}-A_k}\eta_{k+1}),\]
    if $l_{k+1}> l_{k}$, then
    \[( l_k', \eta_k', l_{k+1}',\eta_{k+1}' )= (l_k,\eta_k, l_{k}, (-1)^{A_{k+1}-A_k+1}\eta_{k+1});\]
    \item [$(3')$] in Case 3, if $l_k=l_{k+1}=0$, then we delete $ ([A_{k+1}',B_{k+1}']_{\rho},l_{k+1}',\eta_{k+1}')$ from $ui_k(\EE_{\rho})$.
\end{enumerate}
Otherwise, we define $ui_k(\EE_{\rho})=\EE_{\rho}$. In any case, we define $ui_k(\EE)= \EE^{\rho} \cup ui_k(\EE_{\rho})$.

We say $ui_k$ is applicable on $\EE$ or $\EE_{\rho}$ if $ui_k(\EE)\neq \EE$. We say this $ui_k$ is of type 1 (resp. 2, 3, 3') if $\EE_{\rho}$ is in Case 1 (resp. 2, 3, 3'). 
\end{defn}

We also consider the composition of the row exchange and union-intersection operators as follows.

\begin{defn} \label{def ui}
Suppose $\EE\in\Eseg(G)$ and write
$\EE_{\rho}=\{ ([A_i,B_i]_{\rho},l_i,\eta_i)\}_{i\in (I_{\rho,>})}.$ 
Given $i,j \in I_{\rho}$, we define $ui_{i,j}(\EE_{\rho})=\EE_{\rho}$ unless
\begin{enumerate}
    \item [1.] we have $ A_i< A_j$, $B_i <B_j$ and $(j,i,>')$ is an adjacent pair for some admissible order $>'$ on $I_{\rho}$ and 
    \item [2.] $ui_i$ is applicable on $\EE_{\rho,>'}$
\end{enumerate}
In this case, we define $ui_{i,j}(\EE_{\rho}):=(ui_{i}(\EE_{\rho,>'}))_{>}$, so that the admissible order of $ui_{i,j}(\EE_{\rho})$ and $\EE_{\rho}$ are the same. (If the $ui_i$ is of type 3', then we delete the $j$-th row.) Finally, we define $ui_{i,j}(\EE)= \EE^{\rho} \cup ui_{i,j}(\EE_{\rho})$.

We say $ui_{i,j}$ is applicable on $\EE$ if $ui_{i,j}(\EE) \neq \EE$. Furthermore, we say that $ui_{i,j}$ is of type 1, 2, 3, or 3' if the operation $ui_i$ is of type 1, 2, 3, or 3', respectively, in Definition \ref{ui def}.
\end{defn}

The next operator is called the dual operator.
  
\begin{defn}[dual]\label{def dual}
Let $\EE= \cup_\rho \{([A_i,B_i]_{\rho},l_i,\eta_i)\}_{i\in (I_\rho, >)}$ be an extended multi-segment such that the admissible order $>$ on $I_{\rho}$ satisfies $(P')$ for all $\rho$. We define 
$$dual(\EE)=\cup_{\rho}\{([A_i,-B_i]_{\rho},l_i',\eta_i')\}_{i\in (I_\rho, >')}$$ as follows:
\begin{enumerate}
    \item The order $>'$ is defined by $i>'j$ if and only if $j>i.$ 
    \item We set \begin{align*}
l_i'=\begin{cases}
l_i+B_i  & \mathrm{if} \, B_i\in\mathbb{Z},\\
 l_i+B_i+\frac{1}{2}(-1)^{\alpha_{i}}\eta_i  & \mathrm{if} \, B_i\not\in\mathbb{Z},
\end{cases}
\end{align*}
and
\begin{align*}
\eta_i'=\begin{cases}
(-1)^{\alpha_i+\beta_i}\eta_i  & \mathrm{if} \, B_i\in\mathbb{Z},\\
 (-1)^{\alpha_i+\beta_i+1}\eta_i  & \mathrm{if} \, B_i\not\in\mathbb{Z},
\end{cases}
\end{align*}
where $\alpha_{i}=\sum_{j\in I_\rho, j<i}a_j,$ and $\beta_{i}=\sum_{j\in I_\rho, j>i}b_j,$ $a_j=A_j+B_j+1$, $b_j=A_j-B_j+1$.
\item When $B_i\not\in\mathbb{Z}$ and $l_i=\frac{b_i}{2}$, we set $\eta_i=(-1)^{\alpha_i+1}.$
\end{enumerate}
If $\FF= \EE_{\rho}$, we define $dual(\FF):= (dual(\EE))_{\rho}$.

As a shorthand, for each $i\in I_{\rho},$ let $r_i=([A_i,B_i]_{\rho},l_i,\eta_i)$. Then we let $\widehat{r}_i$ denote the effect of the dual operator on $\EE$ on this row, i.e., 
\[
\widehat{r}_i=([A_i,-B_i]_{\rho},l_i',\eta_i').\]
\end{defn} 

We note that if $\pi(\EE)\neq 0,$ then $\pi(dual(\EE))$ is the Aubert-Zelevinsky dual of $\pi(\EE)$ (\cite[Theorem 6.2]{Ato20b} if $G_n=\Sp_{2n}(F)$ and \cite[Proposition 6.11]{HLL25} if $G_n=\OO_{2n}^\pm(F)$). For our purposes, it is sufficient to use the following direct implication.

\begin{lemma}\label{lemma dual}
    Let $\EE\in\Eseg(G_n)$ be such that $\pi(\EE)\neq 0.$ Then $\pi(dual(\EE))\neq 0.$ Moreover, if $\EE'\in\Eseg(G_n)$ is such that $\pi(\EE')=\pi(\EE)$, then $\pi(dual(\EE'))=\pi(dual(\EE)).$
\end{lemma}

From \cite{HLL22, HLL25}, we understand the inverse of union-intersections not of type 3'.

\begin{lemma}\label{lemma ui inv = dud}
   If $ui$ is not of type 3', then its inverse is of the form $dual\circ ui\circ dual$ of the same type. 
\end{lemma}

To complete the theory of intersections of local Arthur packets, another operator, known as the partial dual operator and denoted $dual_k$ (see \cite[Definition 6.5]{HLL22}), is also needed. However, it is not needed for our main result (Theorem \ref{thm-count-temp}) as it does not apply when $\EE\in\Vsegz(G_n).$

We distinguish certain operators that have a key role in the intersection of local Arthur packets (see Theorem \ref{thm intersections of local Arthur packets} below). We set 
\[
\Psi(\pi)=\{\psi\in\Psi(G_n) \ | \ \pi\in\Pi_\psi\}.
\]

\begin{defn}\label{defn raising operators}
    The operators $dual\circ ui \circ dual,$ $ui^{-1}$, and $dual_k^-$ are called \emph{raising operators}. Given local Arthur parameters $\psi_1, \psi_2\in\Psi(\pi)$, we write $\psi_1\geq_O\psi_2$ if there exists a sequence of raising operators $(T_i)_{i=1}^l$ such that \[\EE_1=(T_l\circ T_{l-1}\circ\cdots\circ T_1)(\EE_2),\] where $\EE_j\in\Eseg(G_n)$ is such that $\psi_{\EE_j}=\psi_j$ and $\pi(\EE_j)=\pi$ for $j=1,2.$ 
\end{defn}

Raising operators fully determine intersections of local Arthur packets.

\begin{thm}\label{thm intersections of local Arthur packets}
    Let $\EE\in\Eseg(G_n)$ be such that $\pi(\EE)\neq 0.$ Then the following hold.
    \begin{enumerate}
        \item If $T$ is a raising operator or its inverse, then $\pi(\EE)=\pi(T(\EE)).$
        \item If $\EE'$ is an extended multi-segment for which $\pi(\EE)=\pi(\EE'),$ then $\EE$ and $\EE'$ are related by a finite composition of raising operators, their inverses, and row exchanges.
        \item $\geq_O$ defines a partial order on $\Psi(\pi).$ Moreover. there exists unique maximal and minimal elements of $\Psi(\pi)$ with respect to $\geq_O.$ We denote these elements by $\psi^{max}(\pi)$ and $\psi^{min}(\pi)$, respectively.
    \end{enumerate}
\end{thm}

The above theorem was proven for quasi-split symplectic and odd special orthogonal groups in \cite[Theorems 1.4, 1.7]{HLL22} and quasi-split even orthogonal groups in \cite[Theorems 6.13, 6.15, 7.3]{HLL25}. With the above theorem in mind, for $\EE,\EE'\in\Eseg(G_n),$ we write $\EE\sim\EE'$, and say that they are (strongly) equivalent, if $\EE$ and $\EE'$ are related by a finite composition of raising operators and their inverses.

 Let $\EE\in\Eseg(G_n)$ and write $\EE=\cup_{\rho} \{ ([A_i,B_i]_{\rho},l_i,\eta_i)\}_{i \in (I_{\rho},>)}.$ We say that $\EE$ is non-negative if $B_i\geq 0$ for any $i\in I_{\rho}$ and for any $\rho.$
We now state the non-vanishing theorem which was proven for $G_n=\Sp_{2n}(F)$ in \cite[Theorems 3.6, 4.4]{Ato20b} and for $G_n=\OO_{2n}^\pm(F)$ in \cite[Theorem 6.17]{HLL25}.

\begin{thm}\label{thm non-vanishing}
Let $\EE\in\Eseg(G_n)$ such that for any $\rho$, if there exists $i \in I_{\rho}$ with $B_{i}<0$, then the admissible order on $I_{\rho}$ satisfies ($P'$).
\begin{enumerate}
    \item [(i)] Write $\EE=\cup_{\rho} \{ ([A_i,B_i]_{\rho},l_i,\eta_i)\}_{i \in (I_{\rho},>)}.$
    We have $\pi(\EE)\neq 0$
 if and only if $ \pi(sh^d (\EE))\neq 0$ for any $d \gg 0$ such that $sh^{d}(\EE):= \cup_{\rho}\{ ([A_i+d,B_i+d]_{\rho}, l_i, \eta_i) \}_{i \in (I_\rho,>)}$ is non-negative, and the following condition holds for all $\rho$ and $i\in I_{\rho}$
    \[ (\ast) \ \ \ \ \ \ \ \ \ B_i+l_i \geq \begin{cases}  0 & \text{ if }B_i \in \Z, \\
    \frac{1}{2} & \text{ if } B_i \not\in \Z \text{ and }\eta_i= (-1)^{\alpha_i+1},\\
        -\frac{1}{2} & \text{ if } B_i \not\in \Z \text{ and }\eta_i= (-1)^{\alpha_i},
        \end{cases} \]
        where 
        \[ \alpha_i:= \sum_{j < i }A_j+B_j+1. \]
    \item [(ii)]If $\EE$ is non-negative, then $\pi(\EE)\neq 0$ if and only if any adjacent pair $(i,j,\gg)$ satisfies the combinatorial conditions of \cite[Proposition 4.1]{Ato20b}.
\end{enumerate}
\end{thm}

Theorem \ref{thm non-vanishing} determines provides a method to determine when $\pi(\EE)\neq 0$ purely combinatorially. We say that $\EE\in\Rep(G_n)$ if $\pi(\EE)\neq 0$. Similarly, we let $\VRep(G_n)$ denote the set of $\mathcal{F}=\cup_{\rho'}\EE_{\rho'}\in \Vseg(G_n)$ such that $\EE\in\Rep(G_n).$ Applying Theorem \ref{thm non-vanishing} to $\EE$, we have that $\mathcal{F}\in \VRep(G_n)$ if and only if $\mathcal{F}$ formally satisfies the combinatorial nonvanishing conditions in Theorem \ref{thm non-vanishing}. Furthermore, we set  $\VRep_\rho(G_n)=\{\EE_\rho \ | \ \EE\in \VRep(G_n)\}$, $\VRep^\mathbb{Z}(G_n)=\VRep(G_n)\cap\Vsegz(G_n)$, and $\VRep_\rho^\mathbb{Z}(G_n)=\VRep_\rho(G_n)\cap\Vsegz(G_n)$. Again, we have that $\mathcal{E}\in \VRep_\rho(G_n)$ if and only if $\mathcal{E}$ formally satisfies the combinatorial nonvanishing conditions in Theorem \ref{thm non-vanishing}. Also, by the good parity condition, $\VRep_\rho^\mathbb{Z}(G_n)$ may be empty. We adopt the convention that $\rho$ is chosen so that $\VRep_\rho^\mathbb{Z}(G_n)$ is nonempty.

\subsection{Tempered representations}

In this subsection, we recall some results related to tempered representations. Let $\Pi_{temp}(G)$ denote the subset of $\Pi(G)$ consisting of tempered representations. We begin by recalling the parameterization of $\Pi_{temp}(G)$ via the local Langlands correspondence due to Arthur.

A local Arthur parameter $\psi\in\Psi^+(G)$ is called tempered if $\psi\in\Psi_{gp}(G)$ and $\psi$ is trivial on the second $\SL_2(\BC),$ i.e., if $\psi=\bigoplus_{i=1}^r \rho_i\otimes S_{a_i}\otimes S_{b_i},$ then $b_i=1$ for any $i=1,\dots,r.$ We let $\Psi_{temp}(G)$ denote the subset of $\Psi^+(G)$ consisting of tempered local Arthur parameters. 

\begin{thm}[{\cite[Theorem 1.5.1]{Art13}}]\label{thm Arthur tempered}
    We have that
    \[\Pi_{temp}(G)=\bigcup_{\psi\in\Psi_{temp}(G)} \Pi_\psi.\]
    Moreover, if $\psi_1,\psi_2\in\Psi_{temp}(G)$ and $\psi_1\neq\psi_2$, then $\Pi_{\psi_1}\cap\Pi_{\psi_2}=\emptyset.$
\end{thm}
In other words, local Arthur packets associated to tempered local Arthur packets partition  $\Pi_{temp}(G)$. 
Theorem \ref{thm red to gp} reduces the computation of tempered local Arthur packets to their computation for tempered Arthur parameters of good parity.
With Theorem \ref{thm Arthur packets extended multi-segment parameterization} in mind, we say that $\EE\in\Rep(G)$ is tempered if $\psi_\EE$ is tempered. We remark that $\EE\in\Eseg(G)$ is tempered if and only if
\[
\EE=\cup_\rho\{([A_i,A_i]_\rho,0,\eta_i)\}_{(i\in I_\rho, >)}\in\Eseg(G),
\]
where for any $i,j\in I_\rho$ with $A_i=A_j$, we have $\eta_i=\eta_j$.  In this setting, for $i\in I_\rho$ we write $\eta_\rho(A_i):=\eta_i$ as a shorthand. Note that it is possible for $\pi(\EE)$ to be tempered even if $\EE$ is not tempered. Our goal is that given $\pi=\pi(\EE)$ for some $\EE\in \Rep^\mathbb{Z}(G_n)$, determine all $\EE'\in \Rep^\mathbb{Z}(G_n)$ such that $\pi(\EE')=\pi(\EE).$

\begin{rmk}\label{rmk Arthur packet obstruction}
    The parameterization of local Arthur packets using extended multi-segments and corresponding theory of intersections is expected more generally. Fixing a quasi-split form $G_n^+$ of a symplectic, odd special orthogonal group, or even orthogonal group, we let $G_n^-$ denote its nontrivial pure inner form (if it exists). We then set $\epsilon_{G_n}=\pm1$ if $G_n=G_n^\pm.$ The theory is then expected to be largely analogous; however, we must also impose that when $G_n\in\{G_n^\pm\}$ is not quasi-split, the local Arthur parameters should be taken to be relevant for $G_n.$ 
    Similar results would also be expected for pure inner forms of quasi-split unitary groups, general spin groups (or general pin groups), and metaplectic groups.

    A crucial roadblock in extending the theory is that determining the existence of local Arthur packets for non-quasi-split groups currently remains an open problem.
\end{rmk}

\section{Statement of results}
\label{sec-results}

Given $\EE\in \VRep_\rho^\mathbb{Z}(G_n),$ we often write $\EE= \{([A_i, B_i], \ell_i, \eta_i)\}_{i \in I}$ where we implicitly identify $I=\{1,\dots,k\}$ with the order $1<2<\cdots<k.$ That is, we omit $\rho$ in the notation and identify the admissible order with the usual order.

For convenience we adopt the following definitions. First, we record the first and last columns of the symbols (see Example \ref{exmp symbol}) attached to extended multi-segments as follows. 

\begin{defn}
\label{def-starts-ends}
    Let $\EE\in \VRep_\rho^\mathbb{Z}(G_n)$. We let $c_{\min} = \min_{r \in \EE} B(r)$ and $c_{\max} = \max_{r \in \EE} A(r)$. We say that $\EE$ \emph{starts at column} $c_{\min}$ and that $c_{\min}$ is the \emph{first column}, and similarly that $\EE$ \emph{ends at column} $c_{\max}$ and that $c_{\max}$ is the \emph{last column}. 
\end{defn}

We are particularly interested in the case that $\EE\in \VRep_\rho^\mathbb{Z}(G_n)$ is tempered, so that for any row $r\in\EE$, we have that $\supp(r)$ is a singleton. In this case, we also keep track of the multiplicities of the columns as follows.

\begin{defn}
\label{def-multiplicity}
    Let $\EE\in \VRep_\rho^\mathbb{Z}(G_n)$, and fix $c\in\Z$. We denote by $m_c$ the \emph{multiplicity of $c$ in $\EE$}, which is defined to be the number of times $([c,c], 0, \eta)$ appears in $\EE$ for any $\eta$.
\end{defn}

In our later results and conjectures, we will make use of a decomposition of $\EE$ into ``blocks'' which form the building blocks of our arguments.

\begin{defn}
\label{def-block}
    Let $\EE = \{([A_i, B_i], \ell_i, \eta_i)\}_{i \in I}\in \VRep_\rho^\mathbb{Z}(G_n)$ be a tempered extended multi-segment (so that $A_i = B_i$ and $\ell_i = 0$). A \emph{block} is a multi-subset $\mathcal{B} = \{([A_i, B_i], \ell_i, \eta_i)\}_{i \in J}$ with $J \subset I$ such that
    \begin{itemize}
        \item if $([A, A], 0, \eta_1)\in\mathcal{B}$ and $([A+1, A+1], 0, \eta_2)\in\mathcal{B}$ then $\eta_1 = - \eta_2$,
        \item if for some $\eta\in\{\pm1\}$, we have $([A, A], 0, \eta)\in\mathcal{B}$ and $([A+2, A+2], 0, \eta)\in\mathcal{B}$, then $([A+1, A+1], 0, -\eta)\in\mathcal{B}$,
        \item if $([A, A], 0, \eta)\in\mathcal{B}$ then it appears with an odd multiplicity,
    \end{itemize}
    and $\mathcal{B}$ is maximal for these properties (i.e. if $\mathcal{B}' = \{([A_i, B_i], \ell_i, \eta_i)\}_{i \in J'}$ with $J' \supset J$ also satisfies these conditions, then $J = J'$).

    We say that $\mathcal{B}$ is an \emph{almost-block} if it satisfies the first two conditions and satisfies the third condition for all $([A, A], 0, \eta)$ except possibly when $A$ is maximal among all $([A, A], 0, \eta)$ appearing in $\mathcal{B}$.

    More generally, we say that $\mathcal{B}$ is a block (or almost-block) if there exists some $\EE\in  \VRep_\rho^\mathbb{Z}(G_n)$ for which $\mathcal{B}$ is a block. We let $\Block_\rho(G)$, resp. $\ABlock_\rho(G)$ denote the set of all blocks, resp. almost-blocks.
\end{defn}

It is straightforward to check the following fact.

\begin{lemma}\label{lem-unique-block-decomp}
    Any tempered  $\EE\in \VRep_\rho^\mathbb{Z}(G_n)(G)$ has a unique decomposition into disjoint blocks.
\end{lemma}

We will denote this decomposition $\BB_1 \cup \dots \cup \BB_k$, where each of the $\BB_i$ lie in $\VRep_\rho^\mathbb{Z}(G_n)(G)$. Then $\EE = \BB_1 \cup \dots \cup \BB_k$ consists of the rows of $\BB_1$, followed by the rows of $\BB_2$, and so on. Note that the decomposition is ordered, in the sense that $\BB_1 \cup \BB_2 \neq \BB_2 \cup \BB_1$. 

We define the remove column operator.

\begin{defn}
    Let $\EE\in  \VRep_\rho^\mathbb{Z}(G_n)$ be tempered. We define the \emph{remove column} operator $\rc_k$ which removes the $k$th column of $\EE$. More precisely, $\rc_k(\EE)$ consists of all the extended segments of $\EE$ except those of the form $([k, k], 0, \eta)$ for any $\eta$. If no such extended segments exist, then $\rc_k(\EE) = \EE$.
\end{defn}

Let $\mathcal{B}\in\ABlock(G)$. We let $\Psi(\pi(\mathcal{B}))$ denote the set of formal local Arthur parameters $\psi_{\mathcal{F}}$ where $\mathcal{F}\in\VRep_\rho^\mathbb{Z}(G_n)$ is such that $\mathcal{F}$ and $\mathcal{B}$ are related by a  finite composition of raising operators and their inverses. We determine the cardinality of this set in the following theorem.

\begin{thm}[count for blocks]
\label{thm-count-block-temp}
    Let $\mathcal{B}$ be a block and define
    \begin{align*}
        \BB' &:= \rc_{c_{\max}}(\BB) \\
        \BB'' &:= \rc_{c_{\max}-1}(\BB')
    \end{align*}
    
    If $c_{min} = 0$, then 
    \begin{equation} 
    \label{eq first recursion}
    |\Psi(\pi(\mathcal{B}))| = \begin{cases} 3 |\Psi(\pi(\mathcal{B}'))| & \text{if } m_{c_{\max}-1} = 1 \\  4 |\Psi(\pi(\mathcal{B}'))| - |\Psi(\pi(\mathcal{B}''))| & \text{if } m_{c_{\max}-1} = 3, 5, \dots. \end{cases}
    \end{equation}
    On the other hand, if $c_{min} > 0$, then 
    \begin{equation}
    \label{eq second recursion}
    |\Psi(\pi(\mathcal{B}))| = \begin{cases} 2 |\Psi(\pi(\mathcal{B}'))| & \text{if } m_{c_{\max}-1} = 1 \\ 3 |\Psi(\pi(\mathcal{B}'))| - |\Psi(\pi(\mathcal{B}''))| & \text{if } m_{c_{\max}-1} = 3, 5, \dots. \end{cases}
    \end{equation}
\end{thm}

Next we state how the numbered of local Arthur parameters containing a fixed tempered representation is determined by its block decomposition.

\begin{thm}[count for tempered representations]
\label{thm-count-temp}
    Let $\EE\in\VRep_\rho^\mathbb{Z}(G_n)$ be tempered. Suppose that $\EE$ decomposes into blocks $\mathcal{B}_1, \dots, \mathcal{B}_k$ in that order. Then \[|\Psi(\pi(\EE))| = |\Psi(\pi(\BB_1))| \cdot \prod_{i=2}^k |\Psi(\pi(sh^1(\BB_i)))|.\]
\end{thm}

We will prove the above theorem in Theorem \ref{thm-count_for_tempered-in-paper}.

\begin{rmk}
    Note that we can compute $|\Psi(\pi(\BB_1))|$ and $|\Psi(\pi(sh^1(\BB_i)))|$ using Theorem \ref{thm-count-block-temp}, so the above theorems together give a complete computation of $|\Psi(\pi(\EE))|$ for any tempered $\EE\in\VRep_\rho^\mathbb{Z}(G_n)$.
\end{rmk}

Note that if $\EE\in\Rep^\Z(G)=\Rep(G)\cap\Vseg^\Z(G),$ then $\EE$ has a unique decomposition $\EE=\EE_{\rho_1}\cup\cdots\cup\EE_{\rho_k}$ where $\rho_i\neq\rho_j$ for any $i\neq j.$ Furthermore, we have $\Psi(\pi(\EE))=\prod_{i=1}^k\Psi(\pi(\EE_{\rho_i})).$ From this observation and Theorem \ref{thm-count-temp}, we obtain a formula for computing $\Psi(\pi(\EE))$ for any tempered $\EE\in\Rep^\Z(G)=\Rep(G)\cap\Vseg^\Z(G)$. That is, Theorem \ref{thm-count-temp} (plus Theorem \ref{thm red to gp}) implies Theorem \ref{thm tempered count intro}.

\begin{rmk}\label{rmk main thms expectation}
    With Remark \ref{rmk Arthur packet obstruction} in mind, we would expect Theorems \ref{thm-count-block-temp} and \ref{thm-count-temp} to hold more generally.
\end{rmk}

\section{Basic notions and lemmas}\label{sec basic notions and lemmas}

In this section we introduce some basic notions and lemmas which will be helpful in the proofs of our results. We begin by recording several explicit relations on the operators which will be useful later.

\begin{lemma}\label{lemma Alex}
    Let $\EE\in\Eseg(G_n)$ and suppose that $\pi(\EE)\neq 0.$ If $\frac{a_i-b_i}{2}<0$ for some $i \in I_{\rho}$, then we also require that the admissible orders below on $I_\rho$ satisfy $(P').$  Then the following hold:
    \begin{enumerate}
    \item Let $i,j,k,l\in I_{\rho}$. Then $(R_{i}\circ R_j) (\EE)=(R_{k}\circ R_l) (\EE)$ provided that the resulting admissible orders agree.
        \item Let $i,j\in I_{\rho}$. Then $(dual\circ R_i)(\EE)=(R_j\circ dual)(\EE)$ provided that the resulting admissible orders agree.
        \item Let $i,j\in I_{\rho}$. Then $(ui_{i}\circ R_j)(\EE)=(R_j \circ ui_i)(\EE)$ provided that the resulting admissible orders agree.
    \end{enumerate}
\end{lemma}

\begin{proof}
    The proof of these claims follow the same pattern. For brevity, we only give the details of the proof of Part (2) that $(dual\circ R_i)(\EE)=(R_j\circ dual)(\EE)$  provided that the resulting admissible orders are the agree.

    By Lemma \ref{lemma dual} and Theorem \ref{thm intersections of local Arthur packets}(1), we have that $\pi((dual\circ R_i)(\EE))=\pi((R_j\circ dual)(\EE)).$ Moreover, $\supp((dual\circ R_i)(\EE))=\supp((R_j\circ dual)(\EE)).$  If the resulting admissible orders are the agree, then we obtain that $(dual\circ R_i)(\EE)=(R_j\circ dual)(\EE)$ from Corollary \ref{cor same support}. This completes the proof of Part (2).
\end{proof}

We provide a further extension of Lemma \ref{lemma Alex}(1).

\begin{lemma}[commutativity of row exchanges]
\label{lem-comm}
    Let $\EE = (r, r_1, r_2)\in \VRep(G_n)$. Assume further that the orders the supports of $r, r_1, r_2$ contain or are contained in each other. Then exchanging $r$ to the third row and then exchanging $r_1$ and $r_2$ gives the same result as exchanging $r_1$ and $r_2$ first, and then exchanging $r$ to the third row. Formally, $R_1(R_2(R_1(\EE))) = R_2(R_1(R_2(\EE)))$.
\end{lemma}
\begin{proof}
    Since $\pi(\EE)\neq 0,$ by Theorem \ref{thm non-vanishing}, we have $\pi(sh^d(\EE))\neq 0$ for some $d\gg 0$ such that $sh^d(\EE)$ is non-negative. We observe from the definitions that the row exchange and shift operators commute. Thus, by Lemma \ref{lemma Alex}(1), we obtain that \begin{align*}
    sh^d(R_1(R_2(R_1(\EE)))) =R_1(R_2(R_1(sh^d(\EE)))) &= R_2(R_1(R_2(sh^d(\EE))))\\ &=sh^d(R_2(R_1(R_2(\EE)))).
    \end{align*}
    Therefore, we obtain that $R_1(R_2(R_1(\EE))) = R_2(R_1(R_2(\EE)))$.
\end{proof}

We also verify that certain operators are ``local'' operators.

\begin{lemma}
\label{lem-local}
    Suppose that $\EE\in\VRep_\rho(G_n)$. 
    Then the operations $ui$ and $dual \circ ui \circ dual$ are ``local'' operations. More specifically, if a certain row is not involved in the union-intersection or any row exchanges, then it is fixed by these operations.
\end{lemma}
\begin{proof}
    For the operation $ui$ the result follows immediately from Definition \ref{def ui}. In particular, the result is also true for the inverse of $ui.$ Suppose now that we consider the operator $dual \circ ui \circ dual$, where the $ui$ is not of type 3'. Then $dual \circ ui \circ dual$ is the inverse of a $ui$ operator (of the same type) by \cite[Corollary 5.6]{HLL22} and thus the result follows. For $dual \circ ui \circ dual$ of type 3', the result follows from direct calculation (note that the rows to which the $ui$ is applied must satisfy that $a_i+b_i\equiv a_j+b_j \mod 2$ since the $ui$ is of type 3').
\end{proof}

Next we give a lemma on a certain relation between the row exchange and union-intersection operators.

\begin{lemma}\label{lemma ui and row exchange commute}
    Let $\EE=(r_1,r_2,r_3)\in\VRep_\rho(G_n)$ and suppose that $\supp(r_1)\supseteq\supp(r)$ for any $r\in\EE$ and that $ui_{2}$ is applicable on $\EE.$ Let $\FF=ui_{2}(\EE)=(r_1,s_2,s_3)$ and consider the row exchange of $r_1$ to the bottom row in both $\EE$ and $\FF.$ We write $\EE'=(r_2',r_3',r_1')$ and $\FF'=(s_2',s_3',r_1'')$ for the resulting virtual extended multi-segments. Then $r_1'=r_1''.$
\end{lemma}

\begin{proof}
    The proof is a direct consequence of Definitions \ref{def row exchange} and \ref{ui def}. 
\end{proof}

The following result which will be used for certain reductions in several technical arguments later.

\begin{cor}
\label{cor Alex}
    Suppose $r$ is a row in $\mathcal{E}\in\VRep_\rho^\mathbb{Z}(G_n)$ immediately followed by some consecutive rows consisting a sub-multi-segment $\mathcal{E}_1 \subset \mathcal{E},$ $\mathcal{E}_2$ is an extended multi-segment equivalent to $\EE_1$. 
    We suppose further that
    $\supp(r)\supseteq\supp(s)$ for any $s\in\EE_1\cup\EE_2$.    
    Then, the result after $r$ is exchanged with $\EE_1$ is the same as the result after $r$ is exchanged with $\EE_2$. 
\end{cor}

\begin{proof}
    Before proving the claim, we first recall an operator, called phantom (dis)appearing, defined by Atobe (\cite[Definition 3.4]{Ato23}). Let $k$ be an integer and  $P_k$ denote the operator  which formally attaches $([k-1,-k]_\rho,k,1)$ (allowed only if $k>0$) or $([k-\frac{1}{2},-k-\frac{1}{2}]_\rho,k,1)$ (allowed only if $k\geq0$) as the first row of a virtual extended multi-segment. The choice is determined by the good parity condition. For split symplectic and odd special orthogonal groups, Atobe showed that any two equivalent extended multi-segments are related by a finite composition of $P_k$'s, union-intersections, row exchanges, and their inverses(\cite[Theorem 1.4]{Ato23}). This result is equivalent with Theorem \ref{thm intersections of local Arthur packets} (see \cite[Remark 10.5]{HLL22}) and so the same result holds for quasi-split even orthogonal groups.

    We return to the proof of the claim. For simplicity, we may assume that $\EE=\{r\}\cup\EE_1.$ Let $\EE'=\EE_1'\cup\{r'\}$ denote the resulting virtual extended multi-segment obtained by row exchanging $r$ with every segment in $\EE_1.$ Similarly, we consider $\{r\}\cup\EE_2$ and let $\EE_2'\cup\{r''\}$ denote the resulting virtual extended multi-segment obtained by row exchanging $r$ with every segment in $\EE_2.$ We must show that $r'=r''.$

    Furthermore, without loss of generality, we may assume that $\EE_2$ is obtained from $\EE_1$ by either a union-intersection or the composition of $P_k$ and a union-intersection. The former case is taken care of by Lemma \ref{lemma ui and row exchange commute}  and so we proceed with the latter case.

    Let $s$ denote the corresponding formally added element $([k-1,-k]_\rho,k,1)$ (recall that $A(r)\in\Z$ and so this must be the case). We note that formally row exchanging $r$ and $s$ leaves both segments unchanged. Indeed, the row exchange of $r$ with $s$ is always in Case 1(b) or 1(c) of Definition \ref{def row exchange} and formally $b(s)-2l(s)=0$. Consequently, the effect of row exchanging $r$ to the last row of $\{r\}\cup\{s\}\cup\EE_1$ is the same as row exchanging $r$ to the last row of $\{r\}\cup\EE_1$. The claim then follows from a similar argument as in the proof of Lemma \ref{lemma ui and row exchange commute}.
\end{proof}

\subsection{The alternating sign condition} In this section, we define and study the alternating sign condition.
First, we give a definition which counts the number of circles in the symbol associated to an extended segment (see Example \ref{exmp symbol}).

\begin{defn}
    Let $r = ([A, B]_\rho, l, \eta)$ be a row. Then we denote the number of circles in $r$ by $C(r)$, which is given by \[C(r) := b - 2l.\] For $\EE$ a virtual extended multi-segment, we let $C(\EE)$ be the total number of circles in $\EE$, which is given by \[C(\EE) := \sum_{r \in \EE} C(r).\]
\end{defn}

We note that $C(r)+1\equiv A-B \mod 2.$ Next, we attach a sign to $\EE.$

\begin{defn}
\label{def-sign-ems}
    We let the \emph{sign} $\eta(\EE)$ of an extended multi-segment $\EE$ be \[\eta(\EE) := \eta(r_1),\]
    where $r_1$ is the first row in $\EE$ in the admissible order.
\end{defn}

The alternating sign condition is an important condition on two rows which affects how they interact under the certain operations.

\begin{defn}
\label{def-alternating-sign-condition}
    Let $r_1, r_2$ be rows with $r_1 < r_2$ in some admissible order. We say that $r_1$ and $r_2$ (in that order) satisfy the \emph{alternating sign condition} if \[\eta(r_2) = (-1)^{C(r_1)} \eta(r_1).\]
    We say that an extended multi-segment $\EE$ is \emph{alternating} if any two consecutive rows of $\EE$ satisfy the alternating sign condition.
\end{defn}

In terms of symbols (see Example \ref{exmp symbol}), the alternating sign condition says that the last circle of the first row $r_1$ and the first circle of the second row $r_2$ have opposite signs.

\begin{defn}
\label{def-last-circle}
    Let $r$ be a row. We call $(-1)^{C(r)-1} \eta(r)$ the \emph{sign of the last circle of $r.$}
\end{defn}

Note that the above definition makes sense even if, in a picture, $r$ does not have any circles. We now seek to prove some properties about rows that satisfy the alternating sign condition. First, we need the following lemma.

\begin{lemma}
\label{Dual sign}
    Let $\EE\in\VRep_\rho^\mathbb{Z}(G_n)$ and $r\in\EE$ be a row. Let $\widehat{r}$ be the image of $r$ under $dual$ (see Definition \ref{def dual}). Then \[\eta(\widehat{r}) = (-1)^{C(\EE) - C(r)} \eta(r).\]
\end{lemma}
\begin{proof}
    Suppose $r = r_i$ is the $i$th row of $\EE$. Since $B \in \mathbb{Z}$ for every row, by Definition \ref{def dual}, $\eta(\widehat{r}) = (-1)^{\alpha_i + \beta_i} \eta(r),$ where $\alpha_i = \sum_{j < i} a_j$ and $\beta_i = \sum_{j < i} b_j.$ Since $\EE$ is integral, for any $j$ we have $a_j \equiv b_j \bmod{2}$, so \[\alpha_i + \beta_i \equiv \sum_{j \neq i} b_j \equiv \sum_{j \neq i} C(r_j) \equiv C(\EE) - C(r_i) \bmod{2}. \qedhere \]
\end{proof}

We use the above lemma to verify that $dual$ preserves the alternating sign condition.

\begin{lemma}
\label{Alternating dual} Let $\EE\in\VRep_\rho^\mathbb{Z}(G_n)(G).$
    \begin{enumerate}
        \item If $r_1 < r_2$ are two rows of $\EE$ satisfying the alternating sign condition, then the rows $\widehat{r_2} < \widehat{r_1}$ in $dual(\EE)$ also satisfy the alternating sign condition. In particular, if $\EE$ is alternating, then so is $dual(\EE).$
        \item If $\EE$ is alternating, then $\eta(\EE) = \eta(dual(\EE)).$
    \end{enumerate}
\end{lemma}

\begin{proof}
By Lemma \ref{Dual sign} and the fact that $r_1$ and $r_2$ satisfy the alternating sign condition, we obtain
    \begin{align*}
        \eta(\widehat{r_1}) = (-1)^{C(\EE) - C(r_1)} \eta(r_1) 
        &= (-1)^{C(\EE) - C(r_1)} \left( (-1)^{C(r_1)} \eta(r_2) \right) \\
        &= (-1)^{C(\EE)} \eta(r_2) \\
        &= (-1)^{C(\EE)} \left( (-1)^{C(\EE) - C(r_2)} \eta(\widehat{r_2}) \right) \\
        &= (-1)^{C(r_2)} \eta(\widehat{r_2}).
    \end{align*}
    Since $C(r_2) = C(\widehat{r_2})$, this shows $\widehat{r_2}$ and $\widehat{r_1}$ satisfy the alternating sign condition, in that order. This proves Part (1).

    For Part (2), let $r$ be the first row of $\EE$ and let $s$ be the last row of $\EE$. Since $\EE$ is alternating, by repeatedly applying the alternating sign condition, we obtain $\eta(s) = (-1)^{C(\EE) - C(s)} \eta(r)$. By Lemma \ref{Dual sign}, we have $\eta(\widehat{s}) = (-1)^{C(\EE) - C(s)} \eta(s)$. Combining these two gives that $\eta(\widehat{s}) = \eta(r)$. Since $\widehat{s}$ is the first row of $dual(\EE)$, this shows $\eta(\EE) = \eta(dual(\EE))$.
\end{proof}

\subsection{Lemmas on row exchanges}

Recall that our goal is to classify all the extended multi-segments $\EE'$ equivalent to a tempered extended multi-segment $\EE$. One type of extended segment that appears in such extended multi-segments $\EE'$ is of the following form.

\begin{defn}
\label{def-hat}
    A \emph{hat} is an extended segment of the form $([A, B]_\rho, l, \eta)$ where $B = -l$.
\end{defn}

Much of the substance of the ensuing proofs will depend on understanding how a hat changes under certain row exchanges. The next few lemmas will be important in analyzing the behavior of hats in special cases of these types of operations.

\begin{lemma}
\label{swap down}
    Suppose $r_1 < r_2$ are consecutive rows satisfying the alternating sign condition and with $\supp(r_1) \supset \supp(r_2).$ Then, if $r_2' < r_1'$ are the images of $r_1$ and $r_2$ under a row exchange, we have:
    \begin{align*}
        (l(r_1'), \eta(r_1')) &= (l(r_1) - C(r_2), (-1)^{C(r_2)} \eta(r_1)) \\
        (l(r_2'), \eta(r_2')) &= (l(r_2), (-1)^{C(r_1)+1} \eta(r_2)).
    \end{align*}
    Moreover, $r_2' < r_1'$ fail the alternating sign condition.
\end{lemma}

\begin{proof}
    First, we calculate that \[\epsilon = (-1)^{A(r_1) - B(r_1)} \eta(r_1) \eta(r_2) = (-1)^{C(r_1) - 1} \eta(r_1) \eta(r_1) \cdot (-1)^{C(r_1)} = -1.\] So the row exchange falls into Case 1(c) of Definition \ref{def row exchange}. The formulas for $l(r_i'), \eta(r_i')$ follow immediately from the given formulas. To see that these rows fail the alternating sign condition, we note that because row exchange preserves supports, $b(r_2) = b(r_2')$, so $C(r_2) \equiv C(r_2') \bmod{2}$. Applying this fact and using the above formulas, \[\eta(r_2') = (-1)^{C(r_1) + 1} \eta(r_2) = -\eta(r_1) = -(-1)^{C(r_2')} \eta(r_1'). \qedhere\]
\end{proof}

The reverse occurs when swapping the larger row up.

\begin{lemma}
\label{swap up}
    Suppose $r_1 < r_2$ are consecutive rows failing the alternating sign condition with $\supp(r_1) \subset \supp(r_2)$, and suppose that $C(r_2) \geq 2 C(r_1)$. Then, if $r_2' < r_1'$ are the images of $r_1$ and $r_2$ under a row exchange, we have:
    \begin{align*}
        (l(r_1'), \eta(r_1')) &= (l(r_1), (-1)^{C(r_2)+1} \eta(r_1)) \\
        (l(r_2'), \eta(r_2')) &= (l(r_2) + C(r_1), (-1)^{C(r_1)} \eta(r_2)).
    \end{align*}
    Moreover, $r_2'$ and $r_1'$ satisfy the alternating sign condition.
\end{lemma}

\begin{proof}
    Since $r_1$ and $r_2$ fail the alternating sign condition, we obtain that $\epsilon = 1$ (see Definition \ref{def row exchange} for notation). Moreover, since $\supp(r_1) \subset \supp(r_2)$ and $C(r_2) \geq 2C(r_1)$, we fall into Case 2(b) of Definition \ref{def row exchange}. The formulas follow immediately from the given formulas. Finally, $r_2'$ and $r_1'$ satisfy the alternating sign condition because \[\eta(r_1') = (-1)^{C(r_2) + 1} \eta(r_1) = (-1)^{C(r_1) + C(r_2)} \eta(r_2) = (-1)^{C(r_2')} \eta(r_2'). \qedhere\]
\end{proof}

We generalize Lemmas \ref{swap down} and $\ref{swap up}$ to analyze what happens when we swap a row with a series of consecutive rows satisfying the alternating sign condition.

\begin{lemma}
\label{big swap down}
    Suppose that $r_0 < r_1 < \cdots < r_k$ are consecutive rows satisfying the alternating sign condition with $\supp(r_0) \supset \supp(r_i)$ for $i > 0$. Let $r_1' < \dots < r_k' < r_0^{(k)}$ be their images after $r_0$ is row exchanged $k$ times. Then we have
    \begin{align*}
        (l(r_0^{(k)}), \eta(r_0^{(k)})) &= \left(l(r_0) - \sum_{i=1}^k C(r_i), (-1)^{\sum_{i=1}^k C(r_i)} \eta(r_0) \right), \\
        (l(r_i'), \eta(r_i')) &= (l(r_i'), (-1)^{C(r_0)+1} \eta(r_i)) \text{ for } i >0.
    \end{align*}
    Moreover, after the row exchanges every pair of adjacent rows satisfies the alternating sign condition except $r_k'$ and $r_0^{(k)}$.
\end{lemma}

\begin{proof}
    We proceed by induction on the number of times $r_0$ is exchanged. The base case is when $r_0$ is not exchanged, which is clear. Suppose the statement holds after $r_0$ has been exchanged $i$ times, giving $r_1' < \dots < r_i' < r_0^{(i)} < r_{i+1} < \dots < r_k$. Then since by inductive hypothesis $(r_i', r_0^{(i)})$ fails the alternating sign condition and $r_i$ and $r_{i+1}$ satisfy the alternating sign condition, we have
    \begin{align*}
        \eta(r_0^{(i)}) = (-1)^{C(r_i') + 1} \eta(r_i') 
        &= (-1)^{C(r_i) + 1} \left( (-1)^{C(r_0) + 1} \eta(r_i) \right) \\
        &= (-1)^{C(r_0^{(i)})} \eta(r_{i + 1}).
    \end{align*}
    So $r_0^{(i)}$ and $r_{i+1}$ satisfy the alternating sign condition. Thus, we can apply Lemma \ref{swap down}. Exchanging these rows gives $l(r_0^{(i+1)}) = l(r_0^{(i)}) - C(r_{i+1})$ and $\eta(r_0^{(i+1)}) = (-1)^{C(r_{i + 1})} \eta(r_0^{(i)})$
    which are the desired values (after applying the induction hypothesis). Meanwhile $l(r_i') = l(r_i)$ and \[\eta(r_{i + 1}') = (-1)^{C(r_0') + 1} \eta(r_{i+1}) = (-1)^{C(r_0) + 1} \eta(r_{i+1})\] as desired. Finally, to show the claims about the alternating sign condition, we observe that $r_i'$ and $r_{i+1}'$ are the same as $r_i$ and $r_{i+1}$ respectively except for their sign, which are both multiplied by the same quantity $(-1)^{C(r_0)+1}$. Since $r_i$ and $r_{i+1}$ satisfy the alternating sign condition, so do $r_i'$ and $r_{i+1}$. Also, by Lemma \ref{swap down}, $r_{i+1}'$ and $r_0^{(i+1)}$ do not alternate. This completes the inductive step.
\end{proof}

\begin{lemma}
\label{big swap up}
    Suppose that $r_1 < \cdots < r_k < r_0$ are consecutive rows such that every pair of consecutive rows satisfies the alternating sign condition except $r_k$ and $r_0$. Suppose further that $\supp(r_0) \supset \supp(r_i)$ for $i > 0$, and that $C(r_0) \geq 2 \sum_{i = 1}^k C(r_i)$. Let $r_0^{(k)} < r_1' < \cdots < r_k'$ be their images after $r_0$ is exchanged up $k$ times. Then:
    \begin{align*}
        (l(r_0^{(k)}), \eta(r_0^{(k)})) &= \left(l(r_0) + \sum_{i=1}^k C(r_i), (-1)^{\sum_{i=1}^k C(r_i)} \eta(r_0) \right) \\
        (l(r_i'), \eta(r_i')) &= (l(r_i), (-1)^{C(r_0) + 1} \eta(r_i)) \text{ for } i > 0.
    \end{align*}
    Moreover, after the row exchanges every pair of adjacent rows satisfies the alternating sign condition.
\end{lemma}
\begin{proof}
    The proof is exactly analogous to the proof of Lemma \ref{big swap down}, except that instead of applying Lemma \ref{swap up}, we apply Lemma \ref{swap down}.
\end{proof}

Finally, we prove the following two lemmas, which demonstrate cases in which swapping a row $h$ with two rows $h_1$ and $h_2$ leaves $h$ unchanged. These results will allow us to simplify future proofs by removing such pairs $h_1$ and $h_2$ from a virtual extended multi-segment.

\begin{lemma}
\label{lem-multiplicity-cancel}
Suppose $h < h_1 < h_2$ are consecutive rows and that $\supp(h) \supset \supp(h_1)$, $\supp(h_2)$. Further suppose $C(h) \neq 0$, and that $C(h_1) = C(h_2) = 1$, $l(h_1) = l(h_2),$ and $\eta(h_1) = \eta(h_2).$ Then $h$ is unchanged after being exchanged with $h_1$ and $h_2$. Moreover, $h_1$ and $h_2$ are unchanged except that their sign is multiplied by $(-1)^{C(h) + 1}$.
\end{lemma}
\begin{proof}
    Note that since $\supp(h) \supset \supp(h_1)$, $\supp(h_2)$ and the support of a row is unchanged by row exchange, both row exchanges will be in Case 1 of Definition \ref{def row exchange}. We let $h_1'<h'<h_2'$ and $h_1'<h_2'<h''$ denote the effects of the row exchanges.
    The second claim about the sign of $h_1'$ and $h_2'$ follows immediately from the definition. Note that since $C(h_1)=1,$ we have that $2l(h_1)+1=2b(h_1)$ and similarly for $h_2.$
    For the first claim, we divide into three cases. 
    
    First suppose that the first row exchange is in Case 1(a) of Definition \ref{def row exchange}. Then $l(h') = b(h) - l(h) - 1$ and $\eta(h') = \eta(h)$. Since $A(h') = A(h)$, $B(h') = B(h)$, and $\eta(h) = \eta(h')$, and the first row exchange had $(-1)^{A(h) - B(h)} \eta(h) \eta(h_1) = 1$ by assumption, we obtain that the second row exchange also has $\epsilon = 1$. Now note that \[C(h') = C(h) + 2(l(h) - l(h')) = C(h) + 2 (2 l(h) - b(h) + 1) = 2 - C(h).\] Since the first row exchange was in Case 1(a), $C(h) < 2$. Since $C(h) \neq 0$, we must have $C(h) = 1$, and hence the second row exchange will be in Case 1(a). So $l(h'') = b(h') - l(h') - 1 
        = b(h) - (b(h) - l(h) - 1) - 1 
        = l(h)$ and $\eta(h'') = \eta(h') = \eta(h).$
    This completes the proof in this case.
    
    Second suppose that the first row exchange is in Case 1(b). Then $l(h') = l(h) + 1$ and $\eta(h') = -\eta(h)$. Since $\epsilon$ for the first row exchange is $1$, and $\eta(h') = -\eta(h)$, we conclude that $\epsilon$ for the second row exchange is $-1$, so it is in Case 1(c). Then $l(h'') = l(h') - 1 = l(h)$ and $\eta(h'') = -\eta(h') = \eta(h)$. This completes the proof in this case.

    Finally suppose that the first row exchange is in Case 1(c). Then $l(h') = l(h) - 1$ and $\eta(h') = - \eta(h)$. So using similar reasoning as before, $\epsilon$ for the second row exchange is $1$. However note that it is impossible for the second row exchange to be in Case 1(a), since $C(h') = C(h) + 2 \geq 2$, so we cannot have $C(h') < 2 C(h_2) = 2$. So the second row exchange is in Case 1(b), where the formulas give $l(h'') = l(h') + 1 = l(h)$ and $\eta(h'') = -\eta(h') = \eta(h)$. This completes the proof in this case and hence proves the lemma.
\end{proof}

\begin{lemma}
\label{lem-multiplicity-cancel-up}
Suppose $h_1 < h_2 < h$ are consecutive rows and that $\supp(h) \subset \supp(h_1)$, $\supp(h_2)$. Further suppose $C(h) \neq 0$, and that $C(h_1) = C(h_2) = 1$, $l(h_1) = l(h_2)$ and $\eta(h_1) = \eta(h_2).$ Then $h$ is unchanged after being exchanged with $h_1$ and $h_2$. Moreover, $h_1$ and $h_2$ are unchanged except that their sign is multiplied by $(-1)^{C(h)}$.
\end{lemma}
\begin{proof}
    The proof is exactly analogous to the proof of Lemma \ref{lem-multiplicity-cancel}, except that we are in Case 2 of Definition \ref{def row exchange} instead of Case 1.
\end{proof}

\subsection{Lemmas on hats}

The alternating sign condition helps provide an important criterion for whether consecutive hats can interact with each other.

\begin{lemma}
\label{merge}
    Given two consecutive hats $h_1 < h_2$ in $(P')$ order where $h_1 = ([A_1, -B_1], B_1, \eta_1)$ and $h_2 = ([A_2, -B_2], B_2, \eta_2),$
    it is possible to apply a nontrivial $dual \circ ui \circ dual$ to $h_1$ and $h_2$ if and only if $A_2 = B_1 - 1$ and $h_1$ and $h_2$ satisfy the alternating sign condition.
    The result has only one row and it is of the form \[h = ([A_1, -B_2], B_2, \eta_1).\]
\end{lemma}

\begin{proof}
    After applying $dual,$ the corresponding dual rows $\widehat{h_2} < \widehat{h_1}$ will be $\widehat{h_2} = ([A_2, B_2], 0, (-1)^{C(\EE) + C(h_2)}\eta_2)$ and $\widehat{h_1} = ([A_1, B_1], 0, (-1)^{C(\EE) + C(h_1)} \eta_1)$.

    Since these rows are composed entirely of circles, the only nontrivial $ui$ is that of type 3'. Since $h_1$ and $h_2$ are in $(P')$ order, $B_1 > B_2$. So such a $ui$ can be applied if and only if $B_1 = A_2 + 1$ and if $\widehat{h_2}$ and $\widehat{h_1}$ satisfy the alternating sign condition. By Lemma \ref{Alternating dual}, the second condition is equivalent to $h_1$ and $h_2$ satisfying the alternating sign condition. Applying $ui$ gives a new row
    \[\widehat{h} = ([A_1, B_2], 0, (-1)^{C(\EE) + C(h_2)} \eta_2).\]
    Dualizing gives a new row of the form
    $h = ([A_1, -B_2], B_2, \eta),$
    for some $\eta.$ 
    
    In order to calculate the new value of $\eta,$ we note that $C(h') = C(h_1') + C(h_2')$ and that dualization preserves the number of circles in each row. Therefore, $C(h') = C(h_1) + C(h_2)$ and $C(\EE) = C(dual(\EE)).$ Now, we can calculate using Lemma \ref{Dual sign} that
    \begin{align*}
        \eta = (-1)^{C(dual(\EE)) - C(\widehat{h})} \left( (-1)^{C(\EE) + C(h_2)} \eta_2 \right) 
        &= (-1)^{C(\EE) - C(h_1) - C(h_2) + C(\EE) + C(h_2)} \eta_2 \\
        &= (-1)^{C(h_1)} \eta_2 \\
        &= \eta_1. \qedhere
    \end{align*}
\end{proof}
    
\begin{defn}
    We refer to the act of performing a $dual \circ ui \circ dual$ of type 3' on two consecutive hats $h_1$, $h_2$ as described above as \emph{merging} the hats. We say two hats are \emph{mergable} if they satisfy the conditions in Lemma \ref{merge}. We denote the merged hat by $h_1 * h_2.$
\end{defn}

We observe from Lemma \ref{merge} that $C$ is additive under mergings. Formally, if $h_1$ and $h_2$ are two hats, then $C(h_1 * h_2) = C(h_1) + C(h_2).$ It follows immediately from the formulas in Lemma \ref{big swap up} that row exchanges commute with the action of merging hats. We state this formally in the following corollary.

\begin{cor}
\label{swap commutes with M}
    Let $r_1 < r_2 < r_3$ be three consecutive rows in a virtual extended multi-segment.
    \begin{itemize}
        \item Suppose $r_2$ and $r_3$ are mergable hats, and suppose $\supp(r_2), \supp(r_3) \subset \supp(r_1)$. Then the image of $r_1$ after being exchanged with $r_2$ and $r_3$ is the same as the image of $r_1$ after being exchanged with $r_2 * r_3$.
        \item Suppose $r_1$ and $r_2$ are mergable hats, and suppose $\supp(r_1), \supp(r_2) \subset \supp(r_3)$. Then the image of $r_3$ after being exchanged with $r_2$ and $r_1$ is the same as the image of $r_3$ after being exchanged with $r_1 * r_2$.
    \end{itemize}
\end{cor}

\section{Individual Blocks}
\label{sec-individual-blocks}

Throughout this section, we let $\mathcal{B}\in\Block_\rho(G)$. Recall that our goal is to determine the size of $\Psi(\pi(\mathcal{B}))$. This problem falls into two cases, namely, whether the block starts at zero or not (see Theorem \ref{thm-count-block-temp}). 

\subsection{\texorpdfstring{Constructing Virtual Extended Multi-Segments from $\mathcal{S}$-data}{}}

Let $\mathcal{B}\in\Block_\rho(G)$. We aim to give a classification of extended multi-segments that are equivalent to $\mathcal{B}$ (see Theorem \ref{block classification}). 

We begin by setting some notation. Recall from Definition \ref{def-starts-ends} that $c_{\min} = \min_{r \in \mathcal{B}} B(r)$ and $c_{\max} = \max_{r \in \mathcal{B}} A(r)$ are the the integers that $supp(\mathcal{B})$ begins and ends at. Also recall from Definition \ref{def-multiplicity} that $m_c$ denotes the number of rows in $\mathcal{B}$ with support $[c, c]$, which we refer to as the multiplicity of column $c$. We let the tuple 
$\mathcal{M}_\B = (m_{c_{\min}}, m_{c_{\min} + 1}, \dots, m_{c_{\max}})$ denote the multiplicities of all rows in $\mathcal{B}.$ If $\B$ is fixed, we will simple write $\mathcal{M}=\mathcal{M}_\B$ for brevity. Note that all the numbers in $\mathcal{M}$ are positive odd integers by Definition \ref{def-block}. More generally, we fix two integers $c_{\min},c_{\max}\in\Z_{\geq 0}$ with $c_{\max}\geq c_{\min}$. We say that a tuple \[\mathcal{M}(c_{\min}, c_{\max})=(m_{c_{\min}}, m_{c_{\min}+1},\dots,m_{c_{\max}})\in\Z^{c_{\max} - c_{\min} +1}\] is a \emph{block-tuple} if each $m_{i}$ is a positive odd integer. Note that if $\B$ is a block, then $\mathcal{M}_\B$ is a block-tuple.

\begin{defn}
\label{valid S}
    We define a \emph{valid} tuple $\mathcal{S}$ for a block-tuple $\mathcal{M}(c_{\min},c_{\max})$ to be a tuple $(\mathcal{S}_1, \dots, \mathcal{S}_k)$ of subsets of $\{c_{\min}, \dots, c_{\max}\}$ satisfying the following conditions.
    \begin{enumerate} 
        \item Each $\mathcal{S}_i \subset \{c_{\min}, \dots, c_{\max} \}$ is a nonempty set of consecutive integers.
        \item $\bigcup_i \mathcal{S}_i = \{c_{\min}, \dots, c_{\max}\}.$
        \item If $i < j$ and $s_i \in \mathcal{S}_i, s_j \in \mathcal{S}_j,$ then $s_i \leq s_j$.
        \item If $i < j$ and $c \in \mathcal{S}_i \cap \mathcal{S}_j,$ then $j - i = 1,$ $|\mathcal{S}_j| \geq 2,$ and $m_c > 1.$
        \item If $i\geq 2$, $|\mathcal{S}_i| \geq 2$, and $c = \min \mathcal{S}_i$ has $m_c > 1,$ then $c \in \mathcal{S}_{i - 1}.$
\end{enumerate}
\end{defn}

Fix a block $\B\in\Block_\rho(G_n)$ and let $\mathcal{M}=\mathcal{M}_\B.$
To any valid tuple $\mathcal{S}$ for $\mathcal{M}$, we associate a virtual extended multi-segment $\mathcal{E}(\mathcal{M}, \mathcal{S}, \eta)$ in the following manner.

\begin{defn}\label{defn E(M,S)}
    Let $\mathcal{M}= (m_{c_{\min}}, m_{c_{\min} + 1}, \dots, m_{c_{\max}})$ be a block-tuple. Given a sign $\eta \in \{\pm 1\}$ and a valid tuple $\mathcal{S}=(\mathcal{S}_1,\dots,\mathcal{S}_k)$ for $\mathcal{M},$ we define a virtual extended multi-segment $\EE(\mathcal{M}, \mathcal{S}, \eta)$ consisting of the following rows.
    \begin{enumerate}
    \item For each $\mathcal{S}_i,$ we include a row of circles (i.e. a row with $l=0$) with support 
    $[\max \mathcal{S}_i, \min \mathcal{S}_i].$
    We refer to these rows as \emph{chains}.
    \item For each $c \in \{c_{\min}, \dots, c_{\max}\},$ we add $m_c - |\{i \mid c \in \mathcal{S}_i\}|$ copies of a row of circles with support
    $[c, c].$
    We refer to these rows as \emph{multiples}.
\end{enumerate}

The rows are ordered in a $(P')$ ordering such that if $A(r) < A(r')$ then $r < r'$. In the case where $r$ and $r'$ both have $supp(r) = supp(r') = [c, c]$ but $r$ is a chain and $r'$ is a multiple, we choose the order so that $r < r'.$

The signs $\eta$ are chosen to be ``odd-alternating'': i.e., the following conditions hold.
    \begin{enumerate}
        \item $\eta(\mathcal{E}(\mathcal{M}, \mathcal{S}, \eta)) = \eta.$
        \item If $r_i$ and $r_{i + 1}$ are consecutive rows and neither is a multiple, then $r_i$ and $r_{i + 1})$ satisfy the alternating sign condition (see Definition \ref{def-alternating-sign-condition}).
        \item If $r_i$ and $r_{i + 1}$ are consecutive rows and $r_{i + 1}$ is a multiple, then $r_i$ and $r_{i + 1}$ fail the alternating sign condition.
        \item If $r_i$ and $r_{i + 1}$ are consecutive rows and only $r_{i}$ is a multiple, then $r_i$ and $r_{i+1}$ satisfy the alternating sign condition if and only if $B(r_{i + 1}) > B(r_i).$
    \end{enumerate}
\end{defn}

\begin{rmk}
    Note that the number of multiples $m_c - |\{\mathcal{S}_i \mid c \in \mathcal{S}_i\}|$ is chosen such that the number of circles in column $c$ is always equal to $m_c.$
\end{rmk}

We introduce some terminology related to the study of the virtual extended multi-segments $\EE(\mathcal{M},\mathcal{S},\eta).$

\begin{defn}
\label{def z-chain}
    If two chains $r$ and $r'$ are associated to $\mathcal{S}_i$ and $\mathcal{S}_{i + 1}$ such that $\mathcal{S}_i \cap \mathcal{S}_j \neq \emptyset,$ then we say that $r'$ is a \emph{$z$-chain}. An example is given in Figure \ref{fig: z-chain}.
\end{defn}

\begin{figure}[h]
    \centering
    $$\bordermatrix{&0 &1 & 2 & 3 \cr & \oplus & \ominus \cr & & \ominus \cr & & \ominus & \oplus & \ominus}$$
    \caption{The $\mathcal{S}$ data for the above multi-segment is $(\{0, 1\}, \{1, 2, 3\}).$ Since $\mathcal{S}_1 \cap \mathcal{S}_2 = \{1\},$ the chain with support $[3, 1]$ is a $z$-chain.}
    \label{fig: z-chain}
\end{figure}

\begin{defn}
\label{def consecutive chains}
    We say that two chains $r_1$ and $r_2$ are \emph{consecutive chains} if they are associated to two consecutive sets $\mathcal{S}_i$ and $\mathcal{S}_{i + 1}$.
\end{defn}
    
\begin{defn}
\label{def multiples belonging to a chain}
    Let $r_1$ and $r_2$ be consecutive chains. If $r_3$ is a multiple such that $r_1 < r_3 < r_2,$ then we say that $r_3$ \emph{belongs} to $r_1.$
\end{defn}

\begin{rmk}
    It follows from the sign condition in Definition \ref{defn E(M,S)} that all multiples belonging to some chain $r$ have the same sign $\eta(r) \cdot (-1)^{C(r) - 1}.$
\end{rmk}

\begin{rmk}\label{rmk odd signs}
    Suppose $\mathcal{E} = \mathcal{E}(\mathcal{M}, \mathcal{S}, \eta)$ as in Definition \ref{defn E(M,S)}. Suppose the circles of $\mathcal{E}$ are read out in order, row by row, then the symbols $\oplus$ or $\ominus$ always appear an odd number of times consecutively. For example, in Figure \ref{fig: z-chain}, there is one $\oplus,$ then three $\ominus$s, then one $\oplus,$ and then one $\ominus.$
\end{rmk}

It turns out that for blocks $\mathcal{B}$ starting at zero (i.e., $c_{\min} = 0$), we will need to specify more data to describe all the virtual extended multi-segments equivalent to $\mathcal{B}.$ In particular, some virtual extended multi-segments equivalent to $\BB$ have rows that are hats, but the construction in Definition \ref{defn E(M,S)} only produces virtual extended multi-segments with rows having $l=0$. We therefore make the following modification to the construction $\mathcal{E}(\mathcal{M}, \mathcal{S}, \eta)$ by defining a new parameter $\mathcal{T}.$

\begin{defn}
\label{valid T}
    Suppose $\BB$ starts at zero and let $\mathcal{M} := \mathcal{M}_\mathcal{B}$ be the corresponding block-tuple. Let $\mathcal{S} = (\mathcal{S}_1, \dots, \mathcal{S}_k)$ be a valid tuple for $\mathcal{M}$. Then we say a collection of partitions of each $\mathcal{S}_i$ of the form $(\mathcal{T}_i^0, \mathcal{T}_i^1, \dots,  \mathcal{T}_i^{\ell_i})$ is \emph{valid} if
    \begin{enumerate}
        \item each $\mathcal{T}_i^j$ is a nonempty set of consecutive integers;
        \item if $j < j'$ and $t_j \in \mathcal{T}_i^j$, $t_{j'} \in \mathcal{T}_i^{j'}$, then $t_j \leq t_j'$;
        \item if $|\mathcal{S}_i \cap \mathcal{S}_{i + 1}| \geq 1,$ then $|\mathcal{T}_{i + 1}^0| \geq 2$.
    \end{enumerate}
A \emph{valid} tuple $(\mathcal{M}, \mathcal{S}, \mathcal{T})$ is one such that $\mathcal{S}$ and $\mathcal{T}$ are both valid.
\end{defn}

\begin{rmk}
    As a means of shorthand, we will represent the subsets $\mathcal{T}_i^j \subset \mathcal{S}_i$ by overlining each $\mathcal{T}_i^j$ for $j > 0$. As an example, suppose $\mathcal{S}_i = \{2, 3, 4, 5, 6, 7, 8, 9\},$ and $\mathcal{T}_i^0 = \{2, 3, 4\}, \mathcal{T}_i^1 = \{5, 6\}, \mathcal{T}_i^2 = \{7\}, \mathcal{T}_i^3 = \{8, 9\}.$ As shorthand, we write
    $$\mathcal{S}_i = \{2, 3, 4, \overline{5, 6}, \overline{7}, \overline{8, 9}\}.$$
\end{rmk}

\begin{defn}
\label{defn E(M,S, T)}
    Let $\mathcal{M} = (m_{c_{\min}}, \dots, m_{c_{\max}})$ be a block-tuple with $c_{\min} = 0$. Given valid $(\mathcal{M}, \mathcal{S}, \mathcal{T})$ along with a sign $\eta \in \{\pm 1\}$, we associate a virtual extended multi-segment $\mathcal{E}(\mathcal{M}, \mathcal{S}, \mathcal{T}, \eta)$ as follows.
\begin{enumerate}
    \item For each $\mathcal{T}_i^j$ with $j \geq 1,$ we include a hat with support
    $[\max \mathcal{T}_i^j, -\min \mathcal{T}_i^j].$
    \item For each $\mathcal{T}_i^0 \subset \mathcal{S}_i,$ we include a row of circles with support given by
    $[\max \mathcal{T}_i^0, \min \mathcal{T}_i^0].$
    \item For each $c \in \{c_{\min}, \dots, c_{\max}\},$ we add $m_c - |\{i \mid c \in \mathcal{S}_i\}|$ copies of a row of circles with support
    $[c, c]$
    for some $\eta_i \in\{\pm 1\}.$
\end{enumerate}
The order of the rows and the signs of each row follow exactly the same rules as in Definition \ref{defn E(M,S)}.
\end{defn}

As before, we call rows falling into case (2) \emph{chains} and rows falling into case (3) \emph{multiples}. We also have exactly the same notions of $z$-chains (Definition \ref{def z-chain}) and a multiple belonging to a chain (Definition \ref{def multiples belonging to a chain}) as before. Note that with these definitions, the construction of a virtual extended multi-segment $\EE(\mathcal{M}, \mathcal{S}, \eta)$ is a special case of the construction of $\EE(\mathcal{M}, \mathcal{S}, \mathcal{T}, \eta)$ where each $\mathcal{T}_i^0 = \mathcal{S}_i$ and $\ell_i = 0$.

\begin{defn}
    We say that $\EE$ is of \emph{type $Y_\mathcal{M}$} if it is of the form $\EE(\mathcal{M}, \mathcal{S}, \mathcal{T}, \eta)$ for valid $(\mathcal{M}, \mathcal{S}, \mathcal{T})$ or $\EE(\mathcal{M}, \mathcal{S}, \eta)$ for valid $(\mathcal{M}, \mathcal{S})$.
\end{defn}

\begin{defn}
    If $\EE$ is of type $Y_\mathcal{M},$ then we refer to the associated $\mathcal{S}$ (or $\mathcal{S}$ and $\mathcal{T}$, if applicable) as the \emph{$\mathcal{S}$-data} of $\EE.$
\end{defn}

Since the construction above is quite elaborate, we provide an example of a type $Y_\mathcal{M}$ virtual extended multi-segment.

\begin{exmp}
    Let $n = 9,$ $\mathcal{M} = (1, 1, 3, 1, 1, 3, 1, 3, 1)$, 
    \begin{align*}
        \mathcal{S}_1 = \{0, \overline{1, 2}\},
        \mathcal{S}_2 = \{2, 3, \overline{4}\},
        \mathcal{S}_3 = \{5\},
        \mathcal{S}_4 = \{5, 6, 7\},
        \mathcal{S}_5 = \{8, 9\}.
    \end{align*}
    and $\mathcal{S} = \{\mathcal{S}_1, \mathcal{S}_2, \mathcal{S}_3, \mathcal{S}_4, \mathcal{S}_5\}.$
    Then the associated multi-segment $\mathcal{E}(\mathcal{M}, \mathcal{S}, \mathcal{T})$ with $\eta(\mathcal{E}) = 1$ is:

    \begin{center}
        \bordermatrix{& -4 & -3 & -2 & -1 & 0 & 1 & 2 & 3 & 4 & 5 & 6 & 7 & 8 & 9 \cr
    & \lhd & \lhd & \lhd & \lhd & \oplus & \rhd & \rhd & \rhd & \rhd \cr
    & & & & \lhd & \ominus & \oplus & \rhd \cr
    & & & & & \ominus \cr
    & & & & & & & \ominus \cr
    & & & & & & & \ominus & \oplus \cr
    & & & & & & & & & & \ominus \cr
    & & & & & & & & & & \ominus \cr
    & & & & & & & & & & \ominus & \oplus & \ominus \cr
    & & & & & & & & & & & & & \oplus & \ominus \cr
    & & & & & & & & & & & & & \ominus \cr
    & & & & & & & & & & & & & \ominus \cr}
    \end{center}

    Here, the row with support $[2, 2]$ is a multiple belonging to the chain with support $[0, 0]$ and the multiples with support $[8, 8]$ belong to the chain with support $[8, 9].$ Meanwhile, the first row with support $[5, 5]$ is a chain, and the second such row is a multiple belonging to it. 
\end{exmp}

\begin{exmp}
    Suppose $\mathcal{B}$ is a block. Then $\mathcal{B}$ is of type $Y_{\mathcal{M}_\mathcal{B}},$ where $\mathcal{S}_i = \{i\}$ (and $\mathcal{T}_i^0 = \{i\}$, if $\mathcal{B}$ starts at zero).
\end{exmp}

One important property of virtual extended multi-segments of type $Y_\mathcal{M}$ is the following lemma.

\begin{lemma}
\label{odd alternating dual}
    If $\EE$ is of type $Y_\mathcal{M}$, then $\eta(\EE) = \eta(dual(\EE)).$
\end{lemma}
\begin{proof}
    Let $r$ be the last row of $\mathcal{E}.$ Then Lemma \ref{Dual sign} implies that \[\eta(dual(\EE)) = \eta(\widehat{r}) = (-1)^{C(\EE) - C(r)} \eta(r),\]
    where $\widehat{r}$ is the image of $r$ under $dual.$ Note that Remark \ref{rmk odd signs} holds for all multi-segments of type $Y_\mathcal{M},$ even those where $\mathcal{M}$ starts after zero. Therefore, listing the symbols $\oplus$ and $\ominus$ in order produces streaks of odd multiplicity. Therefore, an even number of circles can be removed from $\mathcal{E}$ so that each of these streaks has one circle; i.e., so that $\mathcal{E}$ is alternating. Doing this does not change the sign $(-1)^{C(\EE) - C(r)} \eta(r).$ Therefore, the result follows from Lemma \ref{Alternating dual}.
\end{proof}

This definition of type $Y_\mathcal{M}$ has been set up with the intention of proving the following theorem.

\begin{thm}
\label{block classification}
    Let $\mathcal{B} \in\VRep_\rho^\mathbb{Z}(G_n)$ be a block with multiplicities $\mathcal{M} := \mathcal{M}_\BB.$ Then, the set $\Psi(\pi(\BB))$ is precisely the set of $\psi_\EE$ where $\EE\in\VRep_\rho^\mathbb{Z}(G_n)$ is of type $Y_\mathcal{M}$ with $\eta(\EE)=\eta(\BB)$.
\end{thm}

We will prove this theorem in Section \ref{sec exhaustion of Y_M}.

\subsection{\texorpdfstring{Type $X_k$}{}}

In this section, we define a special case of virtual extended multi-segments of type $Y_\mathcal{M},$ which we call type $X_k.$ We will use this simple case to prove the existence of four basic combinations of operators on virtual extended multi-segments $\mathcal{E}$ of type $X_k,$ all of which send such multi-segments to other multi-segments of type $X_k$.

\begin{defn}
    We say that a virtual extended multi-segment $\mathcal{E}$ is of \emph{type $X_k$} if $\mathcal{E}$ is of type $Y_\mathcal{M}$ for a block-tuple $\mathcal{M}$ starting at $0$ and with all multiplicities $1$.
\end{defn}

Here, we note a few properties about virtual extended multi-segments of type $X_k.$

\begin{lemma}
\label{Type X Basic Properties}
    Let $\mathcal{E}$ be of type $X_k.$ Then,
    \begin{enumerate}
        \item $\EE$ is alternating,
        \item $C(\EE) = k + 1,$ and
        \item $dual(\EE)$ is also of type $X_k.$
    \end{enumerate}
\end{lemma}

\begin{proof}
    Suppose that $\EE=\EE(\mathcal{M},\mathcal{S},\mathcal{T},\eta).$
    Since all the multiplicities of $\mathcal{E}$ are 1, $\mathcal{E}$ has no multiples. Therefore, Part (1) of the lemma follows from the sign condition of Definition \ref{defn E(M,S)}.
    
    Since all the multiplicities of $\EE$ are 1, by Condition (4) of Definition \ref{valid S}, none of the $\mathcal{S}_i$ overlap. Each $\mathcal{S}_i$ contributes rows with a total of $|\mathcal{S}_i|$ circles, so \[C(\mathcal{E}) = \sum_{i = 1}^\ell |\mathcal{S}_i| = \left| \bigcup_{i = 1}^\ell \mathcal{S}_i  \right| = |\{0, 1, \dots, k\}|,\]
    which proves Part (2) of the lemma.
    
    For Part (3), we first show $dual(\EE)$ has the supports of a virtual extended multi-segment of type $X_k$. The dual of a hat with support $[\max \mathcal{T}_i^j, -\min \mathcal{T}_i^j]$ is a row of circles with support $[\max \mathcal{T}_i^j, \min \mathcal{T}_i^j]$. The dual of a row of circles with support $[\max \mathcal{T}_i^0, \min \mathcal{T}_i^0]$ is a hat with support $[\max \mathcal{T}_i^0, -\min \mathcal{T}_i^0]$, except when $\min \mathcal{T}_i^0 = 0$, in which case it is still a row of circles. So the supports of $dual(\EE)$ are the same as the supports of $\EE(\mathcal{M}, \mathcal{S}', \mathcal{T}', \eta)$, where $\mathcal{S}'$ and $\mathcal{T}'$ is created by the following procedure. The elements of $\mathcal{S}'$ are
    \begin{itemize}
        \item for all $i$ and for all $0 < j < \ell_i$, the set $\mathcal{T}_i^j$, with the corresponding partition consisting of just the set $\mathcal{T}_i^j$,
        \item for all $i$ with $\ell_i>0$, the union of $\mathcal{T}_i^{\ell_i}$ and every $\mathcal{T}_j^0$ such that $\ell_k = 0$ for all $i < k \leq j$, with the corresponding partition consisting of the sets $\mathcal{T}_i^{\ell_i}, \mathcal{T}_{i+1}^0, \dots$, in that order, or
        \item $\mathcal{T}_1^0$, with the corresponding partition consisting of just the set $\mathcal{T}_1^0$.
    \end{itemize}
    It is clear that the union of all of these sets is the same as the union of the $\mathcal{S}_i$, and that they do not overlap. We simply order them such that condition (3) of Definition \ref{valid S} is satisfied. Conditions (4) and (5) are not applicable. We simply order the partitions such that condition (2) of Definition \ref{valid T} is satisfied, and condition (3) is not applicable. So this is indeed valid. Moreover, we see that, except for $\mathcal{T}_1^0$, every $\mathcal{T}_i^j$ with $j>0$, corresponding to a hat, is some ${\mathcal{T}'}_{i'}^{j'}$ for $j' = 0$, corresponding to a row of circles, and similarly every $\mathcal{T}_i^j$ with $j=0$, corresponding to a row of circles, is some ${\mathcal{T}'}_{i'}^{j'}$ for $j' > 0$. So $\EE(\mathcal{M}, \mathcal{S}', \mathcal{T}', \eta)$ does indeed have the correct supports.
    
    Next we check that $\EE(\mathcal{M}, \mathcal{S}', \mathcal{T}', \eta)$ has the same order as $dual(\EE)$. Note that applying $dual$ to $\mathcal{E}$ replaces each row with its dual and reverses their order. So if $\EE$ is in $(P')$ order then so is $dual(\EE)$. Moreover, since the $\mathcal{S}_i$ do not overlap, each of the rows of $\EE$ starts in a distinct column, and the same is true for $dual(\EE)$. So the $(P')$ condition is enough to specify the order exactly. But $\EE(\mathcal{M}, \mathcal{S}', \mathcal{T}', \eta)$ also satisfies $(P')$, so they have the same order.

    Finally, we check the signs. By Lemma \ref{Alternating dual}, $dual(\EE)$ is alternating, and $\eta(dual(\EE)) = \eta$. Since $\EE(\mathcal{M}, \mathcal{S}', \mathcal{T}', \eta)$ is also alternating and has sign $\eta$, the signs match up, so we are done.
\end{proof}

We aim to prove the following result as a precursor to the more general Theorem \ref{block classification}.

\begin{thm}
\label{supercuspidal}
    Suppose that $\EE\in\VRep_\rho^\mathbb{Z}(G_n)(G)$ is tempered of type $X_k$ with sign $\eta(\EE)$. Then, the set $\Psi(\pi(\EE))$ is  the set of $\psi_\mathcal{F}$ where $\mathcal{F}\in\VRep_\rho^\mathbb{Z}(G_n)(G)$ is of type $X_k$ with $\eta(\mathcal{F})=\eta(\EE)$.
\end{thm}

We will prove this theorem at the end of this subsection.
In order to prove this theorem, we will completely classify all raising operators that are possible on virtual extended multi-segments of type $X_k.$ First, we prove the existence of two very important operations for virtual extended multi-segments of type $X_k$.

\begin{lemma}
\label{U existence}
    Suppose that $\EE$ is of type $X_k$ and that $h = ([A, -B], B, \eta)$ is a hat in $\EE$. Let $c$ be a positive integer and suppose $C(h) > c$. Then there exists a series of operations on $\EE$ resulting in a new virtual extended multi-segment $\EE'$ such that:
    \begin{itemize}
        \item $\EE'$ has a hat of the form $h' = ([A - c, -B], B, \eta)$,
        \item $\EE'$ has a row of $c$ circles with support $[A, A - c + 1]$,
        \item the other segments of $\EE'$ are precisely the same as the segments of $\EE - \{h\},$ except with possibly different signs,
        \item $\eta(\EE') = \eta(\EE),$ and
        \item $\EE'$ is of type $X_k.$
    \end{itemize}
\end{lemma}

\begin{proof}
    We first prove the statement in the case where $A = k.$ The operation whose existence we shall establish consists of the following combination:

    \begin{enumerate}
        \item Row exchange $h$ until it is the last row $\EE.$
        \item Perform a $ui^{-1}$ of type 3' to separate $c$ circles from the bottom row.
        \item Row exchange the second to last row all the way to the top.
    \end{enumerate}

    In light of Corollary \ref{swap commutes with M} and  Lemma \ref{lemma Alex}(3), we know that these row exchanges commute with merging hats and $ui$'s of type 3'. Therefore, to determine the image of $h$ under the row exchanges, we may assume that $\EE$ is completely unmerged: i.e., each row $r \in \EE$ has $C(r) = 1$.

    Thus, the row $h$ is followed by $k + 1 - C(h)$ rows $r_1, \dots, r_{k + 1 - C(h)}$, all satisfying the alternating sign condition. It is clear by analyzing $\mathcal{S}$-data that $\EE - \{h\}$ is of type $X_{k - C(h)}$; therefore, $\supp(h)$ contains the supports of all other rows. Therefore, by Lemma \ref{big swap down}, exchanging $h$ down to the bottom gives us a row $h'$ such that
    $$l(h') = l(h) - (k + 1- C(h)) = B - k - 1 + (k - B + 1) = 0,$$
    $$\eta(h') = (-1)^{k + 1 - C(h)} \eta(h).$$

    Therefore, $h'$ is a row of circles. Because $c < C(h) = k - B + 1,$ we have that $B \leq k - c,$ which means that it is possible to use a $ui^{-1}$ to separate off the last $c$ circles into a new row. We denote the two resulting rows by $h_1, h_2.$

    Lemma \ref{big swap down} says that the images $r_1' < \cdots < r_{k + 1 - C(h)}'$ are alternating but $(r_{k + 1 - C(h)}', h')$ fails the alternating sign condition, This means that $(r_{k + 1 - C(h)}', h_1)$ also fails the alternating sign condition. Again, $supp(h_1) = [k - c, B]$ still contains the supports of all the rows $r_i'.$ Therefore, Lemma \ref{big swap up} tells us that exchanging $h_1$ to the top of the multi-segment gives us new rows $h_1' < r_1'' < \cdots < r_{k + 1 - C(h)}''$ satisfying the alternating sign condition. In particular, we have
    $$l(h_1') = \eta(h_1) + k + 1- C(h) = 0 + l(h) = B$$
    and
    \begin{align*}
        \eta(h_1') = (-1)^{k + 1 - C(h)}  \eta(h_1)
        = (-1)^{k + 1 - C(h)}  \eta(h') 
        &= (-1)^{k + 1 - C(h)}  (-1)^{k + 1 - C(h)}  \eta(h)\\ 
        &= \eta(h).
    \end{align*}

    These equations tell us that the resulting row $h_1'$ is a hat of the form $([A - c, -B], B, \eta).$ In particular, if we call the resulting virtual extended multi-segment $\EE'$, then $\eta(\EE) = \eta(\EE').$ Finally, we need to check that the pair $(r_{k + 1 - C(h)}'', h_2)$ satisfies the alternating sign condition:
    \begin{align*}
        \eta(h_2) = (-1)^{C(h_1) + c}  \eta(h_1) 
        &= (-1)^{C(h) + c}  (-1)^{k + 1 - C(h)}  \eta(h') 
        = (-1)^{k + 1 + c}  \eta(h).
    \end{align*}

    Meanwhile, we have
    \begin{align*}
        \eta(r_{k + 1 - C(h)}'') &= (-1)^{C(h) - c}  (-1)^{C(h)}  \eta(r_{k + 1 - C(h)}) = (-1)^c \eta(r_{k + 1 - C(h)}).
    \end{align*}

    But since the original rows $h < r_1 < \cdots < r_{k + 1 - C(h)}$ were alternating, we obtain $$(-1)^c  \eta(r_{k + 1 - C(h)}) = (-1)^c  (-1)^{k - C(h)}  (-1)^{C(h)}  \eta(h) = -\eta(h_2),$$
    and since $C(r_{k + 1 - C(h)}) = 1,$ this suffices to show that the pair is alternating. Now, it is clear that $\EE' - \{h_1', h_2\}$ has the same segments as $\EE - \{h\}$ except, possibly, for the signs. Since it is alternating, we see that $\EE' - \{h_1', h_2\}$ is of type $X_{k - C(h)}.$ We conclude that $\EE'$ is itself $X_k.$

    Now, we progress to the more general case where $k \geq A.$ Since $\EE$ is of type $X_k,$ the sub-virtual extended multi-segment $\EE_A$ of which $h$ is the first row is  of type $X_A$. As we have seen, performing the aforementioned operations on $\EE_A$ gives a segment $\EE_A'$ that has the same sign and is of type $X_A$. Just as $\EE$ is built from $\EE_A$ by including alternating merged hats and rows of circles, $\EE'$ is likewise built the same way from $\EE_A'.$ Thus, to show that $\EE'$ is of type $X_k$, it suffices to show that it is alternating.
    
    In particular, it suffices to show that $h_1'$ and the row that comes before it satisfy the alternating sign condition, and that $h_2$ and any row following it satisfies the alternating sign condition. The first of these follows from the fact that $h$ and $h_1'$ have the same sign. For the second of these, we note that $C(\EE_A) = C(\EE_A')$. Because $\EE$ satisfies the alternating sign condition, the row following the block $\EE_A'$ should have sign $\eta(\EE_A)(-1)^{C(\EE_A)}.$ This is equal to $ \eta(\EE_A')(-1)^{C(\EE_A')},$ so it has the correct sign in $\EE'$ as well. Thus, $\EE'$ is $X_k.$ That $\eta(\EE') = \eta(\EE)$ is clear, since any rows before $h$ are unaffected.
\end{proof}

Next, we show that it is sometimes possible to perform exactly one $dual \circ ui \circ dual$ on certain virtual extended multi-segments of type $X_k.$

\begin{lemma}
\label{one dual ui dual}
    Suppose $\EE$ is of type $X_k$ and $h = ([k, -B],B, \eta)$ is the first row. Then it is possible to perform exactly one $dual \circ ui \circ dual$ involving the row $h.$ This operation preserves the property of being type $X_k$ and the sign $\eta(\EE).$
\end{lemma}

\begin{proof}
    The image of $h$ in $dual(\EE)$ is a row of circles $\widehat{h} = ([k, B], 0, \eta').$ Since $\EE - \{h\}$ is of type $X_{k - C(h)},$ none of the other rows of $dual(\EE)$ have supports intersecting $[k, B].$ Therefore, the only possible $ui$ in $dual(\EE)$ involving $\widehat{h}$ is one of type 3'. Therefore, a $ui$ must involve the image of the unique row $r\in\EE$ with support ending at $A(r)=B - 1.$ The existence of a row ending at $B - 1$ is guaranteed because $\EE \setminus \{h\}$ is of type $X_{k - C(h)} = X_{B - 1}.$ 
    
    There are two possible cases. The first case is that the row $r \in \EE$ ending at $B - 1$ is the hat directly after $h,$ in which case the unique $dual \circ ui \circ dual$ involving $\widehat{h}$ is the operator merging the hats. The second case is that the row $r \in \EE$ ending at $B - 1$ is a row of circles on the bottom of $\EE.$ In $dual(\EE),$ the image $\widehat{r}$ is the first hat, and its support still ends at $B - 1.$ The assertion that it is possible to perform a $ui$ between $\widehat{r}$ and $\widehat{h}$ is equivalent to the assertion that the $ui^{-1}$ operation in Lemma \ref{U existence} can be undone. The fact that such an operation preserves the sign and the property of being of type $X_k$ follows from Lemmas \ref{Dual sign} and \ref{U existence}.
\end{proof}

In the more general case where $h = ([A, -B], B, \eta)$ is not the first row of $\EE,$ we still have some sub-multi-segment $\EE_A$ of type $X_A$, with $h$ as the first hat. This means that it is always possible to perform a $dual \circ ui \circ dual$ on any hat, though this operation may not be unique. To keep the uniqueness property, we restate the result as such:

\begin{lemma}
\label{D existence}
    Suppose $\EE$ is of type $X_k$ and $h = ([A, B], -B, \eta)\in\EE$ is any hat. Then there exists a unique way to apply $dual \circ ui \circ dual$ to a pair of rows $(h, r)$ where $A(r) < A.$
\end{lemma}

Again, this operation will fall into one of two cases: (1) $r$ is a hat and the operation is merging, and (2) $r$ is a row of circles. For the sake of brevity, we will refer to the latter operation as \emph{dualizing} the hat $h$ to $r$. Note that a hat with support $[A, B]$ can only be dualized to a row of circles $r$ whose support ends at $-B - 1.$

The previous two operations are two of four possible raising operators that can be performed for virtual extended multi-segments $\EE$ of type $X_k$. Here are the four possible raising operators.

\begin{defn}\label{def-smud}\
\begin{itemize}
    \item Separation of a row of circles via $ui^{-1}$ of type 3' which we denote by $S$,
    \item Merging consecutive hats as in Lemma \ref{merge}, which we denote by $M$,
    \item $ui^{-1}$ as in Lemma \ref{U existence} (again of type 3') which we denote by $U$, and
    \item Dualizing a hat $h$ as in Lemma \ref{D existence}, which we denote by  $D$.
\end{itemize}
\end{defn}

Note that operators $S$ and $M$ are in a sense inverse-dual to each other; specifically, $M$ can be obtained by dualizing $\EE$ and performing the inverse of $S$. These operations are different from the other two in that they involve the interaction of consecutive rows. $U$ and $D$ are similarly inverse-duals, and both of them involve non-consecutive interactions: applying $U$ or $D$ to a hat near the top of $\EE$ results in the appearance of circles near the bottom of $\EE.$ These operations are summarized in the following table.
\begin{center}
    \begin{tabular}{|c|c|c|}
    \hline
     & Consecutive & Non-Consecutive \\
     \hline
     $ui^{-1} $& $S$ & $U$\\
     \hline
     $dual \circ ui \circ dual$ & $M$ & $D$\\
     \hline
\end{tabular}
\end{center}

Going forward, we introduce more elaborate notation for the operators above to make the ensuing proofs more precise.
\begin{defn} \
\label{def-operators}
    \begin{itemize}
        \item Let $S_{r, c}$ be the operator separating out $c$ circles from a row $r.$
        \item Let $U_{h, c}$ be the $ui^{-1}$ operator that removes $c$ circles from the hat $h.$
        \item Let $M_{h_1, h_2}$ be the operator merging the hats $h_2, h_2.$
        \item Let $D_{h,r}$ be the operator that dualizes the hat $h$ to the row $r.$
    \end{itemize}
\end{defn}

It turns out that the operators $S, M, U$ and $D$ completely classify equivalence for virtual extended multi-segments of type $X_k.$

\begin{lemma}
\label{Type X exhaustion}
    Let $\EE$ be of type $X_k.$ Then the only raising or lowering operators that can be applied to $\EE$ are those of the form $S, M, U, D$ and their inverses.
\end{lemma}

\begin{proof}
    All raising operators are of the form either $dual \circ ui \circ dual$ or $ui^{-1}.$ We have already seen from Lemma \ref{D existence} that the only possible $dual \circ ui \circ dual$s are $M$ or $D$. Applying $ui^{-1}$ of type 3' to a row of circles is necessarily an $S$ operation. Applying $ui^{-1}$  of type 3' to a hat involves conducting row exchanges until that hat is a row of circles, which is necessarily a $U$ operation (recall that a $ui^{-1}$ not of type 3' must be of the form $dual \circ ui \circ dual$ by Lemma \ref{lemma ui inv = dud}). Meanwhile, if $T$ is a lowering operator on $\EE,$ then $dual \circ T \circ dual$ is a raising operator on $dual(E)$ and must therefore be one of $S, M, U, D$. Since these operations come in inverse dual pairs, $T$ must be one of their inverses.
\end{proof}

\begin{cor}
\label{preservation}
    If $\EE \sim \EE'$ and $\EE$ is of type $X_k$, then $\EE'$ is also of type $X_{k}.$ Furthermore, $\eta(\EE) = \eta(\EE').$
\end{cor}

\begin{proof}
    It suffices to check that the operations $S, M, U, D$, and their inverses all preserve the property of being of type $X_k$ and also preserve the sign. For $S$ and $M$, and their inverses, this fact is obvious. For $U$ and $D,$ it follows from Lemmas \ref{U existence} and \ref{D existence} respectively. For $U^{-1}$ and $D^{-1},$ it follows from the fact that $U$ and $D$ are inverse-duals of each other, combined with Lemmas \ref{odd alternating dual} and \ref{Type X Basic Properties}.
\end{proof}

Now, we are finally ready to prove the classification from Theorem \ref{supercuspidal}.

\begin{proof}[Proof of Theorem \ref{supercuspidal}]
    Suppose that $\EE\in\VRep_\rho^\mathbb{Z}(G_n)$ is tempered of type $X_k$ with sign $\eta(\EE)$. Lemma \ref{preservation} implies that all $\EE'\in\VRep_\rho^\mathbb{Z}(G_n)$ which are equivalent to $\EE$ are also type $X_k$ with the same sign. Conversely, if $\EE'$ is of type $X_k$, applying $D_{h_i}$ for each hat of $\EE'$ gives an equivalent virtual extended multi-segment made up only of circles. Applying $S$ successively leaves a tempered segment with the same multiplicities as $\EE$. Sign preservation guarantees this multi-segment is the same as $\EE.$ 
\end{proof}

\subsection{\texorpdfstring{Operations for type $Y_\mathcal{M}$}{}}

For the duration of this section, let $\mathcal{E}$ be of type $Y_\mathcal{M}$ for some block-tuple $\mathcal{M}$ of odd-integers. The goal of this section is to describe and classify the row operations that can be implemented on $\EE$. These are analogous to the operations $S, M, U,$ and $D$ seen for type $X_k,$ but they are generally more complicated. Furthermore, these operators may be understood nicely through their effect on the $\mathcal{S}$-data of $\EE.$ Note in advance that the discussion of the $S$ operator will apply to all multi-segments of type $Y_\mathcal{M},$ while the other operators only concern virtual extended multi-segments starting at zero.

First, we consider the operation $S.$ For multi-segments of type $X_k,$ applying operations $ui^{-1}$ to rows of circles is simple: all operations $S_{r, c}$ are possible, and no row exchanges are necessary. However, this is not always the case with general multi-segments of type $Y_\mathcal{M}.$ The following lemma describes a situation where we can apply $S$.

\begin{lemma}[Existence of S]
\label{S existence general}
    Suppose $\EE$ is of type $Y_\mathcal{M}$, $r = ([A, B], 0, \eta)$ is a chain in $\EE,$ and $k < C(r)$ is an integer. Then there exists a series of operations preserving the type $Y_\mathcal{M}$ that splits $\mathcal{S}_i = \{B, \dots, A, \dots\}$ into two sets \[\mathcal{S}_i^1 = \{B, \dots, A - c\} \text{ and } \mathcal{S}_i^2 = \{A - c + 1, \dots, A, \dots\}.\]
\end{lemma}

\begin{proof}
    We will conduct the following series of operations.
    \begin{itemize}
        \item Exchange $r$ downwards past all rows whose supports begin at $A - c$ or earlier. The resulting row $r'$ will be equal to $r.$
        \item Apply $S_{r, c}.$
        \item Exchange the first row of the result back up to the original position of $r.$
    \end{itemize}

    We assert that the only rows whose supports begin at $A - c$ or earlier are multiples. This is equivalent to the assertion that there are no chains beginning at $B - c$ or earlier, which must be true because two chains can only overlap for at most one circle. Now, all of these multiplies must belong to the chain $r,$ and since they do not overlap with any other chains, there must be evenly many of each of them. Therefore, by Lemma \ref{lem-multiplicity-cancel}, the row $r$ stays the same after exchanging with these rows. Meanwhile, each multiple's sign is multiplied by $(-1)^{C(r) - 1}.$

    Now, the rows after $r$ are precisely the rows whose supports begin at least at $A - c + 1,$ so we can separate $r$ into two rows: $$r_1 = ([A - c], 0, \eta), ~~~~~ r_2 = ([A, A - c + 1], 0, \eta^{C(r_1)}).$$
    Then exchanging $r_1$ up with all the multiplicities once again preserves $r_1$ by Lemma \ref{lem-multiplicity-cancel}, while all the multiplicities have their signs multiplied by $(-1)^{C(r_1)} = (-1)^{C(r) - 1 - c}.$ It is clear that the resulting multi-segment has the same $\mathcal{S}$-data as $\EE,$ except that $\mathcal{S}_i = \{B, \dots, A, \dots\}$ has been split into $\{B, \dots, A - c\}$ and $\{B - c + 1, \dots, A, \dots\}.$ It remains to be checked that the multi-segment is odd-alternating as in Definition \ref{defn E(M,S)}.

    In $\EE,$ all the multiplicities belonged to $r$ and therefore had sign $\eta  (-1)^{C(r) - 1}.$ These row exchanges have changed their sign by $(-1)^{k},$ so the new sign is $\eta  (-1)^{C(r) - c - 1} = \eta  (-1)^{C(r_1) - 1}.$ This obeys the odd-alternating condition. Before $r_1$ was exchanged up, the pair $(r_1, r_2)$ obeyed the alternating sign condition, so $r_2$ must alternate with the multiplicity right before it after the row exchange: since there are an even number of multiplicities, this is alternating. Finally, the multiplicities belonging to $r_2$ are precisely those that belonged to $r$ but do not now belong to $r_1.$ These still have sign $$\eta  (-1)^{C(r) - 1} = \eta  (-1)^{C(r_1)} (-1)^{C(r_2) - 1} = \eta(r_2)  (-1)^{C(r_2) - 1},$$
    so they obey the alternating sign condition.
\end{proof}

We continue to refer to the operation described in Lemma \ref{S existence general} as $S_{r, c}.$

\begin{rmk}
\label{S exception}
    There is a small subtlety that can occur with the $S$ operation. Suppose we apply $S_{r, k}$ to a chain $r = ([A, B], 0, \eta)$ such that $A - B = k$ and $r$ is a $z$-chain. Then separating $k$ circles leaves behind a one circle row $([B, B], 0, \eta)$ associated to $\mathcal{S}_i^1.$ However, the conditions for being of type $Y_\mathcal{M}$ require that $|\mathcal{S}_i^1| \geq 2$. That is, we would like to count it as a multiple rather than a chain. To fix this notational problem, instead of saying that we have separated $\mathcal{S}_i$ into two sets $\mathcal{S}_i^1 = \{B\}$ and $\mathcal{S}_i^2 = \{B + 1, \dots, A\}$, we instead throw out $\mathcal{S}_i^1$ and replace $\mathcal{S}_i$ with $\mathcal{S}_i^2$ solely. If $A \geq B +2$, the chain corresponding to $\{B + 1, \dots, A\}$ must be row exchanged below the (evenly many) multiples with support $[B + 1, B + 1].$
\end{rmk}

To summarize, we can describe the impact of $S$ on the $\mathcal{S}$-data of $\EE$ as follows:

\begin{rmk}
\label{S S-data}
    Let $r = ([A, B], 0, \eta).$ We have two cases as follows.

\begin{itemize}
    \item If $c = A - B$ and $r$ is a $z$-chain, then applying $S_{r, c}$ replaces:
    $$\{B, \dots, A, \dots\} \longrightarrow \{B + 1, \dots, A, \dots\}.$$ 
    
    \item Otherwise, applying $S_{r, c}$ replaces:
    $$\{B, \dots, A, \dots\} \longrightarrow \{B, \dots, A - c\}, \{A - c + 1, \dots, A, \dots\}.$$
\end{itemize}
\end{rmk}

\begin{defn}
    We shall refer to the case where $k = A - B$ and $r$ is a $z$-chain as $S^2,$ and the other case as $S^1.$
\end{defn}

We note that this $S$ operation still commutes with row exchanges. More precisely, the following corollary follows immediately from Parts (1) and (3) of Lemma \ref{lemma Alex}.

\begin{cor}
\label{swap commutes with S}
    Let $R$ denote the operation of exchanging some row $r$ with a series of rows $r_i < r_{i + 1} < \cdots < r_j.$ Let $S$ be some $ui^{-1}$ of type $S$ involving the rows $r_i < r_{i + 1} < \cdots < r_j.$ Then $R \circ S = S \circ R$.
\end{cor}

For the remaining operations $M, U,$ and $D$ discussed in this section, we assume that $\mathcal{E} = \mathcal{E}(\mathcal{M}, \mathcal{S}, \mathcal{T},\eta)$ begins at zero. The most simple operation that can be implemented on $\EE$ is $M,$ the merging of hats.

\begin{lemma}
\label{M existence general}
    Suppose $\mathcal{E} = \mathcal{E}(\mathcal{M}, \mathcal{S}, \mathcal{T},\eta)$ begins at zero and has hats $h_1 < h_2$ obtained from $\mathcal{T}_i^j$ and $\mathcal{T}_{i'}^{j'}.$ Then $h_1$ and $h_2$ can be merged if and only if $i = i'$ and $j = j' + 1.$
\end{lemma}

\begin{proof}
    Suppose that $h_1 = ([A_1, -B_1], B_1, \eta_1)$ and $h_2 = ([A_2, -B_2], -B_2, \eta_2).$ Then we must have $\mathcal{T}_i^j = \{B_1, \dots, A_1 \}$ and $\mathcal{T}_{i'}^{j'} = \{B_2, \dots, A_2\},$ and these sets must be disjoint for $(\mathcal{S}, \mathcal{T})$ to be valid. Since $h_1 < h_2,$ we must have $B_1 > A_2.$ Because $\EE$ is odd-alternating and in $(P')$ order, $h_1$ and $h_2$ can be merged if and only if $B_1 = A_2 + 1$ by Lemma \ref{merge}.

    We prove that if $B_1 = A_2 + 1,$ then $i = i'$ and $j = j' + 1,$ as the converse is clear. Suppose for the sake of contradiction that $i \neq i',$ in which case $i > i'$ since $B_1 > A_2.$ But because $A_2 = B_1 - 1 \in \mathcal{S}_{i'}$ and $|\mathcal{S}_i \cap \mathcal{S}_{i'}| \leq 1,$ we must have that $\min \mathcal{S}_i = B_1$ or $A_2.$ But since $j \geq 1,$ this would imply that $\mathcal{T}_i^0 \subset \{A_1\},$ contradicting the condition that $|\mathcal{T}_i^0| \geq 2.$ Therefore, we conclude that $i = i'.$ The fact that $j = j' + 1$ follows from the fact that $B_1 = A_2 + 1.$
\end{proof}

\begin{rmk}
\label{M S-data}
    With this lemma in mind, we observe that the operation $M$ has a nice description in terms of the $\mathcal{S}$ data of $\EE.$ If $\mathcal{T}_i^j$ and $\mathcal{T}_i^j$ are the sets associated with $h_1$ and $h_2,$ then merging $h_1 * h_2$ corresponds with replacing these sets with $\mathcal{T}_i^j \cup \mathcal{T}_i^{j + 1}.$

    In other words, applying $M_{h_1, h_2}$ to $h_1 = ([A, -B], B, \eta_1)$ and $h_2 = ([-B + 1, C], -C, \eta_2)$ replaces $$\{\dots \overline{C, \dots, B - 1}, \overline{B, \dots, A}, \dots \} \longrightarrow \{\dots \overline{C, \dots, B - 1, B, \dots, A}, \dots \}.$$
\end{rmk}

Next, we show that it is possible to dualize hats.

\begin{lemma}[Existence of $D$]
\label{D existence general}
    Suppose $\EE = \mathcal{E}(\mathcal{M}, \mathcal{S}, \mathcal{T},\eta)$ with $\mathcal{M}$ beginning at zero and that $h = ([A, -B], B, \eta)$ is a hat in $\EE$. Suppose that $r_1, r_2, \dots, r_k$ are the rows of circles whose support contains $B - 1.$ The following hold.
    \begin{enumerate}
        \item The row $r_k$ ends at $B - 1,$ and  it is possible to perform a $dual \circ ui \circ dual$ between $h$ and $r_k,$ which we notate $D_{h, r_k}.$ We have that
        $D_{h, r_{k}}(\EE)$ is also of type $Y_\mathcal{M}$ and has the same sign. 
        \item Suppose $r = r_i$ is a row of circles and $k - i \in 2\mathbb{Z}$. Then it is possible to perform a $dual \circ ui \circ dual$ between $h$ and $r,$ which we notate $D_{h, r}.$
        \item If $\EE'$ is the image of $\EE$ under such a $dual \circ ui \circ dual$ in either Part (1) or (3), then $\EE$ and $\EE'$ are of type $Y_\mathcal{M}$ for the same $\mathcal{M}.$ Moreover, $\eta(\EE) = \eta(\EE').$ 
        \item $\EE$ is equivalent to a tempered virtual extended multi-segment of type $Y_\mathcal{M}.$
    \end{enumerate}
\end{lemma}

Before proving this lemma, we note that the operation $D_{h, r}$ as it is described above commutes with row exchanges (if it exists).

\begin{lemma}
\label{lem-swap commutes with D}
    Suppose $\mathcal{E} = \mathcal{E}(\mathcal{M}, \mathcal{S}, \mathcal{T},\eta)$ begins at zero. Let $r \in \mathcal{E}$ be a row, and let $r_1$ and $r_2$ be rows such that the operation $D_{r_1, r_2}$ is valid. Suppose that $supp(r)$ contains the supports of the rows between $r_1$ and $r_2$ inclusively. Then exchanging $r$ down with the rows $r_1 < \cdots < r_2$ results in the same $r'$ as exchanging $r$ with the images of these rows under $D_{r_1, r_2}.$
\end{lemma}

\begin{proof}
    We prove the result by inducting on the number of hats below $r_1$. As part of the inductive assumption, we can assume that all hats after $r_1$ can be dualized without affecting the image $r'$ under row exchanges. Therefore, the rows between $r_1$ and $r_2$ (excluding $r_1$ but including $r_2$) are all rows of circles. In light of Lemma \ref{swap commutes with S}, we can presume that all chains between $r_1$ and $r_2$ are completely unmerged so that the rows from $r_1$ to $r_2$ (not inclusive) form a tempered multi-segment. Then, by Lemmas \ref{lem-multiplicity-cancel} and \ref{lem-multiplicity-cancel-up}, we can presume that there are no multiples, so that the rows $r_1 < \cdots < r_2$ consists of a multi-segment of type $X_k$ for some $k.$ Then, these rows must alternate, so Lemma \ref{big swap down} indicates that $r'$ depends only on the total number of circles in these rows, which is preserved under $D_{r_1, r_2}.$
\end{proof}

With the above lemma in hand, we prove Lemma \ref{D existence general}.

\begin{proof}[Proof of Lemma \ref{D existence general}]
    We first verify that $r_k$ ends at $B - 1$. Assume otherwise for the sake of contradiction. Then $r_k$ must have support containing $B$, so $B \in \mathcal{T}_i^0$ for some $i.$ This is impossible because $B$ is in the hat $h$, so $B \in \mathcal{T}_{i'}^{j'}$ for some $j \geq 1.$ Then we must have that $|\mathcal{S}_i \cap \mathcal{S}_{i'}| \geq 2,$ which is a contradiction.
    
    We prove Parts (1) and (4) together within the same inductive argument. In Part (1), we are asserting that it is possible to conduct the following series of operations:
    \begin{itemize}
        \item Dualize $\EE,$ so that the image $\widehat{r_k}$ is a hat and $\widehat{h}$ a row of $C(h)$ circles in $dual(\mathcal{E})$.
        \item Exchange down $\widehat{r_k}$ in $dual(\EE)$ until its image $\widehat{r_k}'$ is the row right before $\widehat{h}.$
        \item Perform a $ui$ of type 3' between $\widehat{r_k}'$ and $\widehat{h}.$
        \item Row exchange the resulting row until it is in the original position of $\widehat{r_k}$.
    \end{itemize}
    
    Firstly, it is clear that none of the rows except for those between $h$ and $r_k$ are relevant to the existence of these operations (because they can only change all the signs of the rows between $h$ and $r_k$ by the same factor), so it suffices to prove the lemma in the case where $h$ is the first row of $\EE$ and $r_k$ is the last. We inductively assume that all hats after $h$ have been dualized; we are able to dualize these hats independently of the proposed operation on $h$ by Lemma \ref{lem-swap commutes with D}. By Corollary \ref{swap commutes with S}, we can also reduce to the case where all chains that come before $r_k$ have exactly one circle. Thus, we have supposed inductively that $\EE \sm \{h, r_k\}$ is a tempered virtual extended multi-segment. Now, we separately consider each of the two following cases.

    The first case is that $m_{B - 1} \geq 3,$ in which case $r_k$ is a multiple of the form $([B - 1, B - 1], 0, \eta').$ Then $\mathcal{E} \sm \{h\}$ is tempered as well since all its rows have exactly one circle. So we have
    $$\EE \setminus \{h\} = \bigcup_{i = 1}^{B - 1} \bigcup_{j = 1}^{m_i} ([i, i], 0, (-1)^i  \eta(\EE)).$$
    This means that  
    $$dual(\EE)\setminus \{\widehat{h}\} = \bigcup_{i = 1}^{B - 1} \bigcup_{j = 1}^{m_i} ([i, -i], i, \eta_i').$$
    However, in light of Lemma \ref{lem-multiplicity-cancel}, proving this result for $\EE$ is equivalent to proving it for a virtual extended multi-segment where all multiples have been removed. We then presume $$\EE \sm \{h\} = \bigcup_{i = 1}^{B - 1} ([i, i], 0, (-1)^i  \eta(\EE)),$$
    which is precisely the case where $\EE$ is of type $X_k,$ so Part (1) follows from Lemma \ref{D existence}. If $\EE - \{h\}$ is tempered, then employing this $dual \circ ui \circ dual$ to combine $h$ and $r_k$ into a single row $r'$ and then using successive operations $S_{r_k', 1}$ shows that $\EE$ itself is equivalent to a tempered multi-segment which proves Part (4) in this case.

    The second case is that $m_{B - 1} = 1,$ in which case $r_k = ([C, B], 0, \eta')$ is a chain. Depending on whether $r_k$ overlaps with another chain at $C,$ we have that
    $$\EE \sm \{h, r_k\} = \bigcup_{i = 1}^{C} \bigcup_{j = 1}^{m_i'} ([i, i], 0, (-1)^i  \eta(\EE)),$$
    where 
    $$m_i' := \begin{cases}
        m_i - 2 & \mathrm{if} \ i = C ~\text{and}~ \text{$r_k$ overlaps with another chain}, \\
        m_i - 1 & \mathrm{if} \ i = C ~\text{and}~ r_k ~ \text{does not overlap with another chain},\\
        m_i & \text{otherwise.}
    \end{cases}$$
    Again, when row exchanging $\widehat{r_k}$ with the other rows of $dual(\EE),$ the multiplicities will not impact the image $\widehat{r_k}'$ due to Lemma \ref{lem-multiplicity-cancel}. Therefore, we can presume that each $m_i$ is equal to $0$ or $1,$ and we have reduced to the case where $\EE$ is of type $X_k.$ Again, this suffices for Part (1). Dualizing $h$ and applying $S$ suffices to show Part (4).

    We note that Part (2) can be reduced to Part (1) by similarly only considering rows between $h$ and $r_i.$  Since $r_i$ and $r_k$ have the same parity, such a multi-segment is still of type $Y_\mathcal{M}$ and the proof still holds.

    Similarly, it suffices to prove Part (3) for the case where $i = k.$ The fact that the sign $\eta(\EE)$ is preserved under the $dual \circ ui \circ dual$ operation is inherited from the fact that the sign is preserved when $\EE$ is of type $X_k.$ This is because the presence of multiplicities does not cause a cumulative change in sign under row operations. To show that type $Y_\mathcal{M}$ is preserved, we note that none of the multiplicities are affected under the operation; it suffices to show that the odd-alternating condition is preserved.

    We note that exchanging $\widehat{r}$ down until it is the row before $dual(h)$ changes the signs of all the intermediary rows by $(-1)^{C(r) - 1},$ and exchanging the merged row back up changes all the signs by $(-1)^{C(r) + C(h) - 1}$ (Lemmas \ref{big swap down} and \ref{big swap up}). Therefore, all the intermediate rows have a net sign change of $(-1)^{C(h)}.$ Consecutive intermediate rows all satisfy the same sign conditions as before. Meanwhile, it follows from the proof of Lemma \ref{D existence} that the signs of the first and last circles of the block of rows $h < \cdots < r$ stays the same under the $D_{h, r}$ operation. Then $r'$ must satisfy the sign condition with the row after it prescribed by Definition \ref{defn E(M,S)}. Meanwhile, the sign of the row before $r$ is changed by $(-1)^{C(h) - 1},$ as is the last circle of $\widehat{r}',$ so $r'$ must also satisfy the proper sign condition with the row before it.
\end{proof}

The above lemma is enough to deduce the following result.

\begin{cor}
\label{Type Y SMUD equivalence}
    Any two $Y_{\mathcal{M}}$ multi-segments with the same $\eta$ are equivalent up to the operators $S,$ $D,$ and their inverses.
\end{cor}

\begin{proof}
    First, we suppose $\mathcal{M}$ starts at zero. Part (4) of the Lemma \ref{D existence general} implies that all $Y_{\mathcal{M}}$ multi-segments with the same $\mathcal{M}$ and $\eta$ are equivalent to a tempered multi-segment of type $Y_{\mathcal{M}}.$ Part (1) guarantees that this tempered multi-segment always has the same $\mathcal{M}$ and $\eta.$

    If $\mathcal{M}$ starts after zero, then any type $Y_{\mathcal{M}}$ multi-segment has no hats and can be separated via the $S$ operation into a tempered multi-segment. By Lemma \ref{S existence general}, this multi-segment still must have the same $\mathcal{M}$ and $\eta.$
\end{proof}

By reducing to the case where $\EE$ is of type $X_k$, we have shown that the operation $D$ has the same structural properties as it did for type $X_k$. That is to say, if $\EE' = D_{h, r}(\EE)$ and $r'$ is the resulting row, then:

\begin{itemize}
    \item $\EE \sm \{h, r\}$ and $\EE' \sm \{r'\}$ have the same rows up to a change in sign.
    \item If $r$ is a row of circles with support $[A, B],$ then $r'$ is a row of circles with support $[A + C(h), B].$
\end{itemize}

\begin{rmk}
\label{equivalent D}
    Lemma \ref{D existence general} and the above description concern only the case where $r = r_i$, when $i$ and $k$ have the same parity. One may wonder what would happen if we dualize $h$ to some multiple $r_i$ where $i$ and $k$ have different parity. Indeed, it is possible to do this, and Lemma \ref{lemma Alex} implies that if $r_i = r_j,$ then $D_{h, r_i}$ and $D_{h, r_j}$ are equivalent up to row exchanges. 
    
Indeed, dualizing a hat $h$ to a multiple $r_i$ with different parity from $k$ will always produce a row with $l \geq 1.$ To see this, suppose $r_1 = ([B, B], 0, \eta)$ and $r_2 = ([A, B], 0, \eta)$ are two consecutive rows with the same sign. The formulae indicate that $r_1$ stays the same under a row exchange except that its sign is changed by $(-1)^{C(r_2)},$ whereas $l(r_2)$ is increased by 1. This is the case regardless of the order of $r_1$ and $r_2.$

Since it does not matter which multiple a hat is dualized to, the convention will always be dualize to the last multiple, $r_k.$
\end{rmk}

The next lemma uses the observation that dualizing to each of the multiples produces the same multi-segment up to row exchanges in order to determine how many ways there are to dualize a given hat.

\begin{lemma}
\label{multiple duds}
    Let $h = ([A, -B], B, \eta)$ be a hat in $\EE=\mathcal{E}(\mathcal{M}, \mathcal{S}, \mathcal{T},\eta')$.
    \begin{enumerate}
        \item If $m_{B - 1} = 1,$ then there is a unique $dual \circ ui \circ dual$ involving $h$ and rows lower than $h.$
        \item If $m_{B - 1} \geq 3$ and $\mathcal{T}_{i}^0 = \{B - 1\} \subset \mathcal{S}_i$ for some $i,$ then there is exactly one $dual \circ ui \circ dual,$ unique up to row exchange, involving $h$ and rows lower than $h.$
        \item Otherwise, there are exactly two $dual \circ ui \circ dual$s involving $h$ and rows lower than $h,$ up to row exchanges.
    \end{enumerate}
\end{lemma}

\begin{proof}
    For Part (1), if $m_{B - 1} = 1,$ then there is exactly one row $r$ with support ending at $B - 1.$ If $r$ is a hat, it must be the row right after $h,$ so the unique $dual \circ ui \circ dual$ is $M_{h, r}.$ If $r$ is a row of circles, then the operation is $D_{h, r}$ by Lemma \ref{D existence general}.

    For Part (2), if $m_{B - 1} \geq 3,$ then there are $m_{B - 1}$ rows with support ending at $B - 1.$ Since there exists some $\mathcal{T}_{i}^0 = \{B - 1\},$ none of these rows can be a hat. Instead, one of them is a chain with support $[B - 1, B - 1],$ and the others are all multiples. Therefore, all these $m_{B - 1}$ rows are consecutive and equal to each other. By Lemma \ref{D existence general}, $h$ can be dualized to each of them, but per Remark \ref{equivalent D} all of these are equivalent up to row exchange.

    For Part (3), if $m_{B - 1} \geq 3$ and there does not exist any $\mathcal{T}_{i}^0 = \{B - 1\},$ then $B$ is contained in some $\mathcal{S}_i = \{\dots, B - 2, B-1, \overline{B, \dots, A}, \dots \}.$ Since two elements of $\mathcal{S}$ have intersection of size at most one, $B - 1\not\in \mathcal{S}_j$ for any $j\neq i$. Now, one of the following must be true.

    \begin{itemize}
        \item If $B - 1$ is contained in a hat, then this hat $h'$ and $m_{B - 1}$ multiples of support $[B - 1, B - 1]$ are precisely the $m_{B - 1}$ rows with support ending at $B - 1.$ One $dual \circ ui \circ dual$ is obtained by merging $h * h'.$ The other is gotten by dualizing $h$ to any of the multiples. Per Remark \ref{equivalent D}, these dualizations are the same up to row exchange since the multiples are all equal to each other.
        \item If $B - 1$ is not contained in a hat, then it must be contained in a chain of length $\geq 2,$ lest there be some $\mathcal{T}_i^0 = \{B - 1\}$.  Then one $dual \circ ui \circ dual$ is gotten by dualizing $h$ to this chain, and the other is gotten by dualizing it to any of the multiples.
    \end{itemize}
    From here, it is clear that the two $dual \circ ui \circ dual$s do not produce virtual extended multi-segments associated with the same Arthur parameters. This can be seen because the supports of the multi-segments are different. Indeed, the one obtained from the second case contains a row of support $[B, B - 1]$ and the one obtained from the first case cannot contain such a row.
\end{proof}

In summary, when $m_{B - 1} \geq 3,$ a hat can sometimes be dualized in two ways: either to a chain or to a multiple. Although these operations are quite similar, they impact the $\mathcal{S}$-data of $\EE$ very differently. Therefore, we will distinguish between them by notating a dualization to a chain as $D^1$ and a dualization to a multiple as $D^2.$ 

\begin{rmk}
\label{D S-data}
These dualizations change the $\mathcal{S}$-data as follows.

\begin{itemize}
    \item Applying $D_{h, r}^1$ to $h = ([A, -B], B, \eta_1)$ and $r = ([B - 1, C], 0, \eta_2)$ replaces
    $$\{C, \dots, B - 1, \overline{B, \dots, A}, \dots\} \longrightarrow \{C, \dots, B - 1, B, \dots, A, \dots\}.$$
    \item Applying $D_{h, r}^2$ to $h = ([A, -B], B, \eta_1)$ and a multiple $r = ([B - 1, B - 1], 0, \eta_2)$ replaces
    $$\{\dots, B - 1, \overline{B, \dots, A}, \dots\} \longrightarrow \{\dots, B - 1\}, \{B - 1, \dots, A, \dots\},$$
    thereby resulting in the creation of a $z$-chain.
\end{itemize}
\end{rmk}

Finally, we prove that the operation $U$ can be applied to $\EE$ in the same way structurally as if $\EE$ were of type $X_k.$

\begin{lemma}[Existence of $U$]
\label{U existence general}
    Suppose $\mathcal{E} = \mathcal{E}(\mathcal{M}, \mathcal{S}, \mathcal{T},\eta)$ begins at zero and that $h = ([A, -B], B, \eta)$ is a hat in $\EE$ with $C(h) > c$. Then there exists a series of operations on $\EE$ resulting in a new virtual extended multi-segment $\EE'$ of type $Y_\mathcal{M}$ satisfying the following
    \begin{itemize}
        \item $\EE'$ has a hat of the form $h' = ([A - c, -B], B, \eta)$,
        \item $\EE'$ has a row of $k$ circles with support $[A, A - c + 1]$,
        \item the other segments of $\EE'$ are precisely the same as the segments of $\EE - \{h\},$ except with possibly different signs, and
        \item $\eta(\EE') = \eta(\EE).$
    \end{itemize}
\end{lemma}

\begin{proof}
    Again, we first prove the statement in the case where $h = ([k, -B], B, \eta),$ where $k = c_{\max}.$ This consists of showing that the following combination of operations is valid.
    \begin{itemize}
        \item Row exchange $h$ until the image $h'$ is at the bottom of the multi-segment. $h'$ should be a row of circles.
        \item Apply $S_{h', c}.$
        \item Row exchange the first row of resulting circles back to the top.
    \end{itemize}
    
    Due to Lemma \ref{D existence general}, we know that $\EE \sm \{h\}$ is equivalent to a tempered representation through some series of row operations, all of which commute with row exchange by Corollary \ref{swap commutes with S} and Lemma \ref{lem-swap commutes with D}. Therefore, we can presume that $\EE \sm \{h\}$ is tempered. Furthermore, we can presume that $\EE \sm \{h\}$ has no multiples due to Lemma \ref{lem-multiplicity-cancel}. Thus, we have reduced to the type $X_k$ case, and the result follows from Lemma \ref{U existence}.

    The fact that $\eta(\EE)$ is preserved is clear from the fact that the sign of $h$ is unchanged throughout the row exchanges: the sign changes from exchanging down and exchanging up cancel each other by Lemmas \ref{big swap down} and \ref{big swap up}. These lemmas also show that the rows that $h$ is exchanged with all have their signs changed by $(-1)^c,$ so they satisfy the same sign conditions that they did before; therefore, the multi-segment is still odd alternating. That $h$ alternates with the row below it after the $U$ operation if and only if it does before the row operation is clear since $C(h)$ is also decreased by $c.$
    
    Finally, the fact that the broken-off row of $c$ circles alternates with the row before it can be seen in two cases. If the row before it is a chain, then the proof is exactly the same as in the type $X_k$ case (see the proof of Lemma \ref{U existence}). If the preceding row is a multiple with support $[C, C]$, it has the same sign as the last circle of the chain that it belongs to. Since $C \leq A - c < A,$ it is impossible for $C$ to be in more than one $ \mathcal{S}_i,$ so there are evenly many such multiples. Therefore, if we reduce to the case where the multiples are not present via Lemma \ref{lem-multiplicity-cancel}, it is clear that the row of $c$ circles alternates with the chain, so it must also alternate with the multiple. Thus, we observe that the image of this $U$ operation is odd-alternating and therefore of type $Y_\mathcal{M}.$
\end{proof}

\begin{rmk}
\label{U S-data}
    As for the $\mathcal{S}$-data, applying $U_{h, c}$ to $h = ([A, -B], B, \eta)$ replaces $$\{\dots \overline{B, \dots, A}, \dots \} \longrightarrow \{\dots \overline{B, \dots, A - c} \}, \{A - c + 1, \dots, A, \dots \}.$$
\end{rmk}

As with the other operators, the $U$ operator commutes with row exchange. Specifically, it follows from a combination of Parts (1) and (3) of Lemma \ref{lemma Alex} that the following corollary holds.

\begin{cor}
\label{swap commutes with U}
    Let $R$ denote the operation of exchanging some row $r$ with a series of rows $r_i < r_{i + 1} < \cdots < r_j.$ Let $T$ be some $U$ operation involving the rows $r_i < r_{i + 1} < \cdots < r_j.$ Then $R \circ T = T \circ R$.
\end{cor}

\subsection{\texorpdfstring{Exhaustion of Operators on $Y_\mathcal{M}$ Segments}{}}
\label{sec exhaustion of Y_M}

We now aim to prove Theorem \ref{block classification}. Due to Corollary \ref{Type Y SMUD equivalence}, we know that all type $Y_{\mathcal{M}}$ multi-segments with the same sign are equivalent. Then, it suffices to show that if $\EE$ is a virtual extended multi-segment of type $Y_\mathcal{M},$ and sign $\eta,$ then any virtual extended multi-segment equivalent to $\EE$ is also $Y_\mathcal{M}$ with sign $\eta.$ First we will prove that $\EE^{min}$ is of type $Y_\mathcal{M}$ with sign $\eta.$ Then we will show that all possible raising operators on type $Y_\mathcal{M}$ virtual extended multi-segments preserve type $Y_\mathcal{M}$ and sign. Since all equivalent virtual extended multi-segments can be obtained from $\EE^{min}$ through raising operators, this will suffice for our proof. We begin with the following simple technical lemma about row exchange.

\begin{lemma}
\label{lem-hat-row-swaps-easy}
    Let $h_1$ and $h_2$ be hats with order $h_1<h_2$, and suppose $C(h_1) = C(h_2) = 1$ and $\eta(h_1) = \eta(h_2)$. Then after row exchanging $h_1$ and $h_2$, we have $h_1' = h_1$ and $h_2' = h_2$.
\end{lemma}
\begin{proof}
Note that since $C(h_i)=1,$ we have that $A(h_i)\equiv B(h_i)\mod 2.$ Thus,
    the row exchange has $\epsilon = (-1)^{A(h_1) - B(h_1)} \eta(h_1) \eta(h_2) = 1$, $b(h_1) - 2 l(h_1) = 1$, and $2(b(h_2) - 2 l(h_2)) = 2$, so it falls into Case 1(a) of Definition \ref{def row exchange}. So $l(h_2') = l(h_2)$, $\eta(h_2') = \eta(h_2)$, $l(h_1') = A(h_1) - B(h_1) + 1 - (l(h_1) + 1) = l(h_1)$, and $\eta(h_1') = (-1)^{A(h_2) - B(h_2)} \eta(h_1) = \eta(h_1).$
    So the rows are unchanged.
\end{proof}

Now, we classify $\mathcal{E}^{\min}$ according to $\mathcal{S}$-data, breaking into cases depending on whether $\mathcal{M}$ starts at zero.

\begin{lemma}
\label{E min classification for blocks after zero}
    Suppose $\mathcal{E}$ is of type $Y_\mathcal{M}$ with $\mathcal{M}$ beginning after zero. Suppose the $\mathcal{S}$-data is such that $\mathcal{S}$ contains some $ \mathcal{S}_i = \{c_{\min}, \dots, c_{\max}\}.$ Then $\mathcal{E} = \mathcal{E}^{\min}.$
\end{lemma}

\begin{proof}
    We check that no lowering operators are possible on $\mathcal{E}.$ There are two cases: either $\mathcal{S} = ({S}_i)$ or $\mathcal{S} = (\{c_{\min}\}, \mathcal{S}_i).$ In the first case, $\mathcal{E}$ consists of one chain and multiples contained within the support of this chain. Clearly, no $ui$ operators involving the chain are possible. All the multiples have circles of the same sign, so it is impossible to apply a $ui$ of type 3'. In the second case, there are two chains, and the second chain contains all the other rows, which each have one circle. Again, it is evident that there are no $ui$ operators.
    
    Less obvious is the fact that there are no $dual \circ ui^{-1} \circ dual$ operators on $\mathcal{E}.$ To see this, note that $dual(\mathcal{E})$ consists of three different kinds of rows

    \begin{itemize}
        \item one hat $h$ with support $[c_{\max}, -c_{\min}]$,
        \item evenly many (possibly zero) hats below $h$ with support $[c_{\min}, -c_{\min}]$, or
        \item numerous hats with one circle each above $h$ corresponding to multiples belonging to the chain with support $[c_{\max}, c_{\min}].$
    \end{itemize}

    Since $c_{\min} > 0$, none of these rows have $l = 0$. In order to perform a $ui^{-1}$ of type 3', we must obtain a row with $l = 0$ through row exchanges. However, the rows above $h$ cannot be exchanged with $h$ or the rows below $h$ since no such pair of rows has one row's support containing the other. By Lemma \ref{lem-hat-row-swaps-easy}, none of the rows above $h$ are changed by any row exchanges with each other. Also, no row exchange involving only rows below $h$ changes $dual(\mathcal{E}).$ Thus, the only row exchanges that could possibly give an $l = 0$ are those involving $h.$ But $h$ can only be exchanged with the hats below, and since each of these hats has one circle of the same sign, by Lemma \ref{lem-multiplicity-cancel}, we need only consider what happens after swapping $h$ with any one of these rows, as doing two row exchanges leaves $h$ unchanged. Since $h$ fails the alternating sign condition with the row under it, this row exchange increases $l(h)$ by one and leaves the second row the same except for a possible sign change. Thus, no row exchanges create a row of circles, so no $ui^{-1}$ of type 3' is possible.
\end{proof}

\begin{lemma}
\label{E min classification}
    Suppose $\EE$ is of type $Y_\mathcal{M}$ where $\mathcal{M}$ begins at zero with $\mathcal{S}$-data $(\{0, \overline{1}, \overline{2}, \dots, \overline{k}\}).$ Then $\EE = \EE^{min}.$
\end{lemma}

\begin{proof}
    Note that $\EE$ consists of hats $([A, -A], A, \eta_A)$ for $A \in \{1, \dots, k\},$ a single chain $([0, 0], 0, \eta_0),$ and multiples. We check that there are no possible lowering operators on $\EE.$ 
    
    First, we check lowering operators of type 3'; i.e., lowering operators involving $ui$ or $ui^{-1}$ of type 3'. It is clear that the only $ui$s of type 3' are $S^{-1}$ and $U^{-1}$: $S^{-1}$ cannot be applied because all the multiples have the same sign as the single chain, and $U^{-1}$ cannot be applied because there is only one chain. Meanwhile, there are no $dual \circ ui^{-1} \circ dual$ operations of type 3' possible on $\mathcal{E}$ because every row $r$ in $\mathcal{E}$ has $C(r) = 1$: therefore, the same is true of $dual(\mathcal{E}),$ so no $ui^{-1}$ can be performed.

    The only lowering operator not of type 3' is a $ui.$ We observe that there are no $ui$ operators not of the type 3', since there are no rows whose supports overlap and do not contain each other. This can be seen because all the hats have supports containing each other, and all other rows have supports of one circle.
\end{proof}

Clearly, any $\mathcal{E}$ of type $Y_\mathcal{M}$ is equivalent to a virtual extended multi-segment of the form described in Lemma \ref{E min classification for blocks after zero} or Lemma \ref{E min classification}, since all $Y_{\mathcal{M}}$ multi-segments of the same sign are equivalent to the same tempered block via applications of $S$. Therefore, given $\mathcal{E}$ of type $Y_{\mathcal{M}},$ $\mathcal{E}^{\min}$ is always of type $Y_{\mathcal{M}}$ and sign $\eta(\mathcal{E}).$ 

Now, it suffices to check that all raising operators on $\mathcal{E}$ preserve sign and type $Y_\mathcal{M}$ form. Such raising operators must be of one of the following three forms
\begin{itemize}
    \item $ui^{-1}$ of type 3',
    \item $dual \circ ui \circ dual$ where the $ui$ is of type 3', or
    \item $dual \circ ui \circ dual$ where the $ui$ is not of type 3'.
\end{itemize} 
For the first of these, it is clear that $S$ and $U$ are the only $ui^{-1}$ operations of type 3'. For the second, when $\mathcal{M}$ starts at zero, any $dual \circ ui \circ dual$ of type 3' must have one of the rows it involves be a hat (lest the rows would have intersecting duals). If the other row is also a hat, then the operation is $M,$ and if the other operation is a row of circles, then the operation is $D.$ When $\mathcal{M}$ starts after zero, it is clear that there are no $dual \circ ui \circ dual$ operations of type 3' because the rows of $dual(\mathcal{E})$ are all hats and therefore have intersecting supports because they all contain zero. 

Since the operations $S, M, U,$ and $D$ all preserve type $Y_\mathcal{M}$ and $\eta,$ it is clear we only need to account for $dual \circ ui \circ dual$s not of type 3'. We will prove that such operations preserve type $Y_\mathcal{M}$ and $\eta$ through the following two lemmas. Lemma \ref{dud not 3' chain} will account for all such $dual \circ ui \circ dual$ operations between two rows which are not hats. Lemma \ref{dud not 3' hat} will account for $dual \circ ui \circ dual$ operations (when $\mathcal{M}$ starts at zero) between two rows, at least one of which is a hat.

\begin{lemma}
\label{dud not 3' chain}
    Suppose $T$ is a $dual \circ ui \circ dual$ not of type 3'. Suppose that $T$ involves two rows $r_1 < r_2$, neither of which are hats. Then $T$ preserves type $Y_\mathcal{M}$ and $\eta(\mathcal{E}).$
\end{lemma}

\begin{proof}
    Suppose $supp(r_1) = [A_1, B_1]$ and $supp(r_2) = [A_2, B_2].$ Since $r_1 < r_2,$ we presume that $B_1 \leq B_2.$ Then, $\widehat{r_1}$ and $\widehat{r_2}$ have supports $[A_1, -B_1]$ and $[A_2, -B_2],$ respectively. A $ui$ not of type 3' between these rows is possible only if $-B_2 < -B_1 \leq A_2 < A_1.$ This implies that $B_1 < B_2 \leq A_2 < A_1,$ so $supp(r_1) \supsetneq supp(r_2).$ Therefore, $r_2$ must be a multiple belonging to $r_1$ and we have $B_2 = A_2 < A_1.$ Since $A_2 \neq A_1,$ $r_2$ cannot be contained in the support of any chain other than $r_1.$ There are evenly many such multiples, and using $dual \circ ui \circ dual$ to combine $r_1$ with any of these multiples gives multi-segments that are equivalent up to row exchange. Therefore, we assume that $r_2$ is the last of these multiples.

    Applying such a $ui$ to $\widehat{r_1}$ and $\widehat{r_2}$ replaces them with rows that have support $[A_1, -A_2]$ and $[A_2, -B_1],$ in that order. Applying dual gives rows with supports $[A_2, B_1]$ and $[A_1, A_2],$ in that order. The resulting rows have the same order and support as the multi-segment obtained by implementing the following change on the $\mathcal{S}$-data of $\mathcal{E}$
    $$\{B_1, \dots, A_1, \dots\} \longrightarrow \{B_1, \dots, A_2\}, \{A_2, \dots, A_1, \dots\}.$$
    The multi-segment resulting from this change in $\mathcal{S}$-data is equivalent to $\mathcal{E},$ which is equivalent to $T \circ \mathcal{E}.$ Therefore, Corollary $\ref{cor same support}$ shows that $T  (\mathcal{E})$ is equal to this multi-segment and therefore has type $Y_\mathcal{M}$ and satisfies $\eta(T ( \mathcal{E})) = \eta(\mathcal{E}).$
\end{proof}

\begin{lemma}
\label{dud not 3' hat}
    Suppose $T$ is a $dual \circ ui \circ dual$ not of type 3'. Suppose that $T$ involves two rows $h < r,$ where $h$ is a hat. Then $T$ preserves type $Y_\mathcal{M}$ and $\eta(\EE).$
\end{lemma}

\begin{proof}
    If $h = ([A, -B], 0, \eta),$ then $\widehat{h}$ is a row of circles with support $[A, B].$ Note that if $h$ is associated to the set $\mathcal{T}_i^j,$ then $\mathcal{T}_i^j = \{B, \dots, A\}.$ Since all $\mathcal{T}_i^j$ for $j \geq 1$ are disjoint, we conclude that if $r$ is a hat, then $\supp(\widehat{h}) \cap \supp(\widehat{r}) = \emptyset,$ rendering the operation $T$ impossible. Therefore, $r$ has to be some row of circles with support $[C, D],$ which dualizes to a hat with support $[C, -D].$ In order to apply $T,$ we need to have $B \leq C < A.$ Given the intersection conditions on the sets $ \mathcal{S}_i$, this is only possible if $r$ is a multiple with support $[C, C].$ We note that since $C \in \{B, \dots, A\},$ there must be an even number of multiples equal to $r.$ Since performing $dual \circ ui \circ dual$ with each of these multiples is equivalent up to row exchanges, we presume that $r$ is the last row.

    We claim that performing $T$ will replace $h$ with a hat of the same sign, i.e., $([C, -B], B, \eta)$, and change $r$ into a chain with support $[A, C].$ This guarantees that $\eta(\EE)$ is preserved. All the while, the segment will remain odd-alternating. Given this claim, we see that $T$ implements the following change in the $\mathcal{S}$-data:
    $$\{\dots, \overline{B, \dots, A}, \dots\} \longrightarrow \{\dots, \overline{B, \dots, C} \}, \{C, \dots, A, \dots\},$$
    which is indeed valid. Thus, the claim implies $T$ preserves type $Y_\mathcal{M}$.
    
    To prove the claim, let $\widehat{r} = ([C, -C], C, \eta_1)$ and $\widehat{h} = ([A, B], 0, \eta_2).$ Should a $ui$ between these rows be possible, they would be replaced with rows of support $[A, -C]$ and $[C, B],$  which would exist in the same positions as $\widehat{r}$ and $\widehat{h}$ respectively in order to satisfy $P'$ order. Dualizing back, we obtain a multi-segment with the same order and support as $\mathcal{E}$ except:
    \begin{itemize}
        \item The row $r$ is replaced by a row with support $[A, C].$
        \item The hat $h$ is replaced by a row of support $[C, - B].$
    \end{itemize}
    Such a virtual extended multi-segment has precisely the same support and order as the multi-segment resulting from the change in $\mathcal{S}$-data described in the claim. These virtual extended multi-segments correspond to the same non-zero tempered blocks, so Corollary \ref{cor same support} suffices to show that they are equal.
\end{proof}

Having seen that all possible $dual \circ ui \circ dual$s not of case 3' preserve $Y_\mathcal{M}$ and $\eta(\EE),$ we conclude Theorem \ref{block classification}.

\section{\texorpdfstring{The Count $|\Psi(\pi(\mathcal{B}))|$}{}}\label{sec count block}

The goal of this section is to prove Theorem \ref{thm-count-block-temp}. We will do this by counting virtual extended multi-segments of type $Y_\mathcal{M}.$ In order to reduce to this case, we must establish a correspondence between $\Psi(\pi(\BB))$ and virtual extended multi-segments $\mathcal{E}$ of type $Y_\mathcal{M_\B}.$ In particular, different type $Y_{\mathcal{M}}$ multi-segments with different $\mathcal{S}$-data should be associated with different virtual extended multi-segments. This is guaranteed by the following lemma.

\begin{lemma}
\label{packets distinct}
    If $\EE_1$ and $\EE_2$ are both of type $Y_\mathcal{M}$ and have $\eta(\EE_1) = \eta(\EE_2),$ then $\psi(\EE_1) = \psi(\EE_2)$ if and only if $\EE_1 = \EE_2.$
\end{lemma}

\begin{proof}
    It is obvious that $\EE_1 = \EE_2$ implies $\psi(\EE_1) = \psi(\EE_2).$ Meanwhile, if $\psi(\EE_1) = \psi(\EE_2),$ then these virtual extended multi-segments must have the same support. The orders of their rows, furthermore, have already been prescribed and therefore must be the same. Since $\EE_1$ and $\EE_2$ are both $Y_\mathcal{M}$ and have $\eta(\EE_1) = \eta(\EE_2),$ we must have $\mathcal{E}_1 \sim \mathcal{E}_2.$ Corollary \ref{cor same support} now suffices to prove the desired result.
\end{proof}

\subsection{Blocks Starting at Zero}

We consider a ``lift'' for arbitrary $\rho$ in the case $\EE$ is of type $Y_{\mathcal{M}}$, with $\mathcal{M}$ starting at zero, as follows.

\begin{defn}\label{defn Theta_1(EE)}
    Suppose that $\EE\in\VRep_\rho^\mathbb{Z}(G_n)$ is of type $Y_\mathcal{M},$ with $\mathcal{M}$ starting at zero.
    By Theorem \ref{block classification} that there exists a tempered virtual extended multi-segment $\EE_{temp}$ which is equivalent to $\EE.$ We assume further that $\EE_{temp}$ starts at 0.
    
    We define
    \[
    \Theta_1(\EE)=\left\{\left(\left[c_{\max}+1,-c_{\max}-1\right]_{\rho}, {c_{\max}+1}, -\eta(\EE) \right)\right\}\cup \EE.
    \]
    
    We remark that the added extended segment should be inserted such that it is the first extended segment in the admissible order. 
\end{defn}

\begin{rmk}
    The above definition of $\Theta_1(\EE)$ is motivated by the calculation of the theta lift (for symplectic-even orthogonal dual pairs) to the first occurrence in the going-up tower for $\pi(\EE).$ This will be exploited in \cite{HKT}.
\end{rmk}

Suppose that $\EE\in\VRep_\rho^\mathbb{Z}(G_n)$ is tempered of type $Y_\mathcal{M}$ with sign $\eta(\EE)$. It is apparent from the definition that $\Theta_1(\mathcal{E})$ is of type $Y_{\mathcal{M}'}$ with $\eta(\Theta_1(\mathcal{E}))=-\eta(\EE),$ where $\mathcal{M}' = (m_0, \dots, m_k, 1).$

We will prove the following theorems, which will help us prove Theorem \ref{thm-count-block-temp}.

\begin{thm}
\label{first block m_n = 1}
    Let $\EE$ be of type $Y_\mathcal{M}$ where $\mathcal{M} = (m_0, \dots, m_{c_{\max} - 1}, 1).$
    Then the following holds.
    \begin{enumerate}
        \item  We can perform a $dual \circ ui \circ dual$ involving the first row of $\Theta_1(\EE)$ to get a virtual extended multi-segment which we denote by $\Theta_2(\EE).$ We can further perform an additional $ui^{-1}$ involving the $(c_{\max} + 1)$-st column of $\Theta_2(\mathcal{E})$ to obtain another virtual extended multi-segment which we denote by $\Theta_3(\EE).$
        \item We have the relation $$\Psi(\Theta_1(\EE)) = \{\psi_{\Theta_i(\EE')} \mid \mathcal{E}' \sim \mathcal{E}; i = 1, 2, 3 \}.$$
    \end{enumerate}
\end{thm}

\begin{thm}
\label{first block m_n > 1}
    Let $\EE$ be of type $Y_\mathcal{M}$ where $\mathcal{M} = (m_0, \dots, m_{c_{\max}})$ with $m_{c_{\max}} > 1.$
    Then the following holds.
    \begin{enumerate}
        \item We can perform two $dual \circ ui \circ dual$s of type 3' involving the first row of $\Theta_1(\mathcal{E})$: one with the first row whose support ends at $c_{\max},$ and one with the last row whose support ends at $c_{\max}.$ We call the resulting segments $\Theta_2(\mathcal{E})$ and $\Theta_4(\mathcal{E}),$ and we have $\psi_{\Theta_2(\mathcal{E})} = \psi_{\Theta_4(\mathcal{E})}$ if and only if the $\mathcal{S}$-data of $\EE$ has some $\mathcal{S}_i = \{c_{\max}\}.$ We can perform an additional $ui^{-1}$ to both $\Theta_2$ and $\Theta_4$ to remove a row of one circle from the row with support ending at $c_{\max} + 1$ to obtain the same $\Theta_3(\EE).$
        \item We have the relation $$\Psi(\pi(\Theta_1(\EE))) = \{\psi_{\Theta_i(\EE')} \mid \mathcal{E}' \sim \mathcal{E}; i = 1, 2, 3, 4 \}.$$
    \end{enumerate}
\end{thm}

We begin by proving Part (1) of Theorem \ref{first block m_n = 1} from the existence of the operators $D, U,$ and $S.$

\begin{proof}[Proof of Theorem \ref{first block m_n = 1} Part (1)]
    Suppose that $\EE$ is of type $Y_\mathcal{M}$ with final multiplicity $m_{c_{\max}} = 1.$ We know that $\Theta_1(\EE) := h \cup \EE,$ where
    $$h: = ([c_{\max} + 1, -c_{\max} - 1], c_{\max} + 1, -\eta(\EE)).$$

    Lemma \ref{multiple duds} states that there is exactly one $dual \circ ui \circ dual$ involving $h.$ There are two cases.

    \begin{itemize}
        \item First, $\Theta_2(\EE) = D_{h, r}\circ \Theta_1(\EE),$ for some chain $r.$ In this case, we let
        $$\Theta_3(\EE) := S_{r', 1} \circ \Theta_2(\EE),$$

        where $r'$ is the image of $r$ under $D_{h, r}.$
        \item Second, $\Theta_2(\EE) = M_{h, h'} \circ \Theta_1(\EE),$ where $h'$ is the first hat of $\EE.$ Here, we let
        \[\Theta_3(\EE) := U_{h * h', 1} \circ \Theta_2(\EE).\qedhere\]
    \end{itemize}
\end{proof}

The proof of Part (1) of Theorem \ref{first block m_n > 1} is similar. To complete it we will utilize the following lemma, which considers how the two different operations $D^1$ and $D^2$ can be related to each other.

\begin{lemma}
\label{supplementary commutativity properties}
    Suppose that $\EE$ is of type $Y_\mathcal{M}$ and has a hat $h = ([A, B], -B, \eta).$ Suppose $m_{B - 1} \geq 3,$ and let $r_1, r_2$ be the first and last rows ending at $B - 1.$

    \begin{enumerate}
        \item If $r_1$ is a hat, then
        $$U_{h * r_1, C(h)} \circ M_{h, r_1} = S_{r_2', C(h)} \circ D_{h, r_2}^2.$$
        \item If $r_1$ is a chain, then
        $$S_{r_1', C(h)} \circ D_{h, r_1}^1 = S_{r_2', C(h)} \circ D_{h, r_2}^2.$$
    \end{enumerate}

    Here, $r_1'$ and $r_2'$ indicate the images of $r_1$ and $r_2$ under $dual \circ ui \circ dual$s with $h.$
\end{lemma}

\begin{proof}
    For Part (1), let $r_1 = ([C, -B + 1], B - 1, \eta_1)$ and let $r_2 = ([B - 1, B - 1], 0, \eta_2).$ Then conducting $U_{h * r_1, C(h)} \circ M_{h, r_1}$ changes the $\mathcal{S}$-data of $\EE$ as follows (see Remarks \ref{M S-data} and \ref{U S-data}).
    \begin{align*}
        \{\dots \overline{C, \dots, B - 1}, \overline{B, \dots, A}, \dots\} &\longrightarrow \{\dots \overline{C, \dots, B - 1, B, \dots, A}, \dots\}\\
        &\longrightarrow \{\dots \overline{C, \dots, B - 1} \}, \{B, \dots, A, \dots\}.
    \end{align*}
    Meanwhile, applying $S_{r_2', C(h)} \circ D_{h, r_2}^2$ gives the changes (see Remarks \ref{S S-data} and \ref{D S-data}):
    \begin{align*}
        \{\dots \overline{C, \dots, B - 1}, \overline{B, \dots, A}, \dots \} &\longrightarrow \{ \dots \overline{C, \dots, B - 1} \}, \{B - 1, B, \dots, A, \dots\}\\
        &\longrightarrow \{ \dots \overline{C, \dots, B - 1} \}, \{B, \dots, A, \dots\},
    \end{align*}
    where the last arrow takes place due to the exception discussed in Remark \ref{S exception}. 
    
    For Part (2), let $r_1 = ([C, B - 1], 0, \eta_1).$ Then applying $S_{r_1', C(h)} \circ D_{h, r_1}^1$ gives:
    \begin{align*}
        \{C, \dots, B - 1, \overline{B, \dots, A}, \dots\} &\longrightarrow \{C, \dots, B - 1, B, \dots, A, \dots\}\\
        &\longrightarrow \{C, \dots, B - 1 \}, \{B, \dots, A, \dots\}.
    \end{align*}
    Meanwhile, applying $S_{r_2', C(h)} \circ D_{h, r_2}^2$ gives the changes:
    \begin{align*}
        \{C, \dots, B - 1, \overline{B, \dots, A}, \dots \} &\longrightarrow \{C, \dots, B - 1 \}, \{B - 1, B, \dots, A, \dots\}\\
        &\longrightarrow \{C, \dots, B - 1 \}, \{B, \dots, A, \dots\},
    \end{align*}
    which suffices to complete the proof.
\end{proof}

\begin{proof}[Proof of Theorem \ref{first block m_n > 1} Part (1)]
    By Lemma \ref{multiple duds}, given any $\Theta_1(\EE)$ where $m_{c_{\max}} \geq 3,$ it is possible to perform two different $dual \circ ui \circ dual$'s of type 3' involving the hat $h = ([c_{\max} + 1, -c_{\max} - 1], c_{\max} + 1, -\eta(\EE)),$ resulting in virtual extended multi-segments $\Theta_2(\EE)$ and $\Theta_4(\EE)$ respectively. Lemma \ref{multiple duds} states that $\Theta_2(\EE)$ and $\Theta_4(\EE)$ are equivalent via row exchanges if and only if the $\mathcal{S}$-data of $\EE$ contains $\{ c_{\max} \}$.

    Now, we consider two cases. If the multi-segment $\EE$ has a hat $h'$ with support containing $c_{\max}$, then the two $dual \circ ui \circ dual$'s associated with $h$ that produce $\Theta_2(\EE)$ and $\Theta_4(\EE)$ are $M_{h, h'}$ and $D^2_{h, r}$ respectively, where $r$ is a multiple. We can produce another virtual extended multi-segment from $\Theta_2(\EE)$ by applying $U_{h * h', 1}$, and we can obtain another virtual extended multi-segment from $\Theta_4(\EE)$ by applying $S_{r', 1},$ where $r'$ is the image of $r$ under $D_{h, r}^2.$ Part (1) of Lemma \ref{supplementary commutativity properties} implies that both of $ui^{-1}$'s produce the same virtual extended multi-segment which we call $\Theta_3(\EE).$

    In the case where the two operators that produce $\Theta_2(\EE)$ and $\Theta_4(\EE)$ are $D^1$ and $D^2,$ then we can perform a $ui^{-1}$ of the form $S$ to separate one circle at $c_{\max} + 1$ from each. Part (2) of Lemma \ref{supplementary commutativity properties} guarantees that both resulting multi-segments are the same, and again we let this be $\Theta_3(\EE).$
\end{proof}

Next we want to establish that, whenever $\mathcal{E}$ is of type $Y_\mathcal{M}$ and final multiplicity $m_{c_{\max}} = 1,$ no two elements of the set $\{ \Theta_i(\EE') \mid i \in \{1, 2, 3\}; \EE' \sim \EE \}$ have the same local Arthur parameter. To do this, we must simply argue that these extended multi-segments have different supports. In particular, we want to prove the following.

\begin{lemma}
\label{thetas distinct}
    Let $\mathcal{E}_1 \sim \mathcal{E}_2$ be of type $Y_{\mathcal{M}},$ with $\mathcal{M}$ starting after zero. Then
    $$supp(\Theta_i(\mathcal{E}_1)) = supp(\Theta_j(\mathcal{E}_2)) \implies i = j, \mathcal{E}_1 = \mathcal{E}_2.$$
\end{lemma}

\begin{proof}
    Starting with the case where $m_{c_{\max}} = 1,$ let $\mathcal{E}$ be of type $Y_\mathcal{M}$ and consider some multi-segment $\Theta_i(\mathcal{E})$ for $i \in \{1, 2, 3\}.$ Let $r \in \Theta_i(\mathcal{E})$ be the unique row whose support contains $c_{\max} + 1.$ Then we must be in one of the following cases:
\begin{itemize}
    \item If $r = ([c_{\max} + 1, -c_{\max} - 1], c_{\max} + 1, \eta)$ for some $\eta \in \{\pm 1\},$ then $i = 1.$
    \item If $C(r) > 1$ (i.e., $r$ is a merged hat or merged chain), then $i = 2.$
     \item If $r = ([c_{\max} + 1, c_{\max} + 1], 0, \eta),$ for some $\eta \in \{\pm 1\},$ then $i = 3.$
\end{itemize}
This confirms that $supp(\Theta_i(\mathcal{E}_1)) = supp(\Theta_j(\mathcal{E}_2)) \implies i = j,$ so it suffices to show that the multi-segment $\mathcal{E}$ can be uniquely determined from $\Theta_i(\mathcal{E}).$ This is obvious in the case where $i = 1$ since $supp(\mathcal{E}) = supp(\Theta_1(\mathcal{E}) \setminus \{r\}),$ as implied by Definition \ref{defn Theta_1(EE)}; note that the support suffices to identify $\mathcal{E}$ by Lemma \ref{packets distinct}. The statement $supp(\mathcal{E}) = supp(\Theta_1(\mathcal{E}) \setminus \{r\})$ is also clearly true for $i = 3.$ In the case $i = 2,$ $\mathcal{E}$ can be determined because $\Theta_3(\mathcal{E})$ can be determined: the operation from $\Theta_2(\mathcal{E})$ to $\Theta_3(\mathcal{E})$ is unique.

Meanwhile, in the case where $c_{\max} > 1,$ let $\mathcal{E}$ be of type $Y_\mathcal{M}.$ Any multi-segment $\Theta_i(\mathcal{E})$ must fall into one of the following slightly different cases:
\begin{itemize}
    \item If $r = ([c_{\max} + 1, -c_{\max}], k, \eta)$ for some $\eta \in \{\pm 1\},$ then $i = 1.$
    \item If $C(r) > 1$ and $|\{\mathcal{S}_i \mid c_{\max} \in \mathcal{S}_i \}| = 1,$ then $i(\EE) = 2.$
    \item If $C(r) > 1$ and $|\{\mathcal{S}_i \mid c_{\max} \in \mathcal{S}_i \}| = 2,$ then $i(\EE) = 4.$
    \item If $r = ([c_{\max} + 1, c_{\max} + 1], 0, \eta),$ then $i = 3.$
\end{itemize}
Again, this shows that the support of $\Theta_i(\mathcal{E})$ determines $i,$ except in the case when the unique chain containing $c_{\max} + 1$ has support $[c_{\max} + 1, c_{\max}]$. This case could occur both for $\Theta_2(\mathcal{E})$ and $\Theta_4(\mathcal{E}),$ but only when the unique chain in $\mathcal{E}$ containing $c_{\max}$ has support $[c_{\max}, c_{\max}].$ This however, is precisely the case where $\Theta_2(\mathcal{E}) = \Theta_4(\mathcal{E}).$ Meanwhile, $\mathcal{E}$ can be determined from $\Theta_i(\mathcal{E})$ again using the same methods as the case where $m_{c_{\max}} > 1.$ This suffices for the proof.
\end{proof}

Now, we prove Part (2) of Theorem \ref{first block m_n = 1}. Lemma \ref{thetas distinct} shows that the elements $\{\psi_{\Theta_i(\EE')} \mid i = 1, 2, 3; \ \EE' \sim \mathcal{E}\}$ are all distinct, and so now it suffices to show that any $\psi_\FF$, where $\FF \sim \Theta_1(\mathcal{E})$, belongs to this set. To do this, we examine the $\mathcal{S}$-data.

\begin{proof}[Proof of Theorem \ref{first block m_n = 1} Part (2)]
    Let $\mathcal{F} \sim \Theta_1(\mathcal{E}),$ where $\mathcal{E}$ is of type $Y_\mathcal{M}$ and $\mathcal{M} = (m_0, \dots, m_{c_{\max}}).$ Then Theorem \ref{block classification} and Part (1) of Theorem \ref{first block m_n = 1} imply that $\FF$ must be type $Y_{\mathcal{M}'},$ for $\mathcal{M}' = (m_0, \dots, c_{\max}, 1)$. Consider the $\mathcal{S}$-data of $\FF.$ Since $m_{c_{\max}} = m_{c_{\max} + 1} = 1,$ we note $c_{\max}$ can only belong to one set $ \mathcal{S}_i,$ and likewise with $c_{\max} + 1.$ 
    
    Let $\mathcal{E}' := \mathcal{E}(\mathcal{M}, \mathcal{S}', \mathcal{T}', \eta(\mathcal{E})$, where $(\mathcal{S}', \mathcal{T}')$ is obtained by removing $k + 1$ from $(\mathcal{S}, \mathcal{T}).$ From Theorem \ref{block classification}, we have $\mathcal{E}' \sim \mathcal{E}.$ We have the following cases.
    \begin{itemize}
        \item If both $k$ and $k + 1$ belong to the same set $\mathcal{S}_\ell = \{\dots, k, k + 1\},$ then clearly $\mathcal{F} = \Theta_2(\mathcal{E}').$
        \item If both $k$ and $k + 1$ belong to $\mathcal{S}_\ell = \{\dots, \overline{k, k + 1}\},$ then $\FF = \Theta_2(\mathcal{E}').$
        \item If both $k$ and $k + 1$ belong to $\mathcal{S}_\ell = \{\dots, \overline{k}, 
        \overline{k + 1}\},$ then $\FF = \Theta_1(\mathcal{E}').$
        \item If both $k$ and $k + 1$ belong to $\mathcal{S}_\ell = \{\dots, k,\overline{k + 1}\},$ then $\FF = \Theta_2(\mathcal{E}').$
        \item If $k + 1$ belongs to its own set $\mathcal{S}_{\ell} = \{k + 1\},$ then $\FF = \Theta_3(\mathcal{E}')$. \qedhere
    \end{itemize}
\end{proof}

The proof of Part (2) of Theorem \ref{first block m_n > 1} is quite similar, since again, in light of Lemma \ref{thetas distinct}, we only need to show that any $\FF \sim \Theta_1(\mathcal{E})$ is of the form $\Theta_i(\EE')$ for some $i \in \{1, 2, 3, 4\}$ and $\EE' \sim \mathcal{E}.$

\begin{proof}[Proof of Theorem \ref{first block m_n > 1} Part (2)]
    Let $\mathcal{F} \sim \Theta_1(\mathcal{E})$. Then again Theorem \ref{block classification} and Part (1) of Theorem \ref{first block m_n = 1} imply that $\FF$ must be of type $Y_{\mathcal{M}'},$ where $\mathcal{M}' = (m_0, \dots, c_{\max}, 1)$. Consider the $\mathcal{S}$-data of $\FF.$ Since $m_{k + 1} = 1,$ $k + 1$ can only belong to one set $ \mathcal{S}_i.$ However, $k$ can belong to at most two $\mathcal{S}_i,$ but only when the latter set has $\mathcal{T}_i^0 = \{k, k + 1\}.$ In such a case, we clearly have $\FF = \Theta_4(\EE'),$ where $\EE'$ is defined as in the proof of Part (2) of Theorem \ref{first block m_n = 1}. The cases where $k$ only belongs to one $ \mathcal{S}_i$ can be identified with various $\Theta_i(\EE')$ in exactly the same manner as the proof for Part (2) of Theorem \ref{first block m_n = 1}.
\end{proof}

Having proved Theorems \ref{first block m_n = 1} and \ref{first block m_n > 1}, we can finally prove Theorem \ref{thm-count-block-temp} for blocks $\BB$ that start at zero. To do this, we need the following lemma.

\begin{lemma}
\label{almost block reduction}
    Suppose $\EE$ is of type $Y_\mathcal{M},$ where $\mathcal{M} = (m_0, \dots, m_{c_{\max}}).$ Let $k\in\Z_{>0}$ and
    $$\EE' := \EE \cup \bigcup_{i = 1}^k ([c_{\max}, c_{\max}], 0, \eta')$$ 
    be the extended multi-segment obtained by $k$ rows of circles of support $[c_{\max}, c_{\max}]$ to $\EE.$ Here, the sign $\eta'$ is chosen to match the sign on the last circle of $\mathcal{E}.$ Then, a multi-segment $\mathcal{E}_1'$ is equivalent to $\mathcal{E}'$ if and only if $\EE_1' = \EE_1 \cup \bigcup_{i = 1}^k ([c_{\max}, c_{\max}], 0, \eta')$ for $\EE_1 \sim \EE.$ In particular, we obtain that
    $|\Psi(\pi(\EE'))| = |\Psi(\pi(\EE))|.$
\end{lemma}

\begin{proof}
    Per the proof of Theorem \ref{block classification}, we have seen that the raising operators on $\EE$ are precisely $S, M, U, D,$ and certain $dual \circ ui \circ dual$s not of type 3' that we have shown to be equivalent to some combination of $S, M, U, D,$ and their inverses (Lemmas \ref{dud not 3' chain} and \ref{dud not 3' hat}). It is clear from the definitions of $S, M, U, D,$ that all of these operators and their inverses can still be performed on the $\EE$ when considered as a sub-virtual extended multi-segment of $\EE'.$ This proves the forward direction of the desired statement. To prove the other direction, we claim that these are the only operators that can be applied to $\EE.'$

    We assume otherwise for the sake of contradiction, i.e, we assume that there exists an operator $T$ applicable on $\EE'$ which is not applicable on $\EE.$ Since none of the multiples $r_i = ([c_{\max}, c_{\max}], 0, \eta')$ allow the other rows of $\EE$ to interact with each other in new ways, such an operator $T$ would have to involve one of these multiples $r_i.$ We consider all possible operators. $T$ cannot be a $dual \circ ui \circ dual$ because $dual(r_i)$ has support $[c_{\max}, -c_{\max}],$ and all other rows in $dual(\EE')$ would be contained in this support. $T$ also cannot be a $ui^{-1}$ of type 3'. since these rows have only one circle. Meanwhile, $T$ obviously cannot be a $dual \circ ui^{-1} \circ dual$ of type 3' since each of these multiples has only one circle.

    The only remaining possibility is that $T$ is a $ui,$ in which case it must be one of type 3' since the rows $r_i$ have only one circle. $T$ would have to merge $r_i$ with $r,$ a row ending at $c_{\max} - 1.$ Consider the sets $\mathcal{S}_i$ containing $c_{\max} - 1$ or $c_{\max}.$ If we have a set $\{\dots, c_{\max} - 1, c_{\max}\},$ then $T$ is impossible because the only rows $r$ ending at $k - 1$ are multiples, which must have the same sign as the rows with support $[c_{\max}, c_{\max}].$ If we have a set $\{\dots, c_{\max} - 1, \overline{c_{\max}} \},$  then the sign of $r_i$ has been chosen so that it can be treated as a multiple belonging to $r$ and therefore has the same sign as the circle at $c_{\max} - 1,$ rendering $T$ impossible. If we have $\{c_{\max} - 1\}, \{c_{\max}\},$ then applying $T$ to merge $r$ and $r_i$ is the same as applying the operation to merge $c_{\max}$ with the chain of support $[c_{\max}, c_{\max}],$ which is already defined in $\EE.$ 

    We conclude that the virtual extended multi-segments equivalent of $\EE'$, up to row exchange, are in one-to-one correspondence with the virtual extended-multi-segments which are equivalent to $\EE$, up to row exchange. In particular, this means that there is a one-to-one correspondence between $\Psi(\pi(\EE))$ and $\Psi(\pi(\EE')).$
\end{proof}

Now, we are ready to prove Theorem \ref{thm-count-block-temp} for blocks that start at $0.$ We shall restate the result here.

\begin{thm}
    Let $\EE$ be a tempered almost-block with columns $0$ through $c_{\max} \geq 2$. Let $\EE_k$ denote the extended multi-segment consisting of only the rows with supports contained in $[k, 0]$. Then,
    $$|\Psi(\EE)| = \begin{cases}
        3 \cdot |\Psi(\pi(\EE_{c_{\max} - 1}))| & \ \mathrm{if} \ m_{c_{\max} - 1} = 1,\\
        4 \cdot |\Psi(\pi(\EE_{c_{\max} - 1}))| - |\Psi(\pi(\EE_{c_{\max} - 1}))| &  \ \mathrm{if} \ m_{c_{\max} - 2} \geq 3.
    \end{cases}$$
\end{thm}

\begin{proof}
    Let $\EE'$ be the extended multi-segment obtained by removing all but one of the rows in $\EE$ of support $[c_{\max}, c_{\max}].$ Then $\EE'$ is a block of type $Y_{\mathcal{M}}$ for some $\mathcal{M}$ ending at $c_{\max}.$ From Lemma \ref{almost block reduction}, we have $|\Psi(\pi(\EE))| = |\Psi(\pi(\EE'))|.$ Therefore, we may assume for the sake of the proof that $\EE$ is a block with $m_{c_{\max}} = 1.$ In this case, we have $\EE= \Theta_3(\EE_{c_{\max} - 1})$ up to twisting by the appropriate supercuspidals. If $m_{c_{\max} - 1} = 1,$ the fact that $|\Psi(\pi(\EE))| = 3 \cdot |\Psi(\pi(\EE_{c_{\max} - 1}))|$ follows immediately from Part (2) of Theorem \ref{first block m_n = 1}.

    Meanwhile, Part (2) of Theorem \ref{first block m_n > 1} gives us that if $m_{c_{\max} - 1} > 1,$ then
    $$|\Psi(\pi(\EE))| = 4 \cdot |\Psi(\pi(\EE_{c_{\max} - 1}))| - R,$$
    where $R$ is the cardinality of the set $\{ \EE \mid \EE \sim \EE_{c_{\max} - 1}, \Theta_2(\EE) = \Theta_4(\EE)  \}.$ According to part (2) of Theorem \ref{first block m_n > 1}, these are precisley the multi-segments with $\{k - 1\}$ in their $\mathcal{S}$-data.

    Applying Lemma \ref{almost block reduction} again allows us to assume each of the $\EE \in \{ \EE \mid \EE \sim \EE_{c_{\max} - 1}, \Theta_2(\EE) = \Theta_4(\EE)\}$ has $m_{c_{\max} - 1} = 1.$ In that case, these are precisely the multi-segments equal to $\Theta_3(\EE'')$ for $\EE'' \sim \EE_{k - 2}.$ The cardinality of this set is precisely $|\Psi(\pi(\EE_{c_{\max} - 2}))|.$
\end{proof}

\subsection{Blocks not starting at zero}

In this subsection, we use Theorem \ref{block classification} to derive Equation \eqref{eq second recursion} in Theorem \ref{thm-count-block-temp}, which we restate below for convenience.

\begin{thm}[count for blocks not starting at zero]
\label{thm-count-block-temp-in-paper}
    Let $\BB$ be a block starting at $c_{min}$ and ending at $c_{max}$, and suppose that $c_{min}>0$. Let $ \BB' := \rc_{c_{max}}(\BB) $ and $ \BB'' := \rc_{c_{max}-1}(\BB').$
    Then
    \[
    |\Psi(\pi(\mathcal{B}))| = \begin{cases} 2 |\Psi(\pi(\mathcal{B}'))| & \text{if } m_{c_{max}-1} = 1, \\ 3 |\Psi(\pi(\mathcal{B}'))| - |\Psi(\pi(\mathcal{B}''))| & \text{if } m_{c_{max}-1} = 3, 5, \dots. \end{cases}
    \]
\end{thm}
\begin{proof}
    First we consider the case that $m_{c_{max}-1} = 1$. From Theorem \ref{block classification}, the set $\Psi(\pi(\BB))$ is in bijection with the set of valid $\mathcal{S}$ for the block-tuple $\mathcal{M}_\mathcal{B}$, and analogously for the set $\Psi(\pi(\BB'))$. To prove the formula, we will partition the set of valid $\mathcal{S}$ for $\mathcal{M}_\mathcal{B}$ into two sets, $\Psi_1$ and $\Psi_2$, each of which is in bijection with the set of valid $\mathcal{S}'$ for $\mathcal{M}_{\mathcal{B}'}$. The set of valid $\mathcal{S}'$ for $\mathcal{M}_{\mathcal{B}'}$ is in turn in bijection with $\Psi(\pi(\BB'))$ by Theorem \ref{block classification}, so this suffices to prove the formula.

    The first set $\Psi_1$ consists of the valid $\mathcal{S} = (\mathcal{S}_1, \dots, \mathcal{S}_k)$ with $\mathcal{S}_k = \{c_{max}\}$. In this case, if we let $\mathcal{S}' = (\mathcal{S}_1', \dots, \mathcal{S}_{k-1}') := (\mathcal{S}_1, \dots, \mathcal{S}_{k-1}),$ then we claim $\mathcal{S}'$ is valid for $\mathcal{M}_{\BB'}$. To see this, we simply check each of the conditions in Definition \ref{valid S}. Condition (1) is clear. Since $\mathcal{M}_{\mathcal{B}'} = (m_{c_{min}}, \dots, m_{c_{max}-1})$, Condition (2) holds. Finally, Conditions (3), (4), and (5) are strictly weaker for $\mathcal{S}'$ compared to $\mathcal{S}$. 
    
    Moreover, we claim that given $\mathcal{S}' = (\mathcal{S}_1', \dots, \mathcal{S}_{k-1}')$ valid for $\mathcal{M}_{\BB'}$, \[\mathcal{S} := (\mathcal{S}_1', \dots, \mathcal{S}_{k-1}', \{c_{max}\})\] is valid for $\mathcal{M}_\BB$. Conditions (1), (2) of Definition \ref{valid S} are clear. Condition (3) holds because every element of any $\mathcal{S}_i'$ is at most $c_{max}-1$, since $\BB'$ ends at $c_{max}-1$. Condition (4) is clear because we only need to check it in the case that $j=k$, but $\mathcal{S}_k$ does not intersect with any other $\mathcal{S}_i$. Finally, Condition (5) is clear because $|\mathcal{S}_k| < 2$, and it holds for $\mathcal{S}'$.

    The second set $\Psi_2$ consists of all other valid $\mathcal{S}$. In this case due to Condition (3) of Definition \ref{valid S}, it must be that $c_{max} \in \mathcal{S}_k$. Since $\mathcal{S}\not\in\Psi_1$, $\mathcal{S}_k$ does not consist solely of $c_{max}$, and so we have $|\mathcal{S}_k| \geq 2$. In this case, if we let \[\mathcal{S}' = (\mathcal{S}_1', \dots, \mathcal{S}_k') := (\mathcal{S}_1, \dots, \mathcal{S}_{k-1}, \mathcal{S}_k \setminus \{c_{max} \}),\] then we claim $\mathcal{S}'$ is valid for $\mathcal{M}_{\BB'}$. Condition (1) holds because $|\mathcal{S}_k| \geq 2$, so $\mathcal{S}_k' \neq \emptyset$. Conditions (2) and (3) are clear, and Conditions (4) and (5) hold for the same reason as before: they are strictly weaker for $\mathcal{S}'$ than for $\mathcal{S}$.

    Moreover, we claim that given $\mathcal{S}' = (\mathcal{S}_1', \dots, \mathcal{S}_k')$ valid for $\mathcal{M}_{\BB'}$, \[\mathcal{S} = (\mathcal{S}_1', \dots, \mathcal{S}_{k-1}', \mathcal{S}_k' \cup \{c_{max}\})\] is valid for $\mathcal{M}_{\BB}$. Conditions (1), (2), (3) are clear. Condition (4) is clear because we only need to check the case of $\mathcal{S}_k' \cup \{c_{max}\}$ intersecting with other $\mathcal{S}_i'$, but since $c_{max} \notin \mathcal{S}_i'$, it must be that $\mathcal{S}_i'$ and $\mathcal{S}_k'$ have nontrivial intersection, so it follows that $k-i=1$, $|\mathcal{S}_k'| \geq 2$, and $m_c > 1$, where $c \in \mathcal{S}_i' \cap \mathcal{S}_k'$. So certainly $|\mathcal{S}_k' \cup \{c_{max}\}| \geq 2$ as well. For Condition (5), it suffices to check the case where $\mathcal{S}_k'$ does not satisfy the hypotheses of Condition (5) in $\mathcal{S}'$ but $\mathcal{S}_k' \cup \{c_{max}\}$ does satisfy the hypotheses. Since $|\mathcal{S}_k' \cup \{c_{max}\}| \geq 2$ while $|\mathcal{S}_k'| < 2$, it must be that $\mathcal{S}_k' = \{c_{max}-1\}$. But $m_{c_{max}-1} = 1$ by assumption, so the hypotheses of Condition (5) are not satisfied, so the condition is satisfied. This completes the bijections for $m_{c_{max}-1} = 1$ case of Theorem \ref{thm-count-block-temp-in-paper}.

    Second we consider the case that $m_{c_{max}-1} > 1$. We will partition the set of valid $\mathcal{S}$ for $\mathcal{M}_\BB$ into three sets $\Psi_1, \Psi_2,$ and $\Psi_2$. The first two sets will both be in bijection with the set of valid $\mathcal{S}'$ for $\mathcal{M}_{\BB'}$. The third set $\Psi_2$ will be in bijection with a subset of the set of valid $\mathcal{S}'$ for $\mathcal{M}_{\BB'}$, with the complement of this subset being in bijection with the set of valid $\mathcal{S}''$ for $\mathcal{M}_{\BB''}$. Together with Theorem \ref{block classification}, this proves the formula.

    The three sets $\Psi_1,\Psi_2,\Psi_3$ are as follows. Let $\mathcal{S} = (\mathcal{S}_1, \dots, \mathcal{S}_k)$ be valid for $\mathcal{M}_\BB$. We have that $\mathcal{S}$
    \begin{itemize}
        \item lies in $\Psi_1$ if $\mathcal{S}_k = \{c_{max} \}$,
        \item lies in $\Psi_2$ if $c_{max}-1$ appears in more than one $\mathcal{S}_i$, or
        \item lies in $\Psi_3$ if $c_{max}-1$ appears in exactly one $\mathcal{S}_i$, say $\mathcal{S}_{i_0}$, and $c_{max} \in \mathcal{S}_{i_0}$.
    \end{itemize}
    Observe that these three sets partition the set of valid $\mathcal{S}$. Certainly $\Psi_2\cap\Psi_3=\emptyset$. If $\mathcal{S}\not\in\Psi_2\cup\Psi_3$, then $c_{max}-1$ appears in exactly one $\mathcal{S}_i$, say $\mathcal{S}_{i_0}$, and $c_{max} \notin \mathcal{S}_{i_0}$. By Condition (3) of Definition \ref{valid S}, the set $\mathcal{S}_{i_0+1}$ can contain only $c_{max}$. By Condition (4), there can be no further sets $\mathcal{S}_j$ with $j>i_0+1$. Hence $k = i_0+1$, and we have $\mathcal{S}_k = \{c_{max}\}$, so $\mathcal{S}$ lies in the first set.

    Now we construct our bijections. The bijection for $\Psi_1$ is exactly the same as before, when $m_{c_{max}-1} = 1$, since we did not use this fact.

    For $\Psi_2$, let $\mathcal{S}$ be valid for $\mathcal{M}_\BB$. If $c_{max}-1$ appears in more than on $\mathcal{S}_i$, then by Condition (3) it appears in exactly two, say $\mathcal{S}_{i_0}$ and $\mathcal{S}_{i_0+1}$. By Condition (4), we have $|\mathcal{S}_{i_0+1}| > 1$, so $\mathcal{S}_{i_0+1} = \{c_{max}-1, c_{max} \}$. By the same reasoning as before, there can be no further sets $\mathcal{S}_j$ with $j>i_0+1$, or in other words $k=i_0+1$. Then we claim \[\mathcal{S}' = (\mathcal{S}_1', \dots, \mathcal{S}_{k-1}') := (\mathcal{S}_1, \dots, \mathcal{S}_{k-1})\] is valid for $\mathcal{M}_{\BB'}$. Condition (1) is clear. Condition (2) follows from the fact that $c_{max}-1 \in \mathcal{S}_{k-1}'$, and the corresponding condition holds for $\mathcal{S}$. Conditions (3), (4), (5) all follow directly from the corresponding conditions for $\mathcal{S}$.

    Moreover, we claim that given $\mathcal{S}' = (\mathcal{S}_1', \dots, \mathcal{S}_{k-1}')$ valid for $\mathcal{M}_{\BB'}$, then \[\mathcal{S} = (\mathcal{S}_1', \dots, \mathcal{S}_{k-1}', \{c_{max}-1, c_{max}\} )\] is valid for $\mathcal{M}_\BB$. Conditions (1), (2), (3) are clear. For Condition (4), it suffices to check overlaps between $\{c_{max}-1, c_{max}\}$ and some other $\mathcal{S}_i'$. Since the condition holds for $\mathcal{S}'$, the only possible such overlap is with $\mathcal{S}_{k-1}'$, with intersection $\{c_{max}-1\}$. Then it is indeed the case that $|\{c_{max}-1, c_{max}\}| \geq 2$ and $m_{c_{max}-1} > 1$. Finally, for Condition (5) it again suffices to check the condition for $\{c_{max}-1, c_{max}\}$, which does indeed satisfy the hypotheses. We have $c_{max}-1 \in \mathcal{S}_{k-1}'$ since $\mathcal{S}'$ is valid for $\mathcal{M}_{\BB'}$ and therefore satisfies Condition (2), (3).

    Finally, we construct our bijection between $\Psi_3$ and a subset of the valid $\mathcal{S}' = (\mathcal{S}_1', \dots, \mathcal{S}_k')$ for $\mathcal{M}_{\BB'}$. The subset consists of those $\mathcal{S}'$ satisfying the following property.
    \[\tag{Q} \text{If } \min(\mathcal{S}_k') = c_{max}-1, \text{then } c_{max}-1 \in \mathcal{S}_{k-1}'.\] First suppose we have a valid $\mathcal{S}$ for $\mathcal{M}_\BB$ lying in $\Psi_3$. If $c_{max}-1$ appears in only $\mathcal{S}_{i_0}$ and $c_{max} \in \mathcal{S}_{i_0}$, then let $\mathcal{S}' = (\mathcal{S}_1', \dots, \mathcal{S}_k') := (\mathcal{S}_1, \dots, \mathcal{S}_{k-1}, \mathcal{S}_k \setminus \{c_{max} \}).$ Again it is clear that the conditions for $\mathcal{S}'$ to be valid for $\mathcal{M}_{\BB'}$ are strictly weaker than the conditions for $\mathcal{S}$ to be valid for $\mathcal{M}_\BB$. We also observe that $\mathcal{S}'$ always has property (Q), because if $\min(\mathcal{S}_k \setminus \{c_{max}\}) = c_{max}-1$, then since $m_{c_{max}-1}>1$, by Condition (5) we have $c_{max}-1 \in \mathcal{S}_{k-1}'$.

    Moreover, suppose we have $\mathcal{S}'$ valid for $\mathcal{M}_{\BB'}$ satisfying property (Q). Then we claim $\mathcal{S} = (\mathcal{S}_1', \dots, \mathcal{S}_k' \cup \{c_{max} \})$ is valid for $\mathcal{M}_\BB$. Conditions (1), (2), (3) are clear. For Condition (4), it suffices to check the case where one of the sets is $\mathcal{S}_k' \cup \{c_{max} \}$. Since $\mathcal{S}'$ is valid, this can only have nontrivial intersection with $\mathcal{S}_{k-1}'$. If $c$ is this intersection, then the fact that $m_c>1$ follows from the fact that Condition (4) holds for $\mathcal{S}'$. Also, $|\mathcal{S}_k' \cup \{c_{max}\}| \geq 2$. Finally, for Condition (5) it suffices to check the set $\mathcal{S}_k' \cup \{c_{max}\}$. If $\min(\mathcal{S}_k') < c_{max}-1$, then the hypotheses of Condition (5) for this set are equivalent to those for $\mathcal{S}_k'$ in $\mathcal{S}'$. So the condition for $\mathcal{S}$ is also satisfied. If $\min(\mathcal{S}_k') = c_{max}-1$, then Condition (5) follows from the fact that $\mathcal{S}'$ has property (Q).

    To complete the proof, we construct a bijection between the $\mathcal{S}'$ not satisfying property (Q) and the $\mathcal{S}''$ valid for $\mathcal{M}_{\BB''}$. Given $\mathcal{S}' = (\mathcal{S}_1', \dots, \mathcal{S}_k')$ not satisfying property (Q), we claim $\mathcal{S}'' = (\mathcal{S}_1', \dots, \mathcal{S}_{k-1}')$ is valid for $\mathcal{M}_{\BB''}$. Since $\mathcal{S}'$ does not satisfy property (Q), we have $\min(\mathcal{S}_k') = c_{max}-1$ and $c_{max}-1 \notin \mathcal{S}_{k-1}'$. Hence Condition (2) follows. Conditions (1), (3), (4), (5) are clear from the fact that $\mathcal{S}'$ is valid.

    Moreover, given $\mathcal{S}'' = (\mathcal{S}_1'', \dots, \mathcal{S}_{k-1}'')$ valid for $\mathcal{M}_{\BB''}$, we claim \[\mathcal{S}' = (\mathcal{S}_1'', \dots, \mathcal{S}_{k-1}'', \{c_{max}-1\})\] is valid for $\mathcal{M}_{\BB'}$. Conditions (1), (2), (3) are clear. Note that no additional constraints are imposed by Conditions (4) and (5) applied to $\mathcal{S}'$ since $\{c_{max}-1\}$ cannot intersect any other $\mathcal{S}_i''$, and its size is not at least 2, respectively. Finally, observe that this map is the inverse of the previous one because $\mathcal{S}'$ not satisfying property (Q) implies that $\mathcal{S}_k' = \{c_{max}-1\}$.
\end{proof}

\section{Interactions of Blocks}\label{sec interation of blocks}

In this section, we consider the cases where a tempered $\EE\in\VRep_\rho^\mathbb{Z}(G_n)$ may consist of multiple blocks (by Lemma \ref{lem-unique-block-decomp}, we can always decompose a tempered $\EE\in\VRep_\rho^\mathbb{Z}(G_n)$ into disjoint blocks).

\subsection{Existence of operations on blocks}\label{sec Existence of operations on blocks}

In this subsection, we study the effects operators on blocks.
Intuitively, the valid operations on a block in a virtual extended multi-segment are the same as if the block was its own virtual extended multi-segment. (A similar statement holds if the extended multi-segment is not tempered and we are considering one of the pieces of its decomposition.) The only exception to this is that, for the blocks after the first block, they behave as if they do not start at zero, even if they do. This is what Lemma \ref{lem-later-blocks-exhaustion} amounts to. Lemmas \ref{lem-first-block} and \ref{lem-later-blocks-existence} show the validity of all of the operations, as the intuitive picture might suggest (with this additional adjustment).

The following definition formalizes the idea of a block which starts at zero behaving as if it does not start at zero.

\begin{defn}
    Let $\EE$ be a block starting at zero of type $Y_\mathcal{M}$, where $\mathcal{M} = (m_0, \dots, m_{c_{\max}})$. We say $\EE$ is of type $Y_{\mathcal{M}}^{>0}$ if $\EE = sh^{-1}(\EE')$ for some $\EE'$ of type $Y_{\mathcal{M}'}$, where $\mathcal{M}' = (m_1', \dots, m_{c_{\max}+1}')$ with $m_i' = m_{i-1}$. (When $\EE$ is a block not starting at zero of type $Y_{\mathcal{M}}$, this condition automatically holds, so we also say $\EE$ is of type $Y_{\mathcal{M}}^{>0}$.)
\end{defn}

\begin{lemma}
\label{lem-later-blocks-exhaustion}
    Let $\EE\in \VRep_\rho^\mathbb{Z}(G_n)$ be equivalent to a tempered virtual extended multi-segment $\EE_{temp}$. Suppose that $\EE_{temp}$ has a block decomposition $\BB_{1, temp} \cup \dots \cup \BB_{k, temp}$ with $k>1$ and that $\EE$ has decomposition $\BB_1 \cup \dots \cup \BB_k$ with $\BB_i$ equivalent to $\BB_{i, temp}$. Suppose further that $\BB_k$ is of type $Y_\mathcal{M}^{>0}$. We claim the following.
    \begin{enumerate}
        \item Every operation on $\EE$ which is an operation on $\BB_k$ (in the sense that it only involves rows from $\BB_k$) is analogous to an operation on the extended multi-segment $sh^1(\BB_k)$, in the following way:
        \begin{itemize}
            \item the operation $ui$ on rows $r_1$ and $r_2$ in $\EE$ corresponds to the operation $ui$ on rows $sh^1(r_1)$ and $sh^1(r_2)$ in $sh^1(\BB_k)$;
            \item the operation $dual$ on $\EE$ corresponds to the operation $dual$ on $sh^1(\BB_k)$;
            \item and the operation $ui\inv$ of type 3' on $\EE$ which splits off $k$ circles from row $r$ corresponds to the operation $ui\inv$ of type 3' in $sh^1(\BB_k)$ which splits off $k$ circles from $sh^1(r)$.
        \end{itemize}
        \item After any operation on $\EE$ which is an operation on $\BB_k$, $\BB_k$ is still of type $Y_\mathcal{M}^{>0}$.
    \end{enumerate}
\end{lemma}

\begin{proof}
    Except for the case of $dual \circ ui\inv \circ dual$ operations, the proof is similar to the proof of exhaustion in  \S\ref{sec exhaustion of Y_M}. The main difference is that we prove a slightly stronger claim here, which is a characterization of the operators possible at each step, rather than a characterization of all equivalent extended multi-segments. Therefore we do not find the minimal extended multi-segment with respect to the admissible order, so in our exhaustion step we must also consider lowering operators.

    For operations of the form $dual \circ ui \circ dual$, the same support argument as before (see discussion before Lemma \ref{dud not 3' chain}) holds to show that no operations are possible except between a chain and a multiple belonging to it. In this case the $dual \circ ui \circ dual$ operation only involves rows between the chain and the multiple in question. In particular, it only involves rows in $\BB_k$. So applying the same reasoning as in the proof of Lemma \ref{dud not 3' chain}, the operation preserves type $Y_{\mathcal{M}}^{>0}$.

    For operations of the form $dual \circ ui\inv \circ dual$ where the $ui\inv$ is of type 3', we claim that no such operations are possible. First suppose a row $\widehat{r}$ in $dual(\EE)$ is exchanged up. Then any such operation only involves rows in $dual(\BB_k)$, so it would be a valid operation on the extended multi-segment $\BB_k$. But no such operations are possible, by a similar reasoning as in Lemma \ref{E min classification for blocks after zero}, applied to each chain and multiples contained in the support of the chain.
    
    So we only consider a row $r$ in $\EE$ which is exchanged down in $dual(\EE)$. If $\BB_k$ does not start at zero, then again we can imitate the proof of Lemma \ref{E min classification for blocks after zero}: in short, after exchanging a row in the dual, it will always have $l > 0$, so no $ui\inv$ of type 3' are possible. 
    
    On the other hand, if $\BB_k$ starts at zero, then no $dual \circ ui\inv \circ dual$ operations are possible, but the reason is slightly different. Let $r$ be the row in $\EE$ to which $ui\inv$ is eventually applied. Note that since $\BB_k$ starts at zero, we must have $k=2$ and $\BB_1$ must consist of some number of circles in column $0$. 
    
    We first claim that if $\widehat{r}$ is not exchanged to the last row in $dual(\EE)$, no $ui\inv$ is possible. Suppose $\widehat{r}$ is split into two rows, $\widehat{r_1}$ and $\widehat{r_2}$, with $\widehat{r_1}<\widehat{r_2}$ in the resulting admissible order. If a column $c$ is in the support of $\widehat{r_1}$, then $c$ is not in the support of $\widehat{r_2}$, and so $B(\widehat{r_2})>0$. But since $\BB_1$ has a circle in column $0$, the last row of $dual(\EE)$ is a single circle in column $0$ and it comes after $\widehat{r_2}$. So the order is not admissible. If on the other hand $\rho$ is not in the support of $\widehat{r_1}$, then $A(\widehat{r_1})<0$ which is of course impossible.
    Hence the $ui\inv$ operation was not valid.

    However, if $\widehat{r}$ is exchanged to the last row in $dual(\EE)$, we claim $ui\inv$ is still not possible. Since a $dual \circ ui\inv \circ dual$ would be inverse to the $D$ operation, by similar reasoning as the existence of the $D$ operation (Lemma \ref{D existence general}) but in reverse, since $\BB_2$ is of type $Y_\mathcal{M}$, when $\widehat{r}$ is exchanged all the way to the last row of $dual(\BB_2)$, it has only circles. 
    Since $\BB_1$ has an odd number of circles, by Lemma \ref{lem-multiplicity-cancel}, we can assume without loss of generality that it has exactly one circle. But the row exchange is of Case 1(b) and therefore after the last row exchange we must have $l \geq 1$ and so no $ui\inv$ is possible.

    For operations of the form $ui\inv$ of type 3', these must be of type S, which have already been addressed above.

    Finally, for operations of the form $ui$, since they do not depend on the column, the valid $ui$ operations on $\BB_k$ are precisely the valid $ui$ operations on $sh^1(\BB_k)$. These all preserve type $Y_\mathcal{M}^{>0}$ by definition.
\end{proof}

\begin{rmk}
    The intuitive reason why the second block $\BB_2$ behaves the same as a block not starting at zero is that, in both cases, it is impossible to form hats via a $dual \circ ui\inv \circ dual$ operation. Hence the only extended multi-segments are ones that can be obtained by doing $ui$'s of type 3' together with row exchanges.
\end{rmk}

The previous lemma allows us to make the following very important observation about the supports of each part of the decomposition.

\begin{lemma}
\label{lem-staircase}
    Let $\EE_{temp}$ have a block decomposition $\BB_{1, temp} \cup \dots \cup \BB_{k, temp}$, and let $\EE$ be equivalent to $\EE_{temp}$. Suppose $\EE$ has a decomposition $\BB_1 \cup \dots \cup \BB_k$ with $\BB_i$ equivalent to $\BB_{i, temp}$. Then this decomposition has the property that for $i_1 < i_2$, and rows $r_1 \in \BB_{i_1}$, $r_2 \in \BB_{i_2}$, we have $A(r_1) \leq B(r_2).$
\end{lemma}
\begin{proof}
    Note that in $\EE_{temp}$, for $i_1 < i_2$ and any row $r_1 \in \BB_{i_1, temp}$ and $r_2 \in \BB_{i_2, temp}$, we have $0 \leq B(r_1) \leq A(r_1) \leq B(r_2) \leq A(r_2).$ Since $\EE$ is equivalent to $\EE_{temp}$ and each of the $\BB_i$ are equivalent to $\BB_{i, temp}$, we can get $\EE$ by applying raising operations and their inverses to each of the $\BB_{i, temp}$ taking them to $\BB_i$. Note that applying raising operations and their inverses on $\BB_{i_1}$ do not change the maximum of $A(r)$ across rows $r \in \BB_{i_1}$. Moreover, the only raising operations or their inverses on $\BB_{i_2}$ which might change the minimum of support is $dual \circ ui\inv \circ dual$ of type 3', but from Lemma \ref{lem-later-blocks-exhaustion} we know these are not possible. So for $r_1 \in \BB_{i_1}$ and $r_2 \in \BB_{i_2}$ with $i_1<i_2$, we have that $A(r_1) \leq B(r_2).$
\end{proof}

\begin{defn}
    We call the above property in Lemma \ref{lem-staircase} the \emph{staircase} property of the decomposition. In future sections, we will also consider other sorts of decompositions arising in different ways, and refer to the same property as the \emph{staircase} property.
\end{defn}

Next we show that any operator on a block induces an operator on the virtual extended multi-segment.

\begin{lemma}
\label{lem-first-block}
    Let $\EE\in\VRep_\rho^\mathbb{Z}(G_n)$ be an extended multi-segment equivalent to a tempered extended multi-segment, and let $\EE$ have decomposition $\BB_1 \cup \dots \cup \BB_k$, with the same conditions as before. Then any operation on $\BB_1$ as its own extended multi-segment is also a valid operation on $\EE$.
\end{lemma}
\begin{proof}
    Let $T$ be a raising operator or its inverse on $\BB_1$. Any raising operator or its inverse is one of $dual \circ ui \circ dual$, $ui\inv$ of type 3', $dual \circ ui\inv \circ dual$ of type 3', or $ui$. Since the conditions for $ui$ are purely local (see Lemma \ref{lem-local}), if $T$ is a union-intersection, then it is clearly also valid on $\EE$. Similarly, observe that $dual(\EE) = dual(\BB_k) \cup \dots \cup dual(\BB_1)$ possibly up to a global sign change on each part $dual(\BB_i)$. Hence if $T$ is an operation of the form $dual \circ ui \circ dual$, then it is also valid on $\EE$, since a global sign change on $dual(\BB_1)$ does not affect whether a union-intersection is valid.

    Now we consider the case of $ui\inv$ of type 3'. Let $\BB_1'$ be such that $ui(\BB_1') = \BB_1$, where the union-intersection is of type 3'. Since union-intersection is local, it is always the case that $ui(\BB_1' \cup \dots \cup \BB_k) = \BB_1 \cup \dots \cup \BB_k = \EE$, unless $\BB_1' \cup \dots \cup \BB_k$ is not in admissible order. Since $ui\inv$ is a valid operation on $\BB_1$, we see that $\BB_1'$ is in admissible order. Also, each $\BB_i$ for $i>1$ is in admissible order. So it suffices to check that for a row $r_1$ from $\BB_1'$ and a row $r_2$ from $\BB_{i_2}$ for $i_2>1$, $r_1$ and $r_2$ are in admissible order. Since $r_2$ comes after $r_1$, this can only fail if $A(r_1) > A(r_2)$ and $B(r_1) > B(r_2)$. But since the decomposition has the staircase property, we always have that $B(r_2) \geq A(r_1)$. So $B(r_2) \geq B(r_1)$, so the order is admissible. Hence we conclude $ui\inv$ is a valid operation on $\EE$.

    Finally we consider the case of $dual \circ ui\inv \circ dual$ where the $ui\inv$ is of type 3'. Here the exact same argument applies, but to $dual(\EE)$. The only possible obstruction to $ui\inv$ being a valid operation is that the result is not in admissible order. Again it suffices to check rows $r_1$ from $dual(\BB_1)$ and $r_2$ from $dual(\BB_{i_2})$ for $i_2>1$. However, every row in $dual(\BB_{i_2})$ has support containing the support of every row in $dual(\BB_1)$, since the decomposition has the staircase property. Since this is still true even after a $ui\inv$, it is impossible for these rows to fail to satisfy the admissibility condition.
\end{proof}

\begin{lemma}
\label{lem-later-blocks-existence}
    Let $\EE_{temp}$ be a tempered extended multi-segment with a block decomposition $\BB_{1, temp} \cup \dots \cup \BB_{k, temp}$, and let $\EE$ be any equivalent extended multi-segment, with decomposition $\BB_1 \cup \dots \cup \BB_k$. Suppose that $\BB_k = \BB_{k, temp}$, and let $\BB_k'$ be an extended multi-segment equivalent to $\BB_k$ of type $Y_\mathcal{M}^{>0}$. Then there exists a sequence of operations taking $\EE$ to $\EE' = \BB_1 \cup \dots \cup \BB_{k-1} \cup \BB_k'$.
\end{lemma}
\begin{proof}
    Note that of the $S$, $M$, $U$, $D^1$, and $D^2$ operators, the only one which applies to extended multi-segments of type $Y_\mathcal{M}^{>0}$ is $S$, since such extended multi-segments have no hats. By Corollary \ref{Type Y SMUD equivalence}, it suffices to show that the $S$ and $S^{-1}$ operations on $\BB_k$ are valid as an operation on $\EE$, where we only assume $\BB_k$ is of type $Y_{\mathcal{M}}^{>0}$. We can then repeatedly apply $S$ or $S^{-1}$ operations to take $\EE$ to $\EE'$.

    Since $S$ and $S^{-1}$ are local, it just remains to check that they do not create rows violating the admissible order condition. This is clear for $S^{-1}$ since it is a $ui$ operation. 
    
    For $S$ operations, since $\BB_k$ is the last piece of the decomposition, the only way this could happen is if it created a row with $A$ smaller than a row above it. But a $ui\inv$ operation does not change the minimum value of $A$ among the rows of $\BB_k$, so it cannot create a non-admissible order.

    Also, the $S$ operation preserves type $Y_\mathcal{M}^{>0}$ since it is local, and the $S$ operation is also valid on $sh^1(\BB_k)$, since it does not depend on the column. So this follows from the existence of $S$ on a single block (Lemma \ref{S existence general}).
\end{proof}

\subsection{Independence of blocks}

In this subsection, we prove that the blocks which occur in the block decomposition a tempered $\EE\in\VRep_\rho^\mathbb{Z}(G_n)$ (see Lemma \ref{lem-unique-block-decomp}) determine $\Psi(\pi(\EE))$ (see Proposition \ref{prop-independence-of-blocks}). We begin by giving some lemmas. The first lemma allows us to only consider certain sequences of row exchanges.

\begin{lemma}
\label{lem-condense-row-swaps}
    Let $\EE\in\VRep_\rho^\mathbb{Z}(G_n)$ and let $\EE'$ be the result after any number of row exchanges on $\EE$. Then there exists a row of $\EE$ so that, when it is exchanged monotonically (i.e. either all up or all down) until it is the $i$-th row, it is the same as the $i$-th row of $\EE'$.
\end{lemma}
\begin{proof}
    Let $r$ be the row of $\EE$ which, after the row exchanges, becomes the $i$th row of $\EE'$. By commutativity of row exchanges (Lemma \ref{lem-comm}), when performing the row exchanges taking $\EE$ to $\EE'$, we can equivalently perform all the row exchanges involving $r$ first, and then perform the other row exchanges. Since the other row exchanges do not involve $r$, and are performed last, they do not affect the $i$th row, so without loss of generality we can omit them. So we are left with a sequence of row exchanges only involving $r$. Since row exchanging is its own inverse, after removing pairs of row exchanges which are their own inverses, we can assume that the sequence of row exchanges either moves $r$ monotonically up or monotonically down. Hence the $i$th row of $\EE'$ is the same as the $i$th row if we only performed row exchanges monotonically on $r$.
\end{proof}

Our second lemma concerns row exchanges within blocks of type $Y_\mathcal{M}$, and follows from the theorems in Section \ref{sec-individual-blocks}.

\begin{lemma}
\label{lem-aj}
    Let $\BB$ be any virtual extended multi-segment equivalent to a block $\BB_{temp}$.
    \begin{enumerate}
        \item Suppose $\BB$ starts at $0$. Suppose that there is a row $r$ such that it is possible to row exchange it until it is the last row, and let $r'$ be its image after those row swaps. Then $l(r') = 0$. Moreover, the sign of the last circle in $\BB$ is the same as the sign of the last circle in $r'$.
        \item Suppose $\BB$ does not start at $0$. Suppose that there is a row $r$ such that it is possible to row exchange it until it is the first row, and let $r'$ be its image. Then $l(r') = 0$, and the sign of the first circle in $\BB$ is the same as the sign of the first circle in $r'$.
    \end{enumerate}
\end{lemma}
\begin{proof}
    For Part (1), first suppose $\mathcal{B}_{temp}$ starts at $0$ and ends at $k$. Then $\mathcal{B}$ is of type $Y_\mathcal{M}$ by Theorem \ref{block classification}.
    
    If $r$ is a hat, then we use similar reasoning as in the proof of Lemma \ref{U existence}. In particular, we can assume that all rows following $r$ have $C=1$. Moreover, by Lemma \ref{lem-multiplicity-cancel} we can assume that the multiplicity of every column in $\BB_{temp}$ is $1$. Hence without loss of generality $\mathcal{B}_{temp}$ is therefore of type $X_k.$ If $r'$ is the image of $r$ after it is exchanged to the bottom, the fact that $l(r') = 0$ follows from the proof of Lemma \ref{U existence}. By the alternating sign condition, $ \eta(r) = \eta(\BB) \cdot (-1)^{\sum_{s < r} C(s)}.$ Lemma \ref{big swap down} implies that the sign of the last circle in $r'$ is \[\eta(r) \cdot (-1)^{\sum_{s > r} C(s)} \cdot (-1)^{C(r') - 1}.\] Excluding $r$, there are $k + 1 - C(r)$ rows $s$ in $\BB$ each with $C(s) = 1,$ so the sign of the last circle is \[ \eta(\BB) (-1)^{\sum_{s \in \mathcal{B} - \{r\}} C(s)} (-1)^{C(r')-1} = (-1)^{k + C(r') - C(r)} \eta(\BB).\] But $C(r') - C(r)$ is even since $r$ and $r'$ have the same support, so the sign is $(-1)^k \eta(\BB),$ which is precisely the sign of the last circle of $\mathcal{B}.$
    
    If $r$ is not a hat, then $r$ is a chain or a multiple. Again it is only possible to exchange $r$ with the rows following it if $\supp(r)$ contains their supports, which means that all rows $s > r$ must be multiples. We presume $r$ is a chain, since otherwise $r$ is trivially unchanged by row exchanges with identical multiples. Since there is no chain after $r$, there must be an even number of each multiple; thus, by Lemma \ref{lem-multiplicity-cancel}, exchanging $r$ with these rows leaves $r$ unchanged, so $l(r') = l(r) = 0$ as desired. The sign of the last circle of $r'$ is the same as the last circle of $r$, which is the same as the sign of the last circle in $\mathcal{B}_{temp}$ (and $\mathcal{B}$) due to the odd-alternating condition.

    For part (2), if $\BB_{temp}$ is a block not starting at zero, we break into two cases. First suppose $\supp(r)$ has length at least $2$. Then it must be a chain. Moreover, its support cannot be contained in the support of a row before it, since by the axioms of $\mathcal{S}$-data (Definition \ref{defn E(M,S)}) the intersection of any two supports has size at most $1$. Since we can exchange $r$ all the way to the top, $\supp(r)$ must contain the support of every row before it. For the same reason, this implies that every row before $r$ must be a single circle. Since $r$ is a chain, by the axioms of type $Y_\mathcal{M}$, all the circles must lie in column $B(r)$ and have the same sign as $\eta(r)$, and there must be an even number of circles. So by Lemma \ref{lem-multiplicity-cancel}, $r$ is unchanged when it is exchanged to the top. So it has the same sign as the sign of $\BB$.

    Second suppose $\supp(r)$ has length $1$. If $\supp(r)$ contains the support of all the rows above it, then it must be a multiple of a chain of length $1$. Since these all have the same sign, row exchanges have no affect. Otherwise, $\supp(r)$ is contained in the support of all the rows above it, which means it belongs to a chain $s$ of length at least $2$. Since $\supp(r)$ is contained in the support of all the rows above it, every row coming before $r$ other than $s$ must be another multiple in the same column as $r$. So without loss of generality $r$ immediately follows the chain $s$. Then we can compute from Definition \ref{def row exchange} that $l(r') = 0$ and $\eta(r') = \eta(s)$, using the fact that $\BB$ is odd-alternating.
\end{proof}

Finally, we need the following lemmas, which will be helpful when we consider operations which involve row exchanges. This lemma allows us to treat the row exchange as if the relevant row was exchanged all the way to the top/bottom, by truncating the extended multi-segment.

\begin{lemma}
\label{lem-truncations}
    Let $\EE$ be of type $Y_\mathcal{M}$, and let $\EE^{\tc}$ be a truncated virtual extended multi-segment containing all but the first $i$ rows of $\EE$. Then regardless of $i$, the extended multi-segment $\EE^{\tc}$ is equivalent to a tempered virtual extended multi-segment where the multiplicity of each column is at most the corresponding multiplicity in $\EE_{temp}$. Moreover, the operations realizing this equivalence can be performed on $\EE$.
\end{lemma}
\begin{proof}
    Observe that $\EE$ can be transformed to $\EE_{temp}$ using two types of operations: undualizing hats, and splitting rows of circles (Corollary \ref{Type Y SMUD equivalence}). Undualizing a hat only involves rows between the hat and the row to which it is undualized. In particular, for any hat in $\EE^{\tc}$, since the truncation removes rows from the top, the corresponding row is also in $\EE^{\tc}$. So the undualization is a valid operation on $\EE^{\tc}$. Similarly, we see that splitting rows of circles can only involve a row of circles and the rows following it, so it is again a valid operation on $\EE^{\tc}$. So we can undualize all the hats in $\EE^{\tc}$, if there are any, and then split rows of circles until the result is tempered. 
    
    After performing all of these operations, we are left with the truncated rows followed by a tempered virtual extended multi-segment equivalent to $\EE^{\tc}$. The operations which take this virtual extended multi-segment to $\EE_{temp}$ only increase the multiplicity of each column, since undualizing hats increases multiplicities, while splitting rows of circles leaves them unchanged. Hence the multiplicity of each column in the extended multi-segment equivalent to $\EE^{\tc}$ must be at most the multiplicity of the corresponding column in $\EE_{temp}$.
\end{proof}

We are now read to begin the proof of independence of blocks. A key observation is that due to the maximality of blocks, there are only a few configurations that can occur between adjacent blocks. The following definition of type 1, type 2, and type 3 boundaries formalizes this.

\begin{defn}
    Suppose $\EE\in\VRep_\rho^\mathbb{Z}(G_n)$ is  equivalent to a tempered virtual extended multi-segment $\EE_{temp}$, and suppose $\EE_{temp}$ has a block decomposition $\BB_{1, temp} \cup \cdots \cup \BB_{k, temp}$. Suppose that $\EE$ has a decomposition $\BB_1 \cup \cdots \cup \BB_k$ with $\BB_i$ equivalent to $\BB_{i, temp}$. In this setup we give a classification of the ways in which $\BB_k$ could start, which depend on how it support overlaps with the support of $\BB_{k-1}$, which we call the \emph{boundary} between $\BB_k$ and $\BB_{k-1}$. Let $H_{col}$ be the last nonempty column in $\BB_{k-1}$ and let $N_{col}$ be the first nonempty column in $\BB_k$.
    \begin{itemize}
        \item A \emph{type 1 boundary} occurs when $N_{col} > H_{col}+1$.
        \item A \emph{type 2 boundary} occurs when $N_{col} = H_{col}$.
        \item A \emph{type 3 boundary} occurs when $N_{col} = H_{col}+1$.
    \end{itemize}
\end{defn}
    
Note that in the case of a type 2 boundary, the sign of the last circle in $\BB_{k-1, temp}$ must be the same as the sign of the first circle in $\BB_{k, temp}$, since $\BB_{k, temp}$ started a new block. Similarly in case of a type 3 boundary, the sign of the last circle in $\BB_{k-1, temp}$ must be the same as the sign of the first circle in $\BB_{k, temp}$. See Figure \ref{figure-boundary-types} below.
{\tiny
    \begin{figure}[ht]
        \begin{minipage}[t]{.3\textwidth}
            \centering
            \caption*{Type 1}
            $\bordermatrix{ & {\scriptstyle 0} & \cdots &  H_{col} & \cdots  & N_{col} & \cdots \cr 
            & & \ddots & & & & \cr
            & & & \oplus & \cdots & & \cr 
            & & & & & \ominus & \cr
            & & & & & & \ddots}$
        \end{minipage} \hfill
        \begin{minipage}[t]{.3\textwidth}
            \centering
            \caption*{Type 2}
            $\bordermatrix{ & 0 & \cdots & H_{col} &  & \cdots \cr 
            & & \ddots & & & \cr
            & & & \oplus & & \cr
            & & & \oplus & & \cr
            & & & & \ominus & \cr
            & & & & & \ddots }$
        \end{minipage} \hfill
        \begin{minipage}[t]{.3\textwidth}
            \centering
            \caption*{Type 3}
            $\bordermatrix{ & 0 & \cdots & H_{col} & N_{col} & \cdots \cr 
            & & \ddots & & & \cr
            & & & \oplus & & \cr
            & & & & \oplus & \cr
            & & & & & \ddots }$
        \end{minipage}
        \caption{Boundary Types}
        \label{figure-boundary-types}
    \end{figure}
    }
    
We also make a few observations about row exchanges across a boundary. A row exchange cannot occur across a boundary (i.e. between rows of the virtual extended multi-segments in question) of type 1 or type 3, since in those cases the supports of the two extended multi-segments are disjoint. 

Moreover, observe that in the case of a type 2 boundary, say between $\EE_1$ and $\EE_2$, at most one row exchange can occur across the boundary. Let $H_{col}$ be the last column of $\EE_1$, which is the same as the first column of $\EE_2$, and let $\EE_{1, temp}$ and $\EE_{2, temp}$ be the tempered extended multi-segments equivalent to $\EE_1$ and $\EE_2$. Note that $\EE_{2, temp}$ has exactly one circle in the first column of $\EE_{2, temp}$, since otherwise the circles would be included in $\EE_{1, temp}$. Since none of the operations can increase the support of a row, and none of the operations can create two rows with support in a column from one row with support in a column, $\EE_2$ can only have one row with support including the first column of $\EE_{2, temp}$, which must be $H_{col}$. By similar reasoning, the support of a row in $\EE_1$ cannot extend beyond $H_{col}$. So the only way for a row exchange to occur is with a single circle in $H_{col}$ from $\EE_2$, after which no other row exchanges can occur because every other row in $\EE_2$ has support not including $H_{col}$.

Now we are ready to state the main lemma of this subsection.

\begin{lemma}
\label{lem-independence-of-blocks} 
    Suppose that $\EE\in\VRep_\rho^\mathbb{Z}(G_n)$ is equivalent to a tempered virtual extended multi-segment $\EE_{temp}$. Let $\EE_{temp}$ have block decomposition $\BB_{1, temp} \cup \dots \cup \BB_{k, temp}$. Suppose that $\EE$ has a decomposition $\BB_1 \cup \dots \cup \BB_k$ with $\BB_i$ equivalent to $\BB_{i, temp}$. Then any raising operator $T$ (or inverse) on $\EE$ cannot involve (even by row exchanges) rows from both $\BB_{i_1}$ and $\BB_{i_2}$ for $i_1 < i_2$.
\end{lemma}
\begin{proof}
    Recall from Theorem \ref{block classification} that any virtual extended multi-segment equivalent to a block is of type $Y_\mathcal{M}$, up to row exchanges. Since these row exchanges do not involve rows from different parts of the decomposition, we can assume without loss of generality that each $\BB_i$ is of type $Y_\mathcal{M}$.
    
    We have four possible operations: $ui\inv$ of type 3', $dual \circ ui\inv \circ dual$ where the $ui\inv$ is of type 3', $dual \circ ui \circ dual$, and $ui$.

    \underline{Operations of the form $ui\inv$ of type 3'}

    First we show that no $ui\inv$ operation of type 3' can occur involving a row from $\BB_{i_1}$ and a row from $\BB_{i_2}$. Note that a general $ui\inv$ operation of type 3' consists of some number of row exchanges, followed by $ui\inv$, followed by the inverse row exchanges. Consequently, we need to rule out the row exchanges affecting another block.
    Suppose $ui\inv$ is applied to a row $r$. By Lemma \ref{lem-condense-row-swaps}, it suffices to consider only monotonic row exchanges involving $r$.
    
    The only way this operation could involve a row from $\BB_{i_1}$ and a row from $\BB_{i_2}$ is if $r$ lies in $\BB_{i_1}$ and it is exchanged into a position in $\BB_{i_2}$, or vice versa. However we know that exchanges across a boundary of type 1 and type 3 are impossible, and the only exchange across a boundary of type 2 is a row exchange with exactly one circle in $H_{col}$. So suppose we have one of these row exchanges. By Lemma \ref{lem-aj}, a row becomes a row of circles after it is exchanged to the end of one of the $\BB_i$. In order for the row exchange to be nontrivial, we need $r$ to have at least two circles, since it is being exchanged with a row with exactly one circle. (Otherwise, both rows would have to be supported in the same column, which would mean they are identical.) Therefore we lie in Case 1(b) or Case 2(b) of Definition \ref{def row exchange}, which means that after the row exchange, $l(r') = 1$. In particular, no $ui\inv$ of type 3' on $r'$ is possible.

    \underline{Operations of the form $dual \circ ui\inv \circ dual$ of type 3'}
    
    Second we show that no $dual \circ ui\inv \circ dual$ operation involving a row from $\BB_{i_1}$ and a row from $\BB_{i_2}$ is possible. Again by Lemma \ref{lem-condense-row-swaps} it suffices to consider two possibilities: either a row $\widehat{r}$ from $dual(\BB_{i_1})$ is exchanged up into $dual(\BB_{i_2})$ (recall from Definition \ref{def dual} that the order is reversed), or a row $\widehat{r}$ from $dual(\BB_{i_2})$ is exchanged down into $dual(\BB_{i_1})$.
    
    When $r$ lies in $\BB_{i_1}$ and we exchange $\widehat{r}$ up, no $dual \circ ui\inv \circ dual$ operation is possible due to support reasons. Observe that $\EE$ is in $(P')$ order. This is because each $\BB_i$ is of type $Y_\mathcal{M}$, so it is in $(P')$ order, and the block decomposition has the staircase property that for a row $r_1$ lying in $\BB_{i_1}$ and a row $r_2$ lying in $\BB_{i_2}$ for $i_1 < i_2$, we have $A(r_1) \leq B(r_2)\leq A(r_2)$.

    Since $\EE$ is in $(P')$ order, so is $dual(\EE)$. Now suppose $\widehat{r}$ is exchanged with some row, say $\widehat{s}$, before the $ui\inv$ occurs. Since the order is $(P')$, and $\widehat{r}$ is being exchanged up (by assumption), we must have $B(\widehat{s}) \leq B(\widehat{r})$. Let $\widehat{r}'$ be the image of $\widehat{r}$ after the row exchanges, and let $\widehat{r_1}$ and $\widehat{r_2}$ be the result of applying the $ui\inv$ operation to $\widehat{r}'$, with $B(\widehat{r_1}) < B(\widehat{r_2})$. Since the operation only row exchanges $\widehat{r_2}$, after the $ui\inv$ operation, the row $\widehat{r_1}$ comes before $\widehat{s}$. Since $dual$ requires the extended multi-segment to have $(P')$ order, we must have $B(\widehat{r_1}) \leq B(\widehat{s})$. But $B(\widehat{r_1}) = B(\widehat{r})$, so we conclude $B(\widehat{s}) = B(\widehat{r})$. However we know the row exchange between $\widehat{r_2}$ and $\widehat{s}$ is possible. By the above reasoning the first column of $\widehat{r}$ lies in the support of $\widehat{s}$, but it clearly does not lie in the support of $\widehat{r_2}$. So we have a contradiction.

    When $r$ lies in $\BB_{i_2}$ and we exchange $\widehat{r}$ down, we can apply a similar analysis as in the proof of Lemma \ref{U existence}, except that this time the operation is not possible. By truncating the extended multi-segment, we can assume that $\widehat{r}$ is exchanged to the last row. By Lemma \ref{lem-truncations} the truncation is equivalent to a tempered extended multi-segment. Since the decomposition satisfies the staircase property (Lemma \ref{lem-staircase}) and $\BB_{i_2}$ is $(P')$ order (as it is type $Y_\mathcal{M}$), the support of $\widehat{r}$ must contain the support of all rows below it. So by Corollary \ref{cor Alex} we can assume the truncation is indeed tempered, or in other words that $\widehat{r}$ is swapped with rows with $C=1$. Let $\widehat{r}'$ be the image of $\widehat{r}$ after all row exchanges, or all but one row exchange in the case that the multiplicity in the first column of the reduction is even. By examining the cases of row exchange, we see that in these cases $l$ cannot decrease by more than 1 for every row exchange. Moreover, the only cases in which $l$ decreases by exactly 1 is in Case 1(c), or Case 1(a) and $C(\widehat{r}) = 0$. Since a row exchange with two hats with identical widths leaves $\widehat{r}$ unchanged (Lemma \ref{lem-multiplicity-cancel}), in order for $l(\widehat{r}')$ to be $0$, every row exchange (ignoring multiplicities) must decrease $l$ by exactly $1$. Moreover, there must be a row exchange with a row with every width less than $|B(\widehat{r})|$. This is impossible if $\widehat{r}$ encounters a boundary of type 1.

    In the case of a boundary of type 2, we note that the row exchanges with the duals of the circles in $H_{col}$ leave $\widehat{r}$ unchanged. In particular, the net effect is that there are no row exchanges with hats with width $H_{col}$. So once again we have a contradiction.
    
    In the case of a boundary of type 3, we see that since the circles in $H_{col}$ and $N_{col}$ fail the alternating sign condition, so do their corresponding hats in the dual (Lemma \ref{Alternating dual}). In particular, the image of $\widehat{r}$ after row exchanges will fail the alternating sign condition with the hat of width $H_{col}$. In order for $l(\widehat{r}')$ to be $0$, by the reasoning above this forces the row exchange to be Case 1(a) and have $C(\widehat{r}) = 0$. However the latter is impossible since $C(\widehat{r})$ starts nonzero, and only increases at each step. Hence in any case $\widehat{r}$ cannot be exchanged across a boundary and still have $l(\widehat{r}') = 0$. This shows no $dual \circ ui\inv \circ dual$ of type 3' is possible.

    \underline{Operations of the form $dual \circ ui \circ dual$}

    Third we show that no operations $dual \circ ui \circ dual$ can occur between a rows in $\BB_{i_1}$ and $\BB_{i_2}$. Since the decomposition has the staircase property (Lemma \ref{lem-staircase}), for any $r_1 \in \BB_{i_1, temp}$ and $r_2 \in \BB_{i_2, temp}$, $\supp(\widehat{r_2}) \supset \supp(\widehat{r_1})$. So there is no nontrivial union-intersection possible between different blocks.

    \underline{Operations of the form $ui$}
    
    Fourth we show that no $ui$ operations can occur between a row $r_1$ in $\BB_{i_1}$ and a row $r_2$ in $\BB_{i_2}$. We have two cases based on $i_1.$
    
    \textbf{Case 1.} First suppose that the row $r_1$ lies in some $\BB_i$ for $i<i_2-1$. Note that in the case of a type 1 or type 3 boundary, the maximum of the supports of all rows in the first extended multi-segment is at least 1 less than the minimum of the supports of all rows in the second extended multi-segment. Also, in the case of two consecutive type 2 boundaries, the middle extended multi-segment must have at least two columns, since blocks are maximal. So the maximum of the supports of the rows in the first extended multi-segment is at least 1 less than the minimum of the supports of the rows in the third extended multi-segment. In order for $r_1$ and $r_2$ to have a valid union-intersection operation, we must have that  $A(r_1) \geq B(r_2)-1$. So $r_1$ cannot lie in $\BB_i$ for $i<i_2-2$, and if $r_1$ lies in $\BB_{i_2-2}$, then the boundaries between $\BB_{i_2-2}$, $\BB_{i_2-1}$ and $\BB_{i_2-1}$, $\BB_{i_2}$ must be either both type 2, or one of them is type 2 and one of them is type 3.

    In the case that both boundaries are type 2, it must be that $\BB_{i_2-1}$ has exactly two columns, say $H_{col}$ and $H_{col}+1$, as shown below. 
    
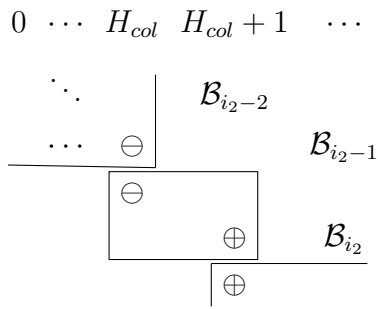
\[\begin{tikzpicture}[baseline=(current bounding box.north)]
      \matrix (m) [matrix of math nodes, nodes in empty cells] {
        0 & \cdots & H_{col} & H_{col} + 1 & \cdots\\
        & \ddots & \phantom{\oplus} & \BB_{i_2-2} & \\
        & \cdots & \ominus & & \BB_{i_2-1} \\
        & & \ominus & &  \\
        & & & \oplus & \BB_{i_2} \\
        & & & \oplus & \phantom{\oplus} \\
      } ;
      \draw (m-3-1.south west) -- (m-3-3.south east);
      \draw (m-3-3.south east) -- (m-2-3.north east);
      \draw (m-6-4.north west) -- (m-6-5.north east);
      \draw (m-6-4.north west) -- (m-6-4.south west);
      \draw (m-4-3.north west) rectangle (m-5-4.south east);
    \end{tikzpicture}\]    
    For support reasons, the only way for the union-intersection to occur is if $r_1$ is exchanged to the first position in $\BB_{i_2-1}$, and $r_2$ is exchanged to the position immediately after. If either $r_1$ or $r_2$ has only one circle, then the first row exchange for $r_1$ or $r_2$ is trivial, so the $ui$ is equivalent to a $ui$ between adjacent parts of the decomposition, which is covered in the cases below. But if $r_1$ has more than one circle, the row exchange with the first row of $\BB_{i_2-1}$ is Case 1(b) of Definition \ref{def row exchange}, since by Lemma \ref{lem-aj} the sign of the last circle of $r_1$ is the same as the sign of the last circle of $\BB_{i_2-2, temp}$, which is the same as the sign of the first circle of $\BB_{i_2-1, temp}$, which is the same as the sign of the first circle of $\BB_{i_2-1}$. So the result after the row exchange, denoted $r_1'$, has $l(r_1') = 1$, so no $ui$ is possible.
    
    In the case that one of the boundaries is type 2 and one of the boundaries is type 3, first suppose that the boundary between $\BB_{i_2-2}$ and $\BB_{i_2-1}$ is of type 3, and the boundary between $\BB_{i_2-1}$ and $\BB_{i_2}$ is of type 2. Let the last column of $\BB_{i_2-2}$ be $H_{col}$. In order for the $ui$ to be possible, it must be that $\BB_{i_2-1}$ is only supported in one column, namely $H_{col}+1$, as depicted below.
    \[\begin{tikzpicture}[baseline=(current bounding box.north)]
      \matrix (m) [matrix of math nodes, nodes in empty cells] {
        0 & \cdots & H_{col} & H_{col} + 1 & \cdots\\
        & \ddots & \phantom{\oplus} & \BB_{i_2-2} & \\
        & \cdots & \oplus & \BB_{i_2-1} & \\
        & & & \oplus & \BB_{i_2} \\
        & & & \oplus & \phantom{\oplus} \\
        & & & \phantom{\oplus} & \ominus \\
      } ;
      \draw (m-3-1.south west) -- (m-3-3.south east);
      \draw (m-3-3.south east) -- (m-2-3.north east);
      \draw (m-5-4.north west) -- (m-5-5.north east);
      \draw (m-5-4.north west) -- (m-6-4.south west);
      \draw (m-4-4.north west) rectangle (m-4-4.south east);
    \end{tikzpicture}\]
    
    The union-intersection must be performed by exchanging $r_1$ to the last row of $\BB_{i_2-2}$ and exchanging $r_2$ to the first row of $\BB_{i_2-1}$. By Lemma \ref{lem-multiplicity-cancel}, without loss of generality we can assume $\BB_{i_2-1}$ has only one circle. Then the row exchange for $r_2$ is in Case 2(b) of Definition \ref{def row exchange} by Lemma \ref{lem-aj}, so the result $r_2'$ has $l(r_2') = 1$ and hence no $ui$ is possible.
    
    Now suppose that the boundary between $\BB_{i_2-2}$ and $\BB_{i_2-1}$ is of type 2, and the boundary between $\BB_{i_2-1}$ and $\BB_{i_2}$ is of type 3. In this case, in order for the support of $r_1$ to include $H_{col}-1$, it must be that $\BB_{i_2-1}$ lies only in the column $H_{col}-1$. Since it follows a boundary of type 2, it must be a single circle. In summary, we have the following situation:
    \[\begin{tikzpicture}[baseline=(current bounding box.north)]
      \matrix (m) [matrix of math nodes, nodes in empty cells] {
        0 & \cdots & H_{col} & H_{col} + 1 & \cdots\\
        & \ddots & \phantom{\oplus} & & \\
        & \cdots & \oplus & \BB_{i_2-2} & \\
        & & \oplus & \BB_{i_2-1} & \BB_{i_2} \\
        & & & \oplus & \phantom{\oplus} \\
        & & & \phantom{\oplus} & \ominus \\
      } ;
      \draw (m-3-1.south west) -- (m-3-3.south east);
      \draw (m-3-3.south east) -- (m-2-3.north east);
      \draw (m-5-4.north west) -- (m-5-5.north east);
      \draw (m-5-4.north west) -- (m-6-4.south west);
      \draw (m-4-3.north west) rectangle (m-4-3.south east);
    \end{tikzpicture}\]
    In this case, by Lemma \ref{lem-aj}, when row $r_1$ is exchanged so that it takes the place of the last row in $\BB_{i_2-2}$, it must be a row of only circles ending in the same sign as $\BB_{i_2-2}$. After another row exchange with the only row of $\BB_{i_2-1}$, it contains a pair of triangles. So no union-intersection is possible with row $r_2$, which after row exchanges, by Lemma \ref{lem-aj}, is a row of circles starting at $H_{col}+1$.
    
    \textbf{Case 2.} Second suppose that $r_1$ lies in $\BB_{i_2-1}$. If the boundary between $\BB_{i_2-1}$ and $\BB_{i_2}$ is type 1 or type 3, since no row exchanges can occur across the boundary, in order for a union-intersection to occur between a row of $\BB_{i_2-1}$ and a row of $\BB_{i_2}$, it must be that a row of $\BB_{i_2-1}$ is exchanged to the last row of $\BB_{i_2-1}$, and a row of $\BB_{i_2}$ is exchanged to the first row of $\BB_{i_2}$. Let $r_1'$ be the image of the last row of $\BB_{i_2-1}$ after the row exchanges, and let $r_2'$ be the image of the first row of $\BB_{i_2}$ after row exchanges. By Lemma \ref{lem-aj}, $l(r_1') = 0$ and $l(r_2') = 0$. However for a type 1 boundary, no $ui$ is possible since $B(r_2') > A(r_1') + 1$. In a type 3 boundary, for support reasons the $ui$ would have to be of type 3', since $A(r_1') = H_{col}$ and $B(r_2') = H_{col}+1$. By Lemma \ref{lem-aj}, since the sign of the last circle of $\BB_{i_2-1, temp}$ is the same as the sign of the first circle of $\BB_{i_2, temp}$, the sign of the last circle of $r_1'$ is the same as the sign of the first circle of $r_2'$. Hence no $ui$ of type 3' is possible.

    If the boundary between $\BB_{i_2-1}$ and $\BB_{i_2}$ is type 2, then there are two ways a union-intersection might occur. First, $r_1$ is exchanged across the boundary, which can only happen if it is exchanged with the first circle of $\BB_{i_2}$. In this case, since by Lemma \ref{lem-aj} the row contains only circles when it is the last row of $\BB_{i_2-1}$, it contains a pair of triangles when it is the first row of $\BB_{i_2}$. However $r_2$ after row exchanges, by Lemma \ref{lem-aj}, is a row of circles starting at $H_{col}+1$, so no union-intersection is possible. Second, $r_1$ is not exchanged across the boundary. After row exchanges $r_1$ and $r_2$ are both only circles, say $r_1'$ and $r_2'$, so the only possible $ui$ is of type 3'. However the sign of the last circle of $\BB_{i_2-1}$ is the same as the sign of the first circle of $\BB_{i_2}$, so by Lemma \ref{lem-aj}, the sign of the last circle of $r_1'$ is the same as the sign of the first circle of $r_2'$. This means that union-intersection is not possible.
\end{proof}

We now prove the main result of this subsection, namely, that the blocks occurring in the block decomposition control the equivalence of tempered virtual extended multi-segments.

\begin{prop}[Independence of blocks]
\label{prop-independence-of-blocks} 
    Let $\EE\in\VRep_\rho^\mathbb{Z}(G_n)(G)$ be  equivalent to a tempered virtual extended multi-segment $\EE_{temp}$. If $\EE_{temp}$ has a block decomposition $\BB_{1, temp} \cup \cdots \cup \BB_{k, temp}$, then there exists a decomposition $\EE = \BB_1 \cup \cdots \cup \BB_k$ such that $\BB_i$ is equivalent to $\BB_{i, temp}$ for all $i$.
\end{prop}
\begin{proof}
    Since $\EE$ is equivalent to $\EE_{temp}$, there exists a sequence of raising operators and their inverses taking $\EE_{temp}$ to $\EE$. Suppose the first operation sends $\EE_{temp}$ to $\EE^{(1)}$, then the second operation sends $\EE^{(1)}$ to $\EE^{(2)}$, and so on. Note that a decomposition of the desired form exists on $\EE_{temp}$, namely the decomposition into blocks. So it suffices to show that if $\EE^{(s)}$ has a decomposition of the desired form, after applying a raising operator or its inverse, we can still find a suitable decomposition. Then by repeating this procedure, we find a suitable decomposition of $\EE$.

    So suppose we have some raising operator or inverse $T$ on $\EE^{(s)}$, assuming that $\EE^{(s)}$ has some decomposition $\BB_1 \cup \cdots \cup \BB_k$ with $\BB_i$ equivalent to $\BB_{i, temp}$. By Lemma \ref{lem-independence-of-blocks}, $T$ cannot involve (either by union-intersection or by row exchanges) a row from $\BB_{i_1}$ and a row from $\BB_{i_2}$ for $i_1 \neq i_2$. So $T$ involves only rows from $\BB_{i_0}$ for some $i_0$. By Lemma \ref{lem-local}, this means $T$ does not affect rows from $\BB_i$ for $i \neq i_0$. Then $T(\EE^{(s)}) = \BB_i \cup \cdots \cup T'(\BB_{i_0}) \cup \cdots \cup \BB_k$, where $T'$ denotes the corresponding operation. Therefore, we have found a suitable decomposition of $T(\EE^{(s)})$.
\end{proof}

We end this subsection with a proof of Theorem \ref{thm-count-temp} which we recall below (with slightly different notation).

\begin{thm}[count for tempered representations]\label{thm-count_for_tempered-in-paper}
    Consider a tempered virtual extended multi-segment $\EE_{temp} \in\VRep_\rho^\mathbb{Z}(G_n)(G)$  and suppose that $\EE_{temp}$ decomposes into blocks $\BB_{1, temp} \cup \dots \cup \BB_{k, temp}$ in that order. Then \[|\Psi(\pi(\EE_{temp}))| = |\Psi(\pi(\BB_{1,temp}))| \cdot \prod_{i=2}^k |\Psi(\pi(sh^1(\BB_{i,temp})))|.\]
\end{thm}

\begin{proof}
    Let $\EE\in\VRep_\rho^\mathbb{Z}(G_n)(G)$ be a virtual extended multi-segment equivalent to $\EE_{temp}$.
    By Proposition \ref{prop-independence-of-blocks}, $\EE$ has a decomposition $\BB_1 \cup \dots \cup \BB_k$ with each $\BB_i$ equivalent to $\BB_{i, temp}$. We claim that each $\BB_i$ for $i>1$ must be of type $Y_\mathcal{M}^{>0}$. Conversely, we claim that for every $\BB_1$ equivalent to $\BB_{1, temp}$ and for every $\BB_i$ of type $Y_\mathcal{M}^{>0}$ equivalent to $\BB_{i, temp}$, there exists $\EE$ equivalent to $\EE_{temp}$ with the decomposition $\BB_1 \cup \dots \cup \BB_k$.

    For the first claim, since $\EE_{temp}$ satisfies these conditions, it suffices to show that after applying any raising operator or inverse, the extended multi-segment still satisfies these conditions. Let $\EE$ be some extended multi-segment equivalent to $\EE_{temp}$ such that each $\BB_i$ for $i>1$ is of type $Y_\mathcal{M}^{>0}$. Let $T$ be an operation on $\EE.$ By Lemma \ref{lem-independence-of-blocks}, we can view $T$ as an operation on one of the parts of the decomposition $\BB_i$. If $T$ is an operation on $\BB_1$, then there is nothing to prove. Otherwise, suppose $T$ is an operation on $\BB_i$. Then applying Lemma \ref{lem-later-blocks-exhaustion} to the truncated extended multi-segment $\BB_1 \cup \dots \cup \BB_i$, we see that after $T$ is applied, $\BB_i'$ is still of type $Y_\mathcal{M}^{>0}$. Note that since $T$ was a valid operation on $\EE$, it is still a valid operation on the truncated extended multi-segment.

    For the second claim, by Lemma \ref{lem-first-block} we can apply the operations taking $\BB_{1, temp}$ to $\BB_1$ to the first block of $\EE_{temp}$. Since these operations are local (see Lemma \ref{lem-local}), they do not affect the other rows and so we obtain an equivalent extended multi-segment with decomposition $\BB_1 \cup \BB_{2, temp} \cup \dots \cup \BB_{k, temp}$. By Lemma \ref{lem-later-blocks-existence}, we can repeat the same logic for each of the other blocks by considering successive truncations and applying the relevant operations to the last part of each truncation.
\end{proof}

\bibliographystyle{amsplain}
\bibliography{On_Arthur_packets_containing_a_fixed_tempered_arXiv}

\end{document}